\documentclass[leqno, final]{siamart251216}

\usepackage[utf8]{inputenc}
\DeclareUnicodeCharacter{2212}{-}

\usepackage{graphicx}
\usepackage{subfigure}
\usepackage{epstopdf}
\usepackage{amsmath, amssymb}
\usepackage{color, xcolor}
\usepackage{algorithm}
\usepackage{multicol}
\usepackage{appendix}
\usepackage[normalem]{ulem}
\usepackage{cancel} 
\usepackage{stmaryrd}
\usepackage{epstopdf} 
\usepackage{microtype} 
\usepackage{mathtools}

\allowdisplaybreaks[2]

\numberwithin{equation}{section}

\newtheorem{remark}{Remark}[section]

\newcommand{\norm}[1]{\left\Vert#1\right\Vert}
\newcommand{\norml}[2]{\left\Vert#1\right\Vert_{L^2(#2)}}
\newcommand{\normE}[1]{\interleave #1 \interleave}
\newcommand{\normef}[2]{\interleave#1\interleave'_{#2}}
\newcommand{\norme}[2]{\left\Vert{\hskip -2.7pt}\left\vert #1 \right\vert{\hskip -2.7pt}\right\Vert_{#2}} 

\newcommand{\abs}[1]{\left\vert#1\right\vert}
\newcommand{\pd}[1]{\left\langle #1\right\rangle}

\newcommand{\set}[1]{\left\{#1\right\}}

\newcommand{\til}[1]{\tilde{#1}}

\newcommand{\R}{\mathbb{R}}
\newcommand{\cs}{C_{\rm sta}}

\newcommand{\al}{\alpha}
\newcommand{\be}{\beta}
\newcommand{\de}{\delta}
\newcommand{\De}{\Delta}

\newcommand{\ga}{\gamma}

\newcommand{\la}{\lambda}

\newcommand{\na}{\nabla}

\newcommand{\Om}{\Omega}
\newcommand{\pa}{\partial}

\newcommand{\si}{\sigma}

\newcommand{\ze}{\zeta}

\renewcommand{\i}{{\rm\mathbf i}}

\newcommand{\roma}{\uppercase\expandafter{\romannumeral1}}
\newcommand{\romb}{\uppercase\expandafter{\romannumeral2}}
\newcommand{\romc}{\uppercase\expandafter{\romannumeral3}}

\newcommand{\ii}{\boldsymbol{ \mathrm{i} }}

\newcommand{\p}{{\rm p}}
\newcommand{\pml}{\mathrm{PML}}
\newcommand{\tru}{\mathrm{tru}}

\newcommand{\Oi}{\Om_{i}^{\mathrm{PML}}}
\newcommand{\Oij}{\Om_{i, j}^{\mathrm{PML}}}
\newcommand{\Oijs}{\Om_{i, j, s}^{\mathrm{PML}}}

\newcommand{\uij}{u_{ij}^{\pm\pm}}
\newcommand{\uijj}{\bar{u}_{ij}^{\pm\pm}}

\newcommand{\tu}{\breve{u}}
\newcommand{\ttu}{\breve{u}}
\newcommand{\uu}{\breve{u}}
\newcommand{\tv}{\breve{v}}
\newcommand{\vv}{\breve{v}}
\newcommand{\tf}{\breve{f}}
\newcommand{\ff}{\breve{f}}
\newcommand{\tF}{\breve{F}}
\newcommand{\ttF}{\breve{F}}
\newcommand{\ttf}{\breve{f}}
\newcommand{\tPsi}{\breve{\Psi}}
\newcommand{\ttPsi}{\breve{\Psi}}
\newcommand{\PPsi}{\breve{\Psi}}

\newcommand{\flay}{\normE{f}'_{\D_\xi}}
\newcommand{\fbar}{\normE{f}'_{\D_{\xi\eta}}}
\newcommand{\fblk}{\normE{f}'_{\D_{\xi\eta\tau}}}

\DeclareMathOperator{\re}{{Re}}
\DeclareMathOperator{\im}{{Im}}

\newcommand{\D}{\mathcal{D}}

\newcommand{\eq}[1]{\begin{align}#1\end{align}}
\newcommand{\eqn}[1]{\begin{align*}#1\end{align*}}
\newcommand{\ls}{\lesssim}

\newcommand{\supp}{\mathrm{supp}}
\newcommand{\diag}{\mathrm{diag}}

\usepackage{hyperref}

\title{A pure source transfer domain decomposition method for Helmholtz equations in 3D unbounded domain}
\author{
	Xinran Bai
	\thanks{Department of Mathematics,  Nanjing University,  Jiangsu,  210093,  P.R. China. ({\tt xrbai @smail.nju.edu.cn, hjw@nju.edu.cn}). This work is partially supported by the National Key R\&D Program of China grant 2024YFA1012600 and the National Natural Science Foundation of China grants 12171238, 12261160361 and 12671480. This work was supported in part by the High Performance Computing Platform of Xiangtan University.}
	\and Yu Du
	\thanks{School of Mathematics and Computational science,  Xiangtan University,  Hunan,  411105,  China
  ({\tt duyu@xtu.edu.cn}).}
	\and
	Haijun Wu\footnotemark[1]
}

\begin{document}
	
	\maketitle
	\vspace{-1.4in}
	\slugger{sinum}{200x}{xx}{x}{xxx--xxx}
	\vspace{1.4in}
	
	\setcounter{page}{1}
	
	\begin{abstract}
		We develop and analyze a three-dimensional pure source transfer domain decomposition method (PSTDDM) for the Helmholtz equation in free space. Extending the two-dimensional PSTDDM proposed by [Du and Wu, J. Sci. Comput., 83(2020),83:68], the proposed method transfers sources successively along the $x_1$, $x_2$, and $x_3$ directions, yielding layer-wise, bar-wise, and block-wise PSTDDMs. The main challenge in the convergence analysis is deriving wavenumber-explicit truncation-error estimates. To address this challenge, we establish inf--sup conditions for the uniaxial perfectly matched layer (UPML) problem under relaxed restrictions on the PML parameters and reformulate the source-transfer operators. We then introduce a positive-definite reconstruction step that produces a globally conforming approximation. In addition, a new  quantitative lower bound is established for the imaginary part of the complex PML distance, more easily yielding exponential decay estimates in three dimensions. These lead to wavenumber-explicit convergence estimates for all three PSTDDMs. Numerical experiments demonstrate the effectiveness of the proposed methods both as direct solvers and as preconditioners for GMRES.
	\end{abstract}

	%
	%
	
	\setcounter{page}{1}

	\section{Introduction}\label{sec_int} 	
	
	We consider the Helmholtz equation in the full space $\R^3$, subject to the Sommerfeld radiation condition:
	\begin{align}
		\De u + k^2 u = f \qquad & \text{in } \R^3, \label{eq1.1a} \\
		r\Big|\frac{\pa u}{\pa r}-\ii ku\Big| \rightarrow 0 \qquad
		& \text{as } r=\abs{x}\rightarrow\infty, \label{eq1.1b}
	\end{align}
	where $k>0$ is the wavenumber and $f\in L^{2}(\R^3)$ has compact support. The Helmholtz problem is of considerable importance in various scientific computation and engineering applications, such as acoustic, elastic and electromagnetic scattering problems. Owing to the highly indefinite nature of the Helmholtz problem with large wavenumber and the huge number of degrees of freedom resulted from the pollution error, it is challenging to solve the discrete Helmholtz equation \cite{BABUSKA1995325, d3, D23, D35, D36, D54}.
	
	The perfectly matched layer (PML), originally introduced by B\'{e}renger \cite{BERENGER1994185}, is an effective and flexible technique in the numerical simulation of wave scattering problems posed in unbounded domains. It surrounds the physical domain by an artificial absorbing layer designed to attenuate outgoing waves without reflection at the interface. The resulting bounded PML problem can then be closed, for example, by imposing a homogeneous Dirichlet condition on the outer boundary. Many PML constructions and analyses have since been developed; see, among others, \cite{D38, D17, D18, D20, d15, chen2008adaptive, chen2013adaptive, 2010Analysis, stddm, llw}. In this paper, we use a uniaxial PML (UPML).

	A variety of fast solvers have been proposed for large wavenumber Helmholtz problems, including multigrid methods \cite{2015Multilevel, Multikry2001,FuGongLiWang24arxiv}, butterfly algorithms \cite{D12}, 
	 Schwarz methods \cite{GrahamSpenceZou20, 1998D21, 2008D29, 2006OptiSchNoOverlap, 2008D51, 2016OptiSchOverlap, 2019D42}, sweeping preconditioners \cite{2010D25, 2014D53, 2013D52, 2016D41, 2016D26}, source transfer domain decomposition methods (STDDMs) \cite{stddm, stddm2, 2019D39, 2019DSTDDM, stddmwang}, and two-level Schwarz methods \cite{LuXuZhengZou25, MaChupeng25}. Inspired by the PML sweeping preconditioner in \cite{2010D25SweepingPML}, Chen and Xiang proposed the STDDM and established its convergence \cite{stddm}. Du and Wu subsequently developed the pure source transfer domain decomposition method (PSTDDM) for the two-dimensional Helmholtz problem \cite{pstddm}. In contrast to the original STDDM, the PSTDDM uses two independent source-transfer sweeps in opposite directions; moreover, at each step, the local two-layer PML problems can be decomposed recursively in a transverse direction. This further reduces the sizes of the local subproblems.

	In this paper, we extend the PSTDDM to the three-dimensional Helmholtz equation in the full space; this extension is not merely a formal repetition of the same construction with one additional coordinate.
	For the PSTDDM applied to the truncated PML problem, we introduce an elliptic projection to reconstruct the output of the algorithm, thereby ensuring that the reconstructed approximation belongs to the appropriate solution space. This reconstruction, however, means that the original source-transfer operators no longer admit the same simple equivalent representation. We therefore reformulate the source-transfer operators used in the truncated algorithm.
	In contrast to \cite{stddmwang}, we establish a new, easy-to-use formula for the lower-bounding the imaginary part of the complex PML distance (see Lemma~\ref{lem-imrho}) and establish inf--sup conditions under more flexible assumptions on the PML parameters, namely, $0\leq\zeta\lesssim1$ and $\sigma_0\geq1$. We also combine three-dimensional complex-distance estimates with a decomposition of the domain into regions and local PML stability estimates to control the truncation errors and establish convergence estimates that are explicit in the wavenumber and independent of the number of degrees of freedom.
	We point out that deriving an wavenumber-explicit convergence estimate presents another challenge: it requires careful handling of the numerous and complex derivations and error estimates resulting from the many error terms arising from nested source transfer and wave propagation.

Next we illustrate the ideas of PSTDDM by considering the Helmholtz problem \eqref{eq1.1a}--\eqref{eq1.1b}.
	Let $B_l=\{x=(x_1, x_2, x_3)\in\R^3:\abs{x_1}<l_1, \abs{x_2}<l_2, \abs{x_3}<l_3\}$. Assume that $f$ is supported in $B_l$. We divide the interval $(-l_1, l_1)$ into $N\ge 2$ segments with the points $\xi_i=-l_1+(i-1)\De\xi$ where $\De\xi=2l_1/N$ and $1\le i\le N+1$. Set $\xi_0=-\infty,  \xi_{N+2}=\infty$. We then decompose $\mathbb{R}^3$ as layers defined by
	\begin{align}
		\Omega_1 = & \set{ x=(x_1, x_2, x_3)\in\R^3: x_1<\xi_{2} }, \notag     \\
		\Omega_i = & \{ x=(x_1, x_2, x_3)\in\R^3:\xi_i<x_1<\xi_{i+1} \}, \ i=2, \cdots, N-1, \label{eq_Nlayers} \\
		\Omega_{N} = & \set{ x=(x_1, x_2, x_3)\in\R^3: x_1>\xi_{N} }.\notag
	\end{align}
We denote this decomposition as $\D_\xi:=\set{\Omega_i, i=1,2,\cdots,N}$. Define $f_i=f$ in $\Omega_i$ and $f_i=0$ elsewhere, and let $\bar{f}_1^+=f_1$ and $\bar{f}_N^-=f_N$. As in \cite{pstddm}, the key idea is to define the source transfer operators $\Psi_i^\pm$ through the identities
	\begin{align*}
		& \int_{\Omega_i} \bar{f}_i^+(y)G(x, y)dy = \int_{\Omega_{i+1}} \Psi_{i+1}^+(\bar{f}_i^+)(y)G(x, y)dy\quad \forall x\in \mathop{\cup}_{j=i+2}^{N+1}\Omega_j,  \\
		& \int_{\Omega_i} \bar{f}_i^-(y)G(x, y)dy = \int_{\Omega_{i-1}} \Psi_{i-1}^-(\bar{f}_i^-)(y)G(x, y)dy\quad \forall x\in\mathop{\cup}_{j=0}^{i-2}\Omega_j, 
	\end{align*}
	and then set $\bar{f}_{i\pm1}^\pm=f_{i\pm1}+\Psi_{i\pm1}^\pm(\bar{f}_i)$. It follows that, for any $x\in\Omega_i$
	\begin{align}\label{eqi1}
		u(x) & = \bigg(-\int_{\Omega_i}f_i(y)G(x, y)dy-\int_{\Omega_{i-1}}\bar{f}_{i-1}^+(y)G(x, y)dy\bigg)  \\
		& \quad +\bigg(-\int_{\Omega_{i+1}}\bar{f}_{i+1}^-(y)G(x, y)dy\bigg).\notag
	\end{align} 
	Here,  $G(x, y)$ is the Green's function for the problem \eqref{eq1.1a}--\eqref{eq1.1b}. From \eqref{eqi1}, we see that $u(x)$ in $\Omega_i$ consists of two independent parts. The first part involves only the sources in $\Omega_i$ and $\Omega_{i-1}$, whereas the second involves only the source in $\Omega_{i+1}$. They can therefore be computed independently by applying the perfectly matched layer (PML) method immediately outside $\Omega_{i-1}\cup\Omega_i$ and $\Omega_i\cup\Omega_{i+1}$, respectively. This procedure leads to the PSTDDM in Algorithm~\ref{alg1} in Subsection~\ref{subs1}. The PSTDDM with a truncated PML is given in Algorithm~\ref{alg2}, and its convergence is proved in Subsection~\ref{subs2}.
	Moreover, in the above PSTDDM procedure, the PML problems on two adjacent layers can themselves be solved by the PSTDDM in coordinate directions parallel to the layers. We thus construct the bar-wise and block-wise PSTDDMs (see Algorithms~\ref{alg3}--\ref{alg5} and Algorithms~\ref{alg41}--\ref{alg44}), which further reduce the sizes of the subdomain problems and are more efficient for large-scale three-dimensional problems. For clarity, we refer to the PSTDDMs in Algorithms~\ref{alg1} and \ref{alg2} as the layer-wise PSTDDMs.

	The remainder of the paper is organized as follows.
	Section~\ref{sec2} introduces the three-dimensional PML formulation and the layer-wise PSTDDM, together with the corresponding PML stability and truncation analysis. Section~\ref{sec3} develops the bar-wise PSTDDM by performing the second family of source transfers in the $x_2$ direction and derives the associated error estimates. Section~\ref{sec4} further decomposes the bar problems through source transfers in the $x_3$ direction, introduces the block-wise PSTDDM, and completes the three-level convergence analysis. Section~\ref{sec5} presents the numerical experiments to illustrate the performances of the PSTDDMs.

	\section*{Notation and conventions}
	For a bounded Lipschitz domain $U\subset\R^3$, we use the standard Sobolev spaces $H^s(U)$ with their usual norms and seminorms; see \cite{brenner2008mathematical}. For $v\in H^1(U)$, we use the following wavenumber-weighted norms:
	\begin{align} \label{eq:vhga12}
		\normE{v}_U^2 &:= \norm{\na v}_{L^2(U)}^2+\norm{k v}_{L^2(U)}^2,\quad
		\normE{v}_{\frac12,\pa U}^2&:=k\norm{v}_{L^2(\pa U)}^2+\abs{v}_{{H^{\frac12}(\pa U)}}^2,
	\end{align}
	where
	\eqn{\abs{v}_{{H^{\frac12}(\pa U)}}^2
	=\int_{\pa U}\int_{\pa U}
	\frac{\abs{v(x)-v(y)}^2}{\abs{x-y}^3}\,dS(x)\,dS(y).}
	A scaling argument gives
	\begin{align}
		\normE{v}_{\frac12,\pa U}
		\ls k^{\frac12}\norm{v}_{L^{\infty}(\pa U)}
		+\norm{\na_\Gamma v}_{L^{\infty}(\pa U)}
		\qquad\forall v\in W^{1,\infty}(\pa U), \label{leq_vga}
	\end{align}
	where the hidden constant may depend on the geometry of $U$ and $\na_\Gamma$ denotes the surface gradient.

	For $g\in H^{-1}(U):=(H_0^1(U))'$ and $g\in(H^1(U))'$, respectively, define
	\begin{align*}
		\normE{g}_U^*&:=\sup_{0\neq v\in H_0^1(U)}
		\frac{\abs{\pd{g,v}}}{\normE{v}_U},
		&
		\normE{g}'_U&:=\sup_{0\neq v\in H^1(U)}
		\frac{\abs{\pd{g,v}}}{\normE{v}_U}.
	\end{align*}
	If $g\in L^2(U)$, then
	$\normE{g}_U^*\leq\normE{g}'_U\leq k^{-1}\norm{g}_{L^2(U)}$. Although $\normE{f}^*_{\Omega_i}\le \normE{f}^*_{B_l}$, it is a bit counterintuitive that $\normE{f}'_{\Omega_i}$ may not be bounded by $\normE{f}'_{B_l}$. We introduce the following negative norm on the decomposition $\D_\xi$: 
\eq{\label{normDD1}
\normE{f}'_{\D_\xi}:=\Big(\sum_{i=1}^N\big(\normE{f}'_{\Omega_i}\big)^2\Big)^\frac12
.}
Clearly, we have  $ \normE{f}'_{B_l}\le\normE{f}'_{\D_\xi}$ and $\normE{f}'_{\Omega_i\cup\Omega_j}\le\normE{f}'_{\D_\xi}, \forall 1\le i\neq j\le N$.
	
	For simplicity of notation, $C$ denotes a generic positive constant that is independent of $k$ and $f$ but may depend polynomially on the PML absorption constant $\si_0$; see \eqref{eq:PMLsigma}. We write $A\lesssim B$ and $B\gtrsim A$ when $A\leq CB$, and $A\eqsim B$ when both $A\lesssim B$ and $B\lesssim A$ hold. Throughout, we assume that $k\gtrsim1$.

	\section{Source transfer layer by layer} \label{sec2}
	In this section,  we first recall the PML formulations for the Helmholtz scattering problem \eqref{eq1.1a}--\eqref{eq1.1b}. Then we introduce the PSTDDM for the PML problem in the whole space as well as in the bounded truncated domain by dividing the computational domain into layers and doing source transfer layer by layer.

	\subsection{The PML problems}\label{ssecpml}
	In this subsection,  we introduce the PML method to truncate the unbounded domain.
	
	It is known that the solution to the Helmholtz problem \eqref{eq1.1a}--\eqref{eq1.1b} has the following integral representation \cite{r2r3}
	\begin{align}\label{represent}
		& u(x) = -\int_{\R^3}f(y)G(x, y)dy = -\sum\limits^{N}_{i=1}\int_{\Omega_{i}} f_i(y)G(x, y)dy ,  \; G(x, y) = \frac{e^{ik\abs{x-y}}}{4\pi\abs{x-y}},  
	\end{align}
	where $G(x, y)$ is the fundamental solution of the Helmholtz problem.
	The integral in \eqref{represent} is well-defined if $f$ is a smooth function with compact support in $\mathbb{R}^3$. In the general case when $f \in H^1\left(\mathbb{R}^3\right)^{\prime}$ with compact support,  it should be understood in the sense of distribution.
	
	We adopt the uniaxial PML method \cite{C18, r2r3, d15, d16, stddm, pstddm}, 
	it is based on the complex coordinate stretching outside the domain $B_l$.
	The model medium properties are defined by
	\begin{align}\label{eq:PMLsigma}
	\begin{cases}
		& \al_j(x_j)=1+(\zeta+\ii)\si_j(x_j),\\
		& \si_j=0\ \mathrm{for}\ \abs{t}\leq l_j, \ \si_j=\si_0>0\ \mathrm{for}\ \abs{t}\geq \bar{l}_j,    \\
		& \si_j(t)=\si_j(-t)\ \mathrm{for}\ t\in\R^3 \text{ and }  0\le \si_j'(t)\ls 1\  \mathrm{for}\ t\geq 0.
	\end{cases}
	\end{align}
	where the constant $0\le \zeta \ls 1$,  the functions $\si_j(x_j)\in C^1(\R)$ are piecewise smooth,  $\bar{l}_j>l_j$ is fixed,  and $\si_0$ is a constant. The key requirement here is $\sigma_j(t)$ is constant for $\abs{t} \geq \bar l_j$ which is essential for using the reflection argument of \cite{r2r3}  to prove the inf-sup condition for the PML problem in the truncated domain and to estimate the dependence of the inf-sup constants on the wave number $k$ when $k$ is large (see Lemma \ref{blstab} below).
	
	We introduce the PML by complex coordinate stretching 
	as follows.
	For $x=(x_1, x_2, x_3)^T$,  define $\tilde{x}(x)=(\tilde{x}_1(x_1), \tilde{x}_2(x_2), \tilde{x}_3(x_3))$ by
	\begin{align}
		\tilde{x}_j(x_j) = \int_0^{x_j} \al_j(t) dt = x_j + ( \zeta + \ii )\int_0^{x_j}\si_j(t)dt, \ j=1, 2, 3,  \label{eq_defPML}
	\end{align}
Then the PML problem is defined as $\tilde{\Delta} \tilde{u}+k^2 \tilde{u}=f$ in $\R^3$,  where $\tilde{\Delta}$ is the Laplacian with respect to the stretched coordinate $\tilde{x}$. This yields by the chain rule that $\tilde{u}$ satisfies the PML equation
	\begin{align}\label{pmleq}
		J^{-1}\na\cdot(A\na\tilde{u})+k^2\tilde{u}=f\quad \mathrm{in}\ \R^3, 
	\end{align}
	where $J(x)=\al_1(x_1)\al_2(x_2)\al_3(x_3)$ and
	$$
	A(x)=\diag\left(\frac{\al_2(x_2)\al_3(x_3)}{\al_1(x_1)}, \frac{\al_1(x_1)\al_3(x_3)}{\al_2(x_2)}, \frac{\al_1(x_1)\al_2(x_2)}{\al_3(x_3)}\right).
	$$ 
The weak formulation of \eqref{pmleq} reads as: Find $\tilde u\in H^1(\R^3)$ such that
	\begin{align}
		(A\na \tilde{u}, \na v) - k^2 (J\tilde{u}, v) = - \pd{Jf, v}\ \forall v\in H^1(\R^3),  \label{eq_exactsol}
	\end{align}
	where $(\cdot, \cdot)$ is the inner product in $L^2(\R^3)$ and $\pd{\cdot, \cdot}$ is the duality pairing between $H^1(\R^3)'$ and $H^1(\R^3)$. 
	
		For any $z\in\mathbb{C}\setminus[0, +\infty)$,  denote by $z^\frac12$ the analytic branch of $\sqrt{z}$ such that $\mathrm{Re}(z^\frac12)>0$. Define the "complex distance"
	\begin{align}\label{rho}
		\rho(\tilde{x}, \tilde{y})=\big((\tilde x_1-\tilde y_1)^2+(\tilde x_2-\tilde y_2)^2+(\tilde x_3-\tilde y_3)^2\big)^\frac12
	\end{align}
	The fundamental solution of the PML equation \eqref{pmleq} is formulated as \cite{stddm, D38}:
	\begin{align}\label{fs2}
		\tilde{G}(x, y)=J(y)G(\tilde{x}, \tilde{y}),  \quad G(\tilde{x}, \tilde{y})=  \frac{e^{\ii k\rho(\tilde{x}, \tilde{y})}}{4\pi\rho(\tilde{x}, \tilde{y})}.
	\end{align}
	$\tilde{G}$ satisfies 
	\begin{align}\label{Gxy}
		J^{-1}\na\cdot(A\na\tilde{G}(x, y))+k^2\tilde{G}(x, y)=-\de_y(x)\quad \mathrm{in}\ \R^3.
	\end{align}
	Notice that $\tilde{G}(x, y) \neq \tilde{G}(y, x)$, while $\tilde{G}(y,  x)$ is the fundamental solution of the adjoint equation \cite{D38}:
	\begin{align}\label{Gyx}
		\na\cdot(A\na(J^{-1}\tilde{G}(y, x)))+k^2\tilde{G}(y, x)=-\de_y(x)\quad \mathrm{in}\ \R^3.
	\end{align}
	Here we recall that $\nabla \cdot (A \nabla J^{-1})$ is the adjoint operator of $J^{-1} \nabla \cdot (A \nabla)$.
	Similar to \eqref{represent}, 	the solution to the PML problem can be expressed as
	\begin{align}\label{tu}
		\tilde{u}(x)= - \int_{\R^3} f(y)\tilde G(x, y) dy=- \int_{\R^3} f(y)J(y)G(\tilde{x}, \tilde{y}) dy\quad \forall x\in\R^3.
	\end{align}
Since $f$ is supported inside $B_l$,  we have $\tilde{u}=u$ in $B_l$.

We recall the following estimates of the complex distance.
	\begin{lemma}[Lemma~3.1, \cite{r2r3}]
		Let $x$, $y$ be in $R^3$, there hold
		\begin{align}
			\label{rho_d}
			\tilde{C_1}^{-1} \abs{x-y} \leq \abs{ \rho(\tilde{x},  \tilde{y}) } \leq (2+2\abs{z_0}^2\si_0^2)^{\frac12}\abs{x-y}, 
		\end{align}		
		where $z_0=\ze +\ii$,  $\tilde{C_1}=1$ for $\zeta \geq 1$,  and $\tilde{C_1}=\abs{1+\frac{\i}{2}\si_0}^{\frac12}$ for $\zeta =0$. 
	\end{lemma}

Moreover, we have the following estimates for the imaginary part of the complex distance from \cite[(2.16) and Lemma 2.6]{chen2016elastic}.
	\begin{lemma}\label{Im_roub} We have 
\begin{align}		\im\rho(\tilde{x}, \tilde{y}) \geq \frac{\si_0^{-1}}{\abs{x-y}}\sum_{j=1}^3\Big| \int_{y_j}^{x_j}\si_j(t)dt \Big|^2 . \label{eq_imr}
	\end{align}
Let $\beta >1$ be a constant,  if $\abs{x-y}\geq 2 \sqrt{3} \beta \overline{l}_{\max}$, where $\overline{l}_{\max}=\mathop{\max}_{j=1, 2, 3}\bar l_j$,  then we have $\im\rho(\tilde{x}, \tilde{y}) \geq(1-\beta^{-1})^2\si_0 \abs{x-y}/3$.
	\end{lemma}	
\begin{proof} We only prove \eqref{eq_imr}  since the rest part is from \cite[Lemma 2.6]{chen2016elastic}.
By \cite[(2.16)]{chen2016elastic}  and the fact that $\abs{x_j-y_j}\ge\si_0^{-1}\int_{y_j}^{x_j}\si_j(t)dt$ we have
\eqn{
\im\rho(\tilde{x}, \tilde{y})&\ge \frac{\mathop{\sum}_{j=1}^{3} \big( \abs{x_j-y_j} \big| \int_{y_j}^{x_j}\si_j(t) dt \big| + \ze \big| \int_{y_j}^{x_j}\si_j(t)dt \big|^2 \big)}{(1+\zeta \si_0)\abs{x-y}}\\
& \geq \frac{\mathop{\sum}_{j=1}^{3} (\si_0^{-1}+ \ze) \big| \int_{y_j}^{x_j}\si_j(t)dt \big|^2 }{(1+\zeta \si_0)\abs{x-y}},}
which implies that \eqref{eq_imr} holds.
 This completes the proof.
\end{proof}
			
We remark that \cite[Lemma 2]{stddmwang} gives similar result under the condition that $\zeta\ge1$ or $\zeta=0$ and $0<\si_0< 1$. 
We also provide the following new estimate for the imaginary part of the complex distance which is easier to use in the subsequent analysis.
\begin{lemma}\label{lem-imrho} Let $b_j=\big|[\min\{x_j,y_j\},\max\{x_j,y_j\}]\setminus[-\bar l_j,\bar l_j]\big|$ be the length of the part of the interval from $x_j$ to $y_j$ that lies in the PML domain and where $\si_j=\si_0$. Suppose that $(b_1^2+b_2^2+b_3^2)^\frac12\ge b_0>0$ and $|x_j-y_j|\le a_j+b_j,\; a_j\ge 0, j=1,2,3$. Then we have
\eq{\label{eq_imrho}\im\rho(\tilde{x}, \tilde{y}) > \frac{\si_0 b_0^2}{\big(\sum_{j=1}^3(a_j+b_0)^2\big)^\frac12} . 
}
\end{lemma}
\begin{proof} From \eqref{eq_imr}, we have
\eqn{\im\rho(\tilde{x}, \tilde{y})&\ge \frac{\si_0(b_1^2+b_2^2+b_3^2)}{\big(\sum_{j=1}^3(a_j+b_j)^2\big)^\frac12}\ge\frac{\si_0(b_1^2+b_2^2+b_3^2)}{\big(\sum_{j=1}^3(a_j)^2\big)^\frac12+\big(\sum_{j=1}^3b_j^2\big)^\frac12}\\
&\ge\frac{\si_0 b_0^2}{\big(\sum_{j=1}^3(a_j)^2\big)^\frac12+b_0},}
which implies that \eqref{eq_imrho} holds. This completes the proof.
\end{proof}

It is clear that we can always take $a_j=2\bar l_j$ in Lemma~\ref{lem-imrho}.

Similar to \cite[Lemma 3]{stddmwang},  we have the following lemma on the decay property of the stretched fundamental solution $G(\tilde{x},  \tilde{y})$.
	\begin{lemma}\label{lemma_G}
		There exists a positive constant $C$ independent of $k$ such that for any $x,  y \in \R^3,  x \neq y$, 
		\begin{align*}
			\abs{G(\tilde{x}, \tilde{y})}  & \lesssim e^{- k\im\rho(\tilde{x}, \tilde{y})} \abs{x-y}^{-1},     \\		
			\abs{\na_xG(\tilde{x}, \tilde{y})} & \lesssim e^{- k\im\rho(\tilde{x}, \tilde{y})}\big( k\abs{x-y}^{-1} + \abs{x-y}^{-2} \big),  \\		
			\abs{\na_x\na_yG(\tilde{x}, \tilde{y})} & \lesssim e^{- k\im\rho(\tilde{x}, \tilde{y})}\big( k^2\abs{x-y}^{-1} + k\abs{x-y}^{-2} + \abs{x-y}^{-3} \big).
		\end{align*}
	\end{lemma}
	
	\begin{proof}
		By using \eqref{fs2} and \eqref{rho_d},  we have	
		\eqn{
		\abs{G(\tilde{x}, \tilde{y})} \lesssim e^{-k \im\rho(\tilde{x}, \tilde{y})} \abs{\rho(\tilde{x}, \tilde{y})}^{-1}
		\lesssim e^{- k\im\rho(\tilde{x}, \tilde{y})}\abs{x-y}^{-1}.}
		The other two estimates can be proved by similar arguments.
	\end{proof}
	
%
	

	 The following lemma gives estimates of the integrals of $\abs{G(\tilde{x}, \tilde{y})}$ and $\abs{\na_xG(\tilde{x}, \tilde{y})}$ about $y$ on $\R^3$ for any given $x$.
	 	\begin{lemma}\label{Bound_G}	
	We have
		\begin{align}\label{elem2.5}
			\mathop{\max}_{x\in \R^3} \int_{\R^3} \abs{G(\tilde{x}, \tilde{y})}dy   \lesssim  1, \qquad  			\mathop{\max}_{x\in \R^3} \int_{\R^3} \abs{\na_xG(\tilde{x}, \tilde{y})}dy \lesssim  k, 
		\end{align}	
	\end{lemma}		
	
	\begin{proof}  \cite[Lemma 6]{stddmwang}] provides similar estimates, but without explicit dependence on $k$. Since  the assumption (H1) in \cite[Lemma 6]{stddmwang} is not needed for the two estimates in \eqref{elem2.5}, we include the proof here for the reader's convenience.	Let $\be_{1}=2 \sqrt{3} \be \overline{l}_{\max}$ and $\la=(1-\beta^{-1})^2\si_0/3$.
		By Lemmas \ref{lemma_G} and \ref{Im_roub}, we have		
		\begin{align*}
			\int_{\R^3} \abs{G(\tilde{x}, \tilde{y})}dy &\lesssim  \bigg( \int_{\abs{x-y}<\be_{1}} \frac{1}{\abs{x-y}}dy + \int_{\abs{x-y}>\be_{1}} \frac{e^{-k\la \abs{x-y}}}{\abs{x-y}}dy \bigg) \lesssim  1,  \\
			\int_{\R^3} \abs{\na_x G(\tilde{x}, \tilde{y})}dy &\lesssim  \bigg( \int_{\abs{x-y}<\be_{1}} \Big( \frac{k}{\abs{x-y}} + \frac{1}{\abs{x-y}^2} \Big)dy \\
			&\quad+ \int_{\abs{x-y}>\be_{1}} \Big( \frac{k}{\abs{x-y}} +\frac{1}{\abs{x-y}^2} \Big) e^{-k\la \abs{x-y}} dy \bigg)\lesssim  k,  
		\end{align*}		
This completes the proof.
	\end{proof}

We make the assumption on the stability of the PML problem \eqref{eq_exactsol}: There exists a constant $\cs$ such that for any $f\in L^2(\R^3)$
\eq{\label{A1}\tag{\bf H1} \norml{\tilde u}{\R^3}\ls \cs \norml{f}{\R^3}.}

In fact we can show that $\cs\ls 1$ which is proved in the following lemma, although we expect it to satisfy $\cs\eqsim k^{-1}$ based on the estimate for the PML in spherical coordinates \cite{lw}. Without loss of generality, suppose that $k\cs>1$. 
\begin{lemma}\label{lem:csta}
We have $\cs\ls 1$.
\end{lemma}
\begin{proof} From \eqref{tu} we have
		\begin{align*}
			\norml{\tilde u}{\R^3}^2&= \abs{\int_{\R^3}\int_{\R^3} J(y)G(\tilde{x}, \tilde{y})f(y)\overline{\tilde u (x)} dxdy}\\
			&\ls  \Big(\int_{\R^3}\int_{\R^3} \abs{G(\tilde{x}, \tilde{y})}\abs{f(y)}^2dxdy\Big)^\frac12 \Big(\int_{\R^3}\int_{\R^3} \abs{G(\tilde{x}, \tilde{y})}\abs{\tilde u(x)}^2 dxdy\Big)^\frac12\\
			&\ls  \max_{x\in\R^3} \int_{\R^3} \abs{G(\tilde{x}, \tilde{y})}dy\;\Vert f\Vert_{L^2(\R^3)} \Vert \tilde u\Vert_{L^2(\R^3)},
		\end{align*}
which together with Lemma~\ref{Bound_G} completes the proof of the lemma.
\end{proof}

	\subsection{The PSTDDM for the PML problem in $\R^3$}\label{subs1}
	In this subsection, following 2D PSTDDM  proposed in \cite{pstddm},  we introduce the layer-wise PSTDDM for the PML equation in $\R^3$. The method first decomposes the computational domain into $N$ layers and then convert the computation of the PML problem into solving $2N-2$ the sub PML problems,  each defined on two adjacent layers. 
	
	We start by introducing some notation. For any $a, b\in\R\cup\{\pm\infty\}$,  we denote $\Omega(a, b)=\set{x=(x_1, x_2, x_3)^T\in\R^3:a<x_1<b}$.
	We recall that the domain $\set{x=(x_1, x_2, x_3):\abs{x_1}\leq l_1}$ is divided into $N$ layers \eqref{eq_Nlayers} and
	that $f_i(x)=f(x)|_{\Omega_i}$ for any $x\in\Omega_i$ and $f_i(x)=0$ for any $x\in\R^3\backslash\bar{\Omega}_i$. We take smooth functions ${\be_i^+}(x_1)$ and ${\be_i^-}(x_1)$ satisfying
    \begin{align}\label{beta1}
\begin{cases}{\be_i^+}=1, &\mathrm{as}\ x_1\leq\xi_i,   \\ {\be_i^+}=0,  &\mathrm{as}\ x_1\geq\xi_{i+1},  \end{cases}
		\quad \big|{\be_i^+}'\big| \lesssim (\De\xi)^{-1},\quad \be_i^-=1-\be_i^+.
	\end{align}
	
	The PSTDDM (see \cite[Algorithm 1]{pstddm}) consists of two independent source transfer steps described in Algorithm \ref{alg1}  below.  
\begin{algorithm}
		\caption{Source Transfers for PML Problem in $\R^3$}
		\label{alg1}
		\begin{multicols}{2}
			
			\textbf{Step 1.}
			
			\hrulefill
			
			1. Let $\bar{f}_{1}^+=f_1$;
			
			2. While $i=1, \cdots, N-2$ do
			
			\qquad$\bullet$ Find $u_i^+\in H^1(\R^3)$ such that
			\begin{align}\label{eqs1a}
				J^{-1}\na\cdot & (A\na u_i^+)+k^2u_i^+     \\
				& =-\bar{f}_{i}^+-f_{i+1}\quad \mathrm{in}\ \R^3.\notag
			\end{align}
			
			\qquad$\bullet$ Compute $\Psi_{i+1}^+(\bar{f}_{i}^+)$ in $\Omega_{i+1}$,
			\begin{align*}
				\Psi_{i+1}^+(\bar{f}_{i}^+)=  J^{-1}\na\cdot\left( A \na\be_{i+1}^+ u_i^+\right)\\  + J^{-1}A \na\be_{i+1}^+\cdot\na u_i^+-{\be_{i+1}^+}f_{i+1}.
			\end{align*}
			\qquad$\bullet$ Set
			\begin{align*}
				\bar{f}_{i+1}^+=\begin{cases}
					f_{i+1}+\Psi_{i+1}^+(\bar{f}_{i}^+) \text{ in } \Omega_{i+1},  \\
					0 \text{ elsewhere. }
				\end{cases}
			\end{align*}
			\quad End while;
			
			3. For $i=N-1$,  find $u_{N-1}^+\in H^1(\R^3)$ such that
			\begin{align}\label{eqs1b}
				J^{-1}\na\cdot & (A\na u_{N-1}^+)+k^2u_{N-1}^+    \\
				& =-\bar{f}_{N-1}^+-f_{N}\; \mathrm{in}\ \R^3. \notag
			\end{align}
			
			\columnbreak
			
			\textbf{Step 2.}
			
			\hrulefill
			
			1. Let $\bar{f}_{N}^-=f_N$;
			
			2. While $i=N-1, \cdots, 2$, 
			
			\qquad$\bullet$ Find $u_i^-\in H^1(\R^3)$ such that
			\begin{align}\label{eqs2a}
				J^{-1}\na\cdot & (A\na u_i^-)+k^2u_i^-    \\
				& =-\bar{f}_{i+1}^-\quad \mathrm{in}\ \R^3.\notag
			\end{align}
			\qquad$\bullet$ Compute $\Psi_{i}^-(\bar{f}_{i+1}^-)$ in $\Omega_{i}$,
			\begin{align*}\Psi_{i}^-(\bar{f}_{i+1}^-)=  J^{-1}\na\cdot\left( A\na\be_{i}^- u_i^-\right)  \\+ J^{-1} A {\na\be_{i}^-}\cdot \na u_i^-.
			\end{align*}
			\qquad$\bullet$ Set
			\begin{align*}
				\bar{f}_{i}^-=\begin{cases}
					f_{i}+\Psi_{i}^-(\bar{f}_{i+1}^-) \text{ in } \Omega_{i},  \\
					0 \text{ elsewhere. }
				\end{cases}
			\end{align*}
			\quad End while;
			
			3. For $i=1$,  find $u_1^-\in H^1(\R^3)$ such that
			\begin{align}\label{eqs2b}
				J^{-1}\na\cdot & (A\na u_1^-)+k^2 u_1^-    \\
				& =-f_1-\bar{f}_{2}^-\; \mathrm{in}\ \R^3.\notag
			\end{align}
			
		\end{multicols}
		\hrulefill
				
		\textbf{Step 3.}
				Output $\tilde{u}_\p$:\;  $\tilde{u}_\p=-(u_{i-1}^++u_{i}^-)$ in $ \Omega_i$ where $u_0^+\equiv u_{N}^-\equiv0$.

	\end{algorithm}

	By \eqref{eqs1a},  \eqref{eqs2a} and \eqref{eqs2b},  we know that $u_i^+$ is given by
	\begin{align}
		& u_i^+(x)=\int_{\Omega_i\cup\Omega_{i+1}} (\bar{f}_{i}^++f_{i+1})J(y)G(\tilde{x}, \tilde{y})dy\ , \forall x\in\R^3, \ i=1, \cdots, N-1, \label{u1i}
	\end{align}
	and $u_i^-$ is given by
	\begin{align}
		& u_i^-(x)=\int_{\Omega_{i+1}}\bar{f}_{i+1}^-(y)J(y)G(\tilde{x}, \tilde{y})dy\ \forall x\in\R^3, \ i=N-1, \cdots, 2,  \\
		& u_1^-(x)=\int_{\Omega_1\cup\Omega_2}(\bar{f}_{2}^-(y)+f_1(y))J(y)G(\tilde{x}, \tilde{y})dy.
	\end{align}
	
	Similar to \cite[Lemma 2.6]{stddm} in 2D, we have the following lemma whose proof will be given in the Appendix. This lemma plays a key role in the analysis of this paper.
	
	\begin{lemma}\label{decayex}
		For $i=1, \cdots, N-2$, we have $u_i^{\pm} \in H^1(\R^3)$ and $ \norme{u_i^{\pm}}{\R^3} \lesssim k(k\cs)^i \flay$. Let $\beta=2(2+\sqrt{3}),M_1 =l_{\max} + 2 \sqrt{3} \beta \overline{l}_{\max}, M_i = 2\sqrt{3}\beta \overline{l}_{\max} +2 M_{i-1}$,  where $l_{\max}=\max(l_1,l_2,l_3)$ and $\bar{l}_{\max}=\max(\bar{l}_1, \bar{l}_2, \bar{l}_3)$. Then we have
		\begin{align}\label{el2.7}
			k\abs{u_i^{\pm}(x)}+\abs{\nabla u_i^{\pm}(x)} \lesssim k^2(k\cs)^{i-1} e^{-\frac{1}{8}k\si_0\abs{x}}\flay \quad\forall \abs{x}\ge {M_i}.
		\end{align}		
	\end{lemma}
	
	We remark that,  in the following we will only use the consequence of this lemma that $\abs{u_i^{\pm}(x)}$ and $\abs{\nabla u_i^{\pm}(x)}$ decay exponentially as $\abs{x}\rightarrow \infty$. The estimate between $ \norme{u_i^{\pm}}{\R^3} $ and $ \flay$ will be strengthened later in next section.
%
%
%
	
	Since $\bar{f}_{i}^{\pm}=0$ in $\Omega_{i+1}$, by simple calculation,  we have 
	\begin{align}
		\Psi_{i+1}^+(\bar{f}_{i}^+)&=J^{-1}\na\cdot (A\na ({\be_{i+1}^+}u_i^+)) +k^2({\be_{i+1}^+}u_i^+)\label{psi1}\;\text{in }  \Omega_{i+1},\\
		\Psi_{i}^-(\bar{f}_{i+1}^-)&=J^{-1}\na\cdot (A\na (\be_{i}^-u_i^-)) +k^2(\be_{i}^-u_i^-)\label{psi2}\quad\text{in }  \Omega_{i}.
	\end{align} 
	Similar to \cite[Theorem 2.3]{pstddm} in 2D, we have the following theorem which verifies the source transfer identities and give the integral representations of $u_i^\pm$. 
The proof is similar to that of \cite[Theorem 2.7]{stddm} and is omitted here.	
	\begin{theorem}\label{th1}
		The following assertions hold:
		\begin{enumerate}
			\item[{\rm (i)}] For the source transfer operators $\Psi_{i}^{+}, i=1, \cdots, N-2$,  we have,  
		\end{enumerate}
		\begin{align}\label{eqt1a}
			\int_{\Omega_i} \bar{f}_{i}^+(y)\tilde{G}(x, y)dy&=\int_{\Omega_{i+1}} \Psi_{i+1}^+(\bar{f}_{i}^+)(y)\tilde{G}(x, y)dy,\quad x\in\Omega(\xi_{i+2},  +\infty),\\
			\label{eqt1b}
			\int_{\Omega_{i+2}} \bar{f}_{i+2}^-(y)\tilde{G}(x, y)dy&=\int_{\Omega_{i+1}} \Psi_{i+1}^-(\bar{f}_{i+2}^-)(y)\tilde{G}(x, y)dy, \quad x\in\Omega(-\infty, \xi_{i+1}).
		\end{align}
		
		\begin{enumerate}
			\item[{\rm (ii)}] For the solution $u_i^+$ in \eqref{eqs1a},  we have,  
		\end{enumerate}
		\begin{align}\label{u1ia}
			u_i^+(x)&=\int_{\Omega(-\infty, \xi_{i+2})} f(y)G(\tilde{x}, \tilde{y})dy, \quad x\in\Omega_{i+1},\;  i=1, \cdots, N-1,\\
		\label{u1ib}
			u_i^-(x)&=\int_{\Omega(\xi_{i+1}, +\infty)} f(y)G(\tilde{x}, \tilde{y})dy,\quad x\in\Omega_{i},\;  i=2, \cdots, N-1,\\
		\label{u1b}
			u_1^-(x)&=\int_{\R^3} f(y)G(\tilde{x}, \tilde{y})dy,\quad x\in\Omega_1.
		\end{align}
	\end{theorem}

	The source transfer identities mean that  the source $\bar{f}_{i}^+$ in $\Omega_{i}$ and  transferred source $\Psi_{i+1}^+(\bar{f}_{i}^+)$ in $\Omega_{i+1}$ generate the same field in $\Omega(\xi_{i+2}, +\infty)$; and  the source $\bar{f}_{i+2}^-$ in $\Omega_{i+2}$ and the transferred  source $\Psi_{i+1}^-(\bar{f}_{i+2}^-)$ in $\Omega_{i+1}$  generate the same field in $\Omega(-\infty, \xi_{i+1})$,$i=1, \cdots, N-2$. 
	
	
	
	By Theorem \ref{th1} ,  we could obtain the main result in this section.
	\begin{theorem}\label{th3}
It holds that $\tilde{u}=\tilde{u}_\p$ in $\R^3$, where $\tilde u$ is the non-truncated PML solution to \eqref{eq_exactsol} and $\tilde u_\p$ is the output of  Algorithm~\ref{alg1}.
	\end{theorem}
	\begin{proof}
		For any $x\in\Omega_i$,  $i=1, \cdots, N$,  it follows from \eqref{u1ia}--\eqref{u1b} and \eqref{tu} that, 
		\begin{align*}
			\tilde{u}(x) & =-\int_{\R^3}f(y) G(\tilde{x}, \tilde{y})dy   =-\bigg(\int_{\Omega(-\infty, \xi_{i+1})}+\int_{\Omega(\xi_{i+1}, +\infty)}\bigg)f(y)G(\tilde{x}, \tilde{y})dy \\
			& =-(u_{i-1}^+(x)+u_{i}^-(x)), 
		\end{align*}
		where we have used that $f$ is supported in $B_l$. This completes the proof of the theorem.
	\end{proof}
	
	\subsection{The PSTDDM for the truncated PML problem} \label{subs2}
	Note that the PML problems \eqref{eqs1a}--\eqref{eqs2b} in Algorithm~\ref{alg1} are defined in $\R^3$. In this section we study the PSTDDM method for the truncated PML equation.
	
	Let $B_L=(-L_1, L_1)\times(-L_2, L_2)\times(-L_3, L_3),  L_i \geq \bar l_i,  i=1, 2, 3$ be a rectangle that includes $B_l$. Let $L_{\max}=\max(L_1, L_2, L_3)$. Considering the following truncated PML problem in the bounded domain $B_L$ whose solution approximates $\tilde{u}$ in \eqref{pmleq}:
	\begin{align}
		J^{-1}\na\cdot(A\na\hat{u})+k^2\hat{u}=f \qquad     & \text{in } B_L  \label{eq_pmlBL}  \\
		\hat u = 0 \qquad & \text{on } \pa B_L. \label{eq_pmlBLBC}
	\end{align}	
	The corresponding weak formulation reads as: Find $\hat u\in H_0^1(B_L)$ such that
	\begin{align}
		(A\na \hat{u}, \na v) - k^2 (J\hat{u}, v) = - \pd{Jf, v}\quad \forall v\in H_0^1(B_L). \label{eq_tpml}
	\end{align}
	where $(\cdot, \cdot)$ is the inner product in $L^2(B_L)$,  and $\pd{\cdot, \cdot}$ is the duality pairing between $H_0^1(B_L)'$ and $H_0^1(B_L)$. 
	To facilitate the analysis, we make the following assumption: 
	\eq{\label{A2} \tag{\bf{H2}}
	 \qquad d_j:=L_j-l_j\le 2(L_j-\bar{l}_j),\; j=1, 2, 3. \quad \text{Set }  d:=\min\{d_1, d_2, d_3\}.}
 
 	We introduce local PML problems by using the PML complex coordinate stretching outside the domain $(\xi_i, \xi_{i+2})\times(-l_2, l_2)\times(-l_3, l_3)$. The PML stretching is $\tilde{x}^i(x)=(\tilde{x}_1^i(x_1), \tilde{x}_2(x_2), \tilde{x}_3(x_3))^T$,  where
	\begin{equation}\label{tx1}
		\tilde{x}_1^i(x_1)=x_1+(\ze+\ii)\int_{\xi_{i+1}}^{x_1}\si_1^i(t) dt,  \\
		\si_1^i(t):=
		\begin{cases}
			\si_1(t+\xi_{N+1}-\xi_{i+2}) & \text{if } t >\xi_{i+1},  \\
			\si_1(t-\xi_i+\xi_1)  & \text{if } t\le\xi_{i+1}.
		\end{cases}
	\end{equation}
	We define $J_i(x)=\tilde{x}_1^i(x_1)' \tilde{x}_2 (x_2)' \tilde{x}_3 (x_3)'$ and
	$$
	A_i(x)=\mathrm{diag}\left(\frac{\tilde{x}_2(x_2)'\tilde{x}_3(x_3)'}{\tilde{x}_1^i(x_1)'}, \frac{\tilde{x}_1^i(x_1)'\tilde{x}_3(x_3)'}{\tilde{x}_2(x_2)'}, \frac{\tilde{x}_1^i(x_1)'\tilde{x}_2(x_2)'}{\tilde{x}_3(x_3)'}\right)
	$$
	Denote by $\Omega_i^\pml=\set{x=(x_1, x_2, x_3)\in B_L:\xi_i-d_1\leq x_1\leq\xi_{i+2}+d_1}$.
	The local PML problems in truncated domains can be defined for some wave source $F\in H^1(\Omega_i^\pml)'$ as: find $\phi\in H_0^1(\Omega_i^\pml)$ such that
	\begin{align}\label{TPML}
		(A_i\na \phi, \na \psi) - k^2(J_i\phi, \psi)=-\pd{J_iF, \psi}\quad \forall \psi\in H_0^1(\Omega_i^\pml).
	\end{align}
	Then the PSTDDM (see \cite[Algorithm 2]{pstddm}) for the PML equation in a truncated bounded domain is described in Algorithm~\ref{alg2}.
	
	\begin{algorithm}
		\caption{Source Transfers for Truncated PML Problem}
		\label{alg2}
		\begin{multicols}{2}
			
			\textbf{Step 1.}
			
			\hrulefill
			
			1. Let $\hat{f}_{1}^+=f_1$;
			
			2. While $i=1, \cdots, N-2$ do
			
			\qquad$\bullet$ Find $\hat{u}_i^+\in H_0^1(\Omega_i^{\pml})$, 
			\begin{align}\label{eqs3a}
				J_i^{-1}\na\cdot & (A_i\na \hat{u}_i^+) + k^2\hat{u}_i^+ \\
				& = - \hat{f}_{i}^+ - f_{i+1}.\notag
			\end{align}
			
			\qquad$\bullet$ Compute $\hat{\Psi}_{i+1}^+(\hat{f}_{i}^+)$ in $\Omega_{i+1}\cap B_L$, 
			\begin{align*} \hat{\Psi}_{i+1}^+(\hat{f}_{i}^+) =  J_i^{-1}\na\cdot\left( A_i \na\be_{i+1}^+ \hat u_i^+\right)   \\
				 + J_i^{-1}A_i \na\be_{i+1}^+\cdot\na \hat u_i^+-\be_{i+1}^+f_{i+1}.
			\end{align*}
			\qquad$\bullet$ Set
			\begin{align*}
				\hat{f}_{i+1}^+=\begin{cases}
					f_{i+1}+\hat{\Psi}_{i+1}^+(\hat{f}_{i}^+) \text{ in } \Omega_{i+1}\cap B_L,  \\
					0 \text{ elsewhere. }
				\end{cases}
			\end{align*}
			\quad End while;
			
			3. For $i=N-1$,  find $\hat{u}_{N-1}^+\in H_0^1(\Omega_{N-1}^{\pml})$, 
			\begin{align}\label{eqs3b}
				J_{N-1}^{-1}\na\cdot & (A_{N-1} \na\hat{u}_{N-1}^+) + k^2\hat{u}_{N-1}^+ \\
				& = - \hat{f}_{N-1}^+ - f_{N}.\notag
			\end{align}

			\columnbreak
			
			\textbf{Step 2.}
			
			\hrulefill
			
			1. Let $\hat{f}_{N}^- =f_N$;
			
			2. While $i=N-1, \cdots, 2$, 
			
			\qquad$\bullet$ Find $\hat{u}_i^-\in H_0^1(\Omega_i^{\pml})$, 
			\begin{align}\label{eqs4a}
				J_i^{-1}\na\cdot & (A_i\na \hat{u}_i^-) + k^2\hat{u}_i^- \\
				& = - \hat{f}_{i+1}^-.\notag
			\end{align}
			\qquad$\bullet$ Compute $\hat{\Psi}_i^-(\hat{f}_{i+1}^-)$ in $\Omega_i\cap B_L$, 
			\begin{align*}
				\hat{\Psi}_i^-(\hat{f}_{i+1}^-)=  J_i^{-1}\na\cdot(A_i\na \be_i^-\hat{u}_i^-) \\
				 + J_i^{-1}A_i\na \be_i^-\cdot\na\hat{u}_i^-.
			\end{align*}
			\qquad$\bullet$ Set
			\begin{align*}
				\hat{f}_i^-=\begin{cases}
					f_i+\hat\Psi_i^-(\hat{f}_{i+1}^-) \text{ in } \Omega_i\cap B_L,  \\
					0 \text{ elsewhere. }
				\end{cases}
			\end{align*}
			\quad End while;
			
			3. For i=1,  find $\hat{u}_{1}^-\in H_0^1(\Omega_1^{\pml})$, 
			\begin{align}\label{eqs4b}
				J_1^{-1}\na\cdot & (A_1\na \hat{u}_{1}^-) + k^2\hat{u}_{1}^- \\
				& = - \hat{f}_{2}^- - f_1.\notag
			\end{align}
		\end{multicols}
		\hrulefill
				
		\textbf{Step 3.}
		
		1. Let  $\hat{u}_{0}^+=\hat{u}_{N}^-=0$ and $\hat{v}=-(\hat{u}_{i-1}^++\hat{u}_{i}^-)$ in $ \Omega_i\cap B_L$, $i=1, \cdots, N$.
				
		2. Output $\hat{u}_\p \in H_0^1(B_L)$ by solving
\begin{align}\label{eup}
(\nabla \hat{u}_\p, \nabla v) + k^2(\hat{u}_\p, v) = \sum_{i = 1}^{N} \big((\nabla \hat{v}, \nabla v)_{\Omega_i \cap B_L} + k^2(\hat{v}, v)_{\Omega_i \cap B_L}\big) \; \forall v \in H_0^1(B_L).
\end{align}
	\end{algorithm}
	We remark that, inspired  by Algorithm~\ref{alg1} and Theorem~\ref{th3}, $\hat v$ given in Step~3 of Algorithm~\ref{alg2} provides a good approximation to $\tilde u$, but it is generally discontinuous at the interfaces $x_1=\xi_i$ and hence does not belong to $H^1(B_L)$. To remedy this, we output a lifting \( \hat{u}_\p \in H_0^1(B_L) \) of $\hat v$ in Step~3 of Algorithm~\ref{alg2} by solving \eqref{eup}. Note that this problem \eqref{eup} is positive definite and can be solved by efficient and cheap iteration methods 
(see e.g. \cite{Xu1992,2005Domain}).
We call this $\hat{u}_\p$ the layer-wise PSTDDM approximation of $\tilde u$.

	\subsection{Convergence analysis}
	In this section we will show the $\inf$--$\sup$ condition for the sesquiliner form associated with the PML problem in $\R^3$ and $B_L$, the exponential convergence of the solution obtained by the PSTDDM for the truncated PML problem to the solution of the whole space with the increase of PML layer thickness or the PML medium property $\si_0$. 
	\subsubsection{Convergence of the PML method}
	
	Similar to \cite[Lemmas 2.4--2.5]{r2r3}, we have the following stability result.
\begin{lemma}\label{lem:pmldefinite} Let $0\le \ze\ls 1$. Then there exists a positive constant $\la\eqsim 1$ independent of $k$ such that 
		\begin{align}\label{infsup+}					\norme{u}{\R^3} \lesssim \sup_{v\in H^1(\R^3)} \frac{ \abs{(A\na u, \na v)+\la k^2  (J u, v)} }{ \norme{v}{\R^3}  }
			  \quad \forall u \in H^1(\R^3)
		\end{align}

\end{lemma}
\begin{proof} Following the proof of \cite[Lemma 2.4]{r2r3},for $u$ be in $H^1(\R^3)$, we set $v = e^{-\ii \theta}u$ where $\theta = \arg\big((1+(\ze+\ii)\si_0)J^{-1}\big)$.
	Here we take the branch cut for the logarithm to be the negative real axis. Simple calculation shows that 
	\eqn{\theta=\arg(1+(\ze+\ii)\si_0)-\sum_{j=1}^3\arg\al_j(x_j)
			=\arctan\frac{\si_0}{1+\ze\si_0}-\sum_{j=1}^3\arctan\frac{\si_j(x_j)}{1+\ze\si_j(x_j)}.}
		Since the function $s/(1+\ze s)$ is increasing for $s\ge0$, each $\arg\al_j(x_j)$ belongs to $[0,\arg(1+(\ze+\ii)\si_0)]$. Therefore $-\pi<\theta<\pi/2$.
		The assumptions on $\si$ imply that $v$ is in $H^1(\R^3)$. Moreover
		\begin{align*} 
\Big|\frac{\partial \theta}{\partial x_j}\Big| = \Big|-\im \Big( \i\frac{\si '(x_j)}{\al_j(x_j)}  \Big)\Big|\ls 1.
		\end{align*}
Next we only sketch the rest of the proof, as it is quite similar to that of \cite[Lemma 2.4]{r2r3}. Following the approach used in the  proof of \cite[Lemma 2.4]{r2r3}, we may obtain the following estimates:
\eq{
\norme{v}{\R^3} &\ls  \norme{u}{\R^3},\\
\re (A\na u, \na v)  &\geq C_1 \abs{u}_{H^1(\R^3)}^2 - C_2  \norml{u}{\R^3}^2,\\
\re ( J u, v ) &\ge C_3\norml{u}{\R^3}^{2},
}
where $C_j, j=1,2,3$ are positive constant independent of $k$ but may polynomially depend on $\si_0$. Since $k\gtrsim 1$, there exists a positive constant $\la\eqsim1$ such $\la k^2 C_3-C_2\gtrsim k^2$ and hence
\eqn{\re \big((A\na u, \na v)+\la k^2 ( J u, v )\big)\gtrsim \norme{u}{\R^3}^2\gtrsim \norme{u}{\R^3}\norme{v}{\R^3}.}
This completes the proof of the lemma.
\end{proof}
\begin{remark} {\rm (i)} In \cite[Lemmas 2.4 and 2.5]{r2r3} similar stability results were established in a certain $H^1$-equivalent norm for the cases $\ze=0$ and $\ze=1$ respectively. In contrast, our results hold for all $ 0\le \ze\ls 1$ and use the energy norm $\norme{\cdot}{\R^3}$ instead. 

{\rm (ii)} The work \cite{stddmwang} studied a STDDM for the time harmonic Maxwell equations under the assumption the PML parameters are assumed to satisfy either $\ze\ge 1$ or $\ze=0, 0<\si_0< 1$. In our setting, we allow $0\le \ze\ls 1$ and $\si_0\ge 1$.

\end{remark}

	The following lemma proves the stability of the PML problem in $\R^3$.
	\begin{lemma}\label{rstab}
		Let $0\le \ze\ls 1$ and \eqref{A1} holds. Then  for any $F\in H^1(\R^3)'$, there  exists a unique solution $w\in H^1(\R^3)$ of the problem
		\begin{align}\label{runi}
			(A\na w, \na \psi) - k^2(Jw, \psi)=\pd{F, \psi}\quad \forall \psi\in H^1(\R^3).
		\end{align}
		Moreover, we have $\norme{w}{\R^3} \lesssim  \mu_0^{-1}\norme{F}{\R^3}^*$ with $\mu_0^{-1}:=k^2\cs$. Consequently, there holds the following $\inf$--$\sup$ condition for the sesquiliner form associated with the PML problem in $\R^3$:
	\begin{align}\label{rinfsup}
		\sup_{\psi\in H^1(\R^3)} \frac{\abs{(A\na\phi, \na\psi)-k^2(J\phi, \psi)}}{\norme{\psi}{\R^3}} \gtrsim \mu_0
		\norme{\phi}{\R^3}  \quad \forall \phi\in H^1(\R^3), 
	\end{align}

	\end{lemma}
	\begin{proof}
	For any $F\in H^1(\R^3)'$, by Lemma~\ref{lem:pmldefinite},  there exists a unique solution $v\in H^1(\R^3)$ of the problem
		\begin{align*}
			(A\na v, \na \psi) + \la k^2(Jv, \psi)=\pd{F, \psi} \quad \forall \psi\in H^1(\R^3).
		\end{align*}
		Moreover,  $\norme{v}{\R^3} \lesssim  \norme{F}{\R^3}^*$. For $F_1 = (1+\la)k^2 v \in L^2(\R^3)$, from the assumption \eqref{A1}, there exists a unique $\Phi\in H^1(\R^3)$ satisfying  
		\begin{align}\label{}
			(A\na \Phi, \na \psi)- k^2(J\Phi, \psi)=\pd{JF_1, \psi}\quad \forall \psi\in H^1(\R^3).
		\end{align}
It is clear that $w=v+\Phi$ satisfies \eqref{runi}. Moreover $\norml{\Phi}{\R^3}\ls \cs \norml{F_1}{\R^3}$ and 
\eqn{\norme{\Phi}{\R^3}&\ls \sup_{v\in H^1(\R^3)} \frac{ \abs{(A\na\Phi, \na v)+\la k^2  (J \Phi, v)} }{\norme{v}{\R^3}}\\
&=\sup_{v\in H^1(\R^3)} \frac{\abs{ \pd{JF_1, v}+(\la+1) k^2  (J \Phi, v)}} {\norme{v}{\R^3}}\\
&\ls k^{-1}\norml{F_1}{\R^3}+k{\norml{\Phi}{\R^3}}\\
&\ls k\cs \norml{F_1}{\R^3}\ls k^3\cs \norml{v}{\R^3}\ls k^2\cs \norme{F}{\R^3}^*.}		
This concludes the proof.     
	\end{proof}

	Then we will show the sesquiliner form associated with \eqref{eq_tpml} satisfies the $\inf$--$\sup$ condition, which gives the stability of the truncated PML problem.
	
	\begin{lemma}\label{blstab}
		Under the conditions of \eqref{A1}-\eqref{A2}, for sufficiently large $k\si_0d$, the following $\inf$--$\sup$ condition holds for the truncated PML problem \eqref{eq_tpml}:
		\begin{align}\label{blinfsup}
			\sup_{\psi\in H_0^1(B_L)} \frac{\abs{(A\na\phi, \na\psi)-k^2(J\phi, \psi)}}{\norme{\psi}{B_L}}\gtrsim \mu_0
			\norme{\phi}{B_L}  \quad \forall \phi\in H_0^1(B_L), 
		\end{align}
		where $\mu_0^{-1}= k^2\cs$ is given in Lemma~\ref{rstab}.
	\end{lemma}
	
	\begin{proof} We prove the lemma by using the reflection argument in \cite[Theorem 5.5]{r2r3} and \cite[Lemma 3.4]{stddm}.
		For any $\phi \in H_0^1(B_L)$, define $F\in H^{-1}(B_L)$ by
		\begin{align*}
			\pd{F, \psi}=(A\na \phi, \na \psi) - k^2(J\phi, \psi), \quad \forall \psi\in H_0^1(B_L).
		\end{align*}
		Then the $\inf$--$\sup$ condition(\ref{blinfsup}) is equivalent to show $\norme{\phi}{B_L}\lesssim  \mu_0^{-1} \norme{F}{B_L}^*$.
		Denote by $B_L^R=(-2L_1+\bar{l}_1, 2L_1-\bar{l}_1)\times(-2L_2+\bar{l}_2, 2L_2-\bar{l}_2)\times(-2L_3+\bar{l}_3, 2L_3-\bar{l}_3)$. We
		define $\phi^R \in H^1(B_L^R)$ as follows, which is the reflection extension of $\phi=(\phi_1, \phi_2, \phi_3)^T$.
		
		\quad First define $\phi_1$ by
\begin{align*}
			\phi_1(x)= 
			\begin{cases}
			-\phi(2L_1-x_1, x_2, x_3), &\ \mathrm{if}\ x_1 \in (L_1, 2L_1-\bar{l}_1), \\       
				\phi(x), &\ \mathrm{if}\ x_1 \in (-L_1, L_1), \\
				-\phi(-2L_1-x_1, x_2, x_3),&\ \mathrm{if}\ x_1 \in (-2L_1+\bar{l}_1, -L_1). \\ 		
			\end{cases}
		\end{align*}
		Obviously, $\phi_1$ is the extension of $\phi$ from $B_L$ to $B_L^{R_1}:=(-2L_1+\bar{l}_1, 2L_1-\bar{l}_1)\times(-L_2, L_2)\times(-L_3, L_3)$ by odd reflection with respect to $x_1=\pm L_1$.  Similarly, we define $\phi_2$ as the extension of $\phi_1$ from $B_L^{R_1}$ to $B_L^{R_2}:=(-2L_1+\bar{l}_1, 2L_1-\bar{l}_1)\times(-2L_2+\bar{l}_2, 2L_2-\bar{l}_2)\times(-L_3, L_3)$ by odd reflection with respect to $x_2=\pm L_2$:
\begin{align*}
			\phi_2(x)= 
			\begin{cases}
			-\phi_1(x_1, 2L_2-x_2, x_3), &\ \mathrm{if}\ x_2 \in (L_2, 2L_2-\bar{l}_2), \\       
				\phi_1(x), &\ \mathrm{if}\ x_2 \in (-L_2, L_2), \\
				-\phi_1(x_1, -2L_2-x_2, x_3),&\ \mathrm{if}\ x_2 \in (-2L_2+\bar{l}_2, -L_2). \\ 		
			\end{cases}
		\end{align*}		
Finally, we define $\phi^R$ as the extension of $\phi_2$ from $B_L^{R_2}$ to $B_L^{R}$  by odd reflection with respect to $x_3=\pm L_3$:
\begin{align*}
			\phi^R(x)= 
			\begin{cases}
			-\phi_2(x_1, x_2, 2L_3-x_3), &\ \mathrm{if}\ x_3 \in (L_3, 2L_3-\bar{l}_3), \\       
				\phi_2(x), &\ \mathrm{if}\ x_3 \in (-L_3, L_3), \\
				-\phi_2(x_1, x_2, -2L_3-x_3),&\ \mathrm{if}\ x_3 \in (-2L_3+\bar{l}_3, -L_3). \\ 		
			\end{cases}
		\end{align*}		
		We know $\norme{\phi^{R}}{B_L^R} \lesssim  \norme{\phi}{B_L}$.   
		Next we define the functional $\tilde{F} \in H^{-1}(B_L^R)$ by
		\begin{align*}
			\pd{\tilde{F}, \psi}=(A\na \phi^R, \na \psi) - k^2(J\phi^R, \psi) \quad \forall \psi\in H_0^1(B_L^R).
		\end{align*}
		It is clear that $\tilde{F}$ is an extension of $F \in H^{-1}(B_L)$ and since $\si_j=\si_0 $ for $\abs{t}\geq \bar{l}_j$, we have $\normE{\tilde{F}}_{B_L^R}^* \lesssim  \norme{F}{B_L}^*$.
		Now we extend $\tilde{F}$ to a functional
		$F_1 \in H^1(\R^3)'$ by Hahn-Banach theorem such that $\norme{F_1}{\R^3}^*=\normE{\tilde{F}}_{B_L^R}^*  \lesssim \norme{F}{B_L}^*$. For $F_1 \in H^1(\R^3)'$
		we use Lemma \ref{rstab} and the $\inf$--$\sup$ condition (\ref{rinfsup}) to conclude that there exists a $v \in H^1(\R^3)$ such that
		\begin{align*}
			(A\na v, \na \psi) - k^2(Jv, \psi)=\pd{F_1, \psi}, \quad \forall \psi\in H^1(\R^3).
		\end{align*}
		and satisfies $\norme{v}{\R^3}\lesssim  \mu_0^{-1}\norme{F_1}{\R^3}^*\lesssim  \mu_0^{-1}\norme{F}{B_L}^* $. Thus
		\begin{align*}
			\norme{\phi}{B_L} \le  \norme{v}{B_L}+\norme{v-\phi^R}{B_L} \le  C\mu_0^{-1}\norme{F}{B_L}^*+\norme{v-\phi^R}{B_L}
		\end{align*}
		It is easy to see that $ v-\phi^R \in H^1(B_L^R)$ satisfies the homogeneous PML equation
		\begin{align*}
			\big(A\na (v-\phi^R), \na \psi \big) - k^2\big(J(v-\phi^R), \psi\big)=0, \quad \forall \psi\in H_0^1(B_L^R)
		\end{align*}
		Following the proof in \cite[Theorem 5.5]{r2r3} while tracing the dependence on $k$, one can show that $\norme{v-\phi^R}{B_L}$decays exponentially and can be bounded by $\frac12\norme{\phi}{B_L}$
		for sufficiently large $k\si_0d$. This completes the proof. 
	\end{proof}	
	
		\begin{remark}\label{rmkmu_order}
			By Lemma~\ref{rstab}, $\mu_0^{-1}=k^2\cs$. The present estimate $\cs\lesssim1$ in Lemma~\ref{lem:csta} yields only $\mu_0\gtrsim k^{-2}$, which is not expected to be sharp. We expect the lemma holds with $\mu_0\gtrsim k^{-1}$. Actually this has been proved for the radial PML \cite{lw,FreletNicaise2022} and the uniaxial PML in 2D with constant PML medium parameter \cite{WangZheng25}, although it has not been proved for the uniaxial PML setting of this paper yet.
			%
		\end{remark}

To proceed, we notice that for $x\in \pa B_L$,  $y\in B_l$, by Lemma~\ref{lem-imrho} with $b_0=\min_{j=1,2,3}(L_j-\bar l_j)$ and $a_j=2\bar l_j$ and \eqref{A2}, we have (cf. \cite[(6.3)]{chen2016elastic})
\begin{align}\label{egamma}
\im\rho(\tilde{x}, \tilde{y})&\ge\ga\bar\si, \quad\ga:=\frac{d}{8\big(\sum_{j=1}^{3}L_j^2\big)^\frac12},\quad\bar\si:=\si_0 d.
		\end{align}

	\begin{lemma}\label{lemma_tildeuBC}
		Under Assumption \eqref{A2}, for sufficiently large $k\si_0d$, the following estimate holds for all $x\in \pa B_L$: 
		\begin{align*}
			\normE{\tilde{u}}_{\frac12,\partial B_L} \lesssim  k^2 e^{-\ga k \bar{\si}} \normef{f}{B_l}.
		\end{align*}
	\end{lemma}
	
	\begin{proof} From \eqref{leq_vga},  we have
\eqn{\normE{\tilde{u}}_{\frac12,\partial B_L}\ls \max_{x\in\pa B_L}	\big(k^\frac12 \abs{\tilde{u}(x)} + \abs{\nabla\tilde{u}(x)}\big)}
By \eqref{tu}, \eqref{A2}, and Lemma \ref{lemma_G}, together with the fact that $\abs{x-y}\geq d$ for $x\in \pa B_L$ and  $y\in B_l$, we have
		\begin{align}
			\abs{\tilde{u}(x)} &\ls  \normef{f}{B_l}\max_{y\in B_l}\big( k
			\abs{G(\tilde{x}, \tilde{y})}+\abs{\nabla_yG(\tilde{x}, \tilde{y})}\big)\notag\\
			&\ls  e^{- k\im\rho(\tilde{x}, \tilde{y})}(kd^{-1}+kd^{-1}+d^{-2})\normef{f}{B_l} \notag\\
			&\ls  ke^{-k\im\rho(\tilde{x}, \tilde{y})}\normef{f}{B_l}.\label{eq_tu} 
		\end{align}
	Similarly,  one can show that for any $x \in \pa B_L$,  $\abs{\nabla\tilde{u}(x)}\lesssim  k^2e^{-k\im\rho(\tilde{x}, \tilde{y})}\normef{f}{B_l}$, which together with \eqref{egamma} completes the proof.    
	\end{proof}	
	
	\begin{theorem}\label{ex}
		Under the conditions of \eqref{A1}-\eqref{A2}, for sufficiently large  $k\si_0d$, we have
		\begin{align}\label{exleq}
			\norme{\tilde{u}-\hat{u}}{B_L}  \lesssim  k^4\cs e^{-\ga k \bar{\si}} \normef{f}{B_l}, 
		\end{align}
		where $L$ is the diameter of $B_L$ ,  $\ga$ and $\bar{\si}$ are defined in \eqref{egamma}. Here $\tilde{u}$ is defined in (\ref{tu})
		and $\hat{u}$ is the solution of the problem (\ref{eq_tpml}). 	
	\end{theorem}
	
	\begin{proof}
		For $\tilde{u} \in H^1(\R^3)$, it is clear that its trace $\tilde{u} \in H^{\frac12}(\pa B_L)$. By \cite[Lemma 4.22]{melenk2011}, there exists a lifting function $\phi\in H^1(B_L)$ such that $\tilde{u}=\phi$ on $\pa B_L$ and$\norme{\phi}{B_L}\lesssim  \normE{\tilde{u}}_{\frac12,\pa B_L}$.
		Since $\tilde{u}-\hat{u}-\phi \in H_0^1(B_L)$, by \eqref{blinfsup} and the equations \eqref{tu} and \eqref{eq_tpml} satisfied by $\tilde{u}$ and $\hat{u}$,  we obtain
		\begin{align*}
			\norme{\tilde{u}-\hat{u}-\phi}{B_L} \lesssim 
			\mu_0^{-1}\sup_{\psi\in H_0^1(B_L)} \frac{\abs{(A\na\phi, \na\psi)-k^2(J\phi, \psi)}}{\norme{\psi}{B_L}}\lesssim \mu_0^{-1} \norme{\phi}{B_L}
		\end{align*}
		This implies by \eqref{leq_vga} and Lemma~\ref{blstab} that
		\begin{align*}
			\norme{\tilde{u}-\hat{u}}{B_L} &\ls  (1+\mu_0^{-1})\norme{\phi}{B_L}\ls  k^2\cs \norme{\phi}{B_L}\ls  k^2\cs\normE{\tilde{u}}_{\frac12,\pa B_L}.
		\end{align*}
	Combining with Lemma~\ref{lemma_tildeuBC}, we complete the proof.    
	\end{proof}

	\subsubsection{Convergence of the PSTDDM in truncated area}
	
	Similar to (\ref{blinfsup}),  the following $\inf$--$\sup$ condition holds for the truncated PML problem \eqref{TPML} if {$k\si_0d$} is sufficiently large.
	\begin{align}\label{localinfsup}
		\sup_{\psi\in H_0^1(\Omega_i^\pml)} \frac{\abs{(A_i\na\phi, \na\psi)-k^2(J_i\phi, \psi)}}{\norme{\psi}{\Omega_i^\pml}}\gtrsim \mu_0\norme{\phi}{\Omega_i^\pml}\quad \forall \phi\in H_0^1(\Omega_i^\pml), 
	\end{align}
	where $\mu_0^{-1}=k^2\cs$.
	
 In order to estimate the PML truncation errors of Algorithm~\ref{alg2},  we introduce auxiliary functions $\bar{u}_{i}^\pm, \ i=1, \cdots, N-1$ , which are defined as follows:
	\begin{align}\label{pmlalg1}
		J_i^{-1}\nabla \cdot (A_i\nabla \bar u_i^\pm)+k^2\bar u_i^\pm = F_i^\pm\quad \mathrm{in}\ \R^3,  \quad i=1, \cdots,  N-1, 
	\end{align}
	where
	\begin{align}\label{pmlalg1F}
		\begin{split}
			F_i^+&=-\bar{f}_{i}^+-f_{i+1},  \quad i=1, \cdots,  N-1, \\
			F_i^-&=-\bar{f}_{i+1}^-, \quad i=2, \cdots,  N-1, \quad F_1^-=-f_1-\bar{f}_{2}^-,
		\end{split}
	\end{align}	
and	$\bar f_i^\pm$ are from Algorithm~\ref{alg1}.
 	It is clear that
	\begin{align}\label{barui}
		u_i^\pm=\bar{u}_{i}^\pm \quad\text{in}\quad \Omega_i\cup\Om_{i+1}.
	\end{align}
	We want to show that $\hat{f}_{i}^\pm$ is a good approximation of $\bar{f}_{i}^\pm$ in Algorithm~\ref{alg1}. Before it,  we proof two Lemmas that are needed. Given $t\in (\frac12,1)$, let
\eq{\label{Omi+-}
\Om_{i,t}^+ := \{x\in \Om_{i}:\;\abs{x_2}\geq l_2+t d_2 \ \mathrm{or}\  \abs{x_3}\geq l_3+t d_3 \},\quad \Om_{i,t}^- :=\Om_{i}\setminus \Om_{i,t}^+.}

\begin{lemma}\label{lem3_}
		Under the condition of  \eqref{A2}, if $k\si_0d$ is sufficiently large,   then we have for $i=1, \cdots, N-1$, 
		\begin{align*}
k\abs{\bar{u}_{i}^+(x)}+\abs{\nabla\bar{u}_{i}^+(x)}&\lesssim  k^2e^{- (2t-1)^2 k\ga\bar{\si}}\flay\quad\forall x\in \Om_{i+1,t}^+,\\
k\abs{\bar{u}_{i}^-(x)}+\abs{\nabla\bar{u}_{i}^-(x)}&\lesssim  k^2e^{-(2t-1)^2 k\ga\bar{\si}}\flay\quad\forall x\in \Om_{i,t}^+.
		\end{align*} 
	\end{lemma}
	\begin{proof} The proof is similar to that of \cite[Lemma 3.5]{stddm} for the 2D case.
By \eqref{u1ia} and \eqref{barui}, we have
\begin{align}\label{barui+}
			\bar{u}_{i}^+(x)=\int_{\Om_{(-\infty, \xi_{i+2})}}f(y)J(y)G(\tilde{x}, \tilde{y})dy,\quad x \in \Om_{i+1}.
		\end{align}  
		Now, by \eqref{A2} and Lemma~\ref{lem-imrho} with $b_0=\min_{j=1,2,3}(l_j+t d_j-\bar l_j)=\min_{j=1,2,3}(L_j-\bar l_j-(1-t) d_j)$ and $a_j=2\bar l_j$, we have 
\eq{\label{imrho0}
\im\rho(\tilde{x}, \tilde{y})&\geq\frac{\si_0b_0^2}{\big(\Sigma_{j=1}^{3}(2\bar l_j+b_0)^2\big)^\frac12}\ge (2t-1)^2\ga\bar{\si},\quad\text{for } x \in \Om_{i+1}^+, y \in B_l.}  
Thus a similar argument as in \eqref{eq_tu} and using \eqref{A2} indicates that for $x \in \Om_{i+1}^+, y \in B_l$, we have $\abs{\bar{u}_{i}^+}\lesssim  k e^{- (2t-1)^2 k\ga\bar{\si}}\flay$. The proofs of the estimates for $\abs{\nabla \bar{u}_{i}^+}$, $\abs{\bar{u}_{i}^-}$, $\abs{\nabla \bar{u}_{i}^-}$ are similar. This completes the proof.
	\end{proof}

	\begin{lemma}\label{lem4}
		Under the condition of  \eqref{A2},  we have
		\begin{align*}
			\normE{\bar{u}_{i}^\pm}_{\frac12,\pa\Omega_i^\pml}\lesssim k^3 e^{-\frac14k\ga\bar\si}\flay, \quad i=1, \cdots, N-1, 
		\end{align*}
	\end{lemma}
	
	\begin{proof}
		We only give the proof of $\normE{\bar{u}_{i}^+}_{\frac12,\pa\Omega_i^\pml}$,  since the proof of $\normE{\bar{u}_{i}^-}_{\frac12,\pa\Omega_i^\pml}$ is similar. By \eqref{leq_vga} we have
		\begin{align}\label{eq_barupm}
			\normE{\bar{u}_{i}^\pm}_{\frac12,\pa\Omega_i^\pml}\lesssim  \max_{x\in \pa \Om_{i}^\pml } \big(k^\frac12\abs{\bar{u}_{i}^+(x)}+\abs{\nabla\bar{u}_{i}^+(x)} \big), \quad i=1, \cdots, N-1, 
		\end{align}
		Since $\bar{u}_{i}^+=u_{i}^+$ in $\Om_i\cup\Om_{i+1}$,  by Lemma \ref{decayex} ,  and $\abs{\bar{u}_{i}^+}$, $\abs{\nabla \bar{u}_{i}^+}$ decay exponentially
		in $\Om_i\cup\Om_{i+1}$ as $\abs{y} \rightarrow \infty$ . Thus by using Algorithm~\ref{alg1},  we know that
		\begin{align}\label{bu_i+}
			\bar{u}_{i}^+(x)&=\int_{\Omega_i\cup\Omega_{i+1}}(\bar{f}_{i}^+(y)+f_{i+1}(y))J_i(y)G(\tilde{x}^i, \tilde{y}^i)dy\notag\\
			&=\int_{\Omega_i\cup\Omega_{i+1}}\big(f_{i}(y)+f_{i+1}(y)\big)J_i(y)G(\tilde{x}^i, \tilde{y}^i)dy \notag\\
			&\quad+ \int_{\Omega_i}\Psi_{i}^+(\bar{f}_{i-1}^+)(y)J_i(y)G(\tilde{x}^i, \tilde{y}^i)dy\notag\\
			&=\int_{\Omega_i\cup\Omega_{i+1}}\big((1-\be_i^+)f_{i}(y)+f_{i+1}(y)\big)J_i(y)G(\tilde{x}^i, \tilde{y}^i)dy \\
&\quad+\int_{\Omega_i}\na_y\cdot\left( A \na_y\be_{i}^+\right)\bar{u}_{i-1}^+(y)G(\tilde{x}^i, \tilde{y}^i)dy\notag\\
			&\quad+2\int_{\Omega_i}A {\na_y\be_{i}^+}\cdot \na_y u_{i-1}^+(y)G(\tilde{x}^i, \tilde{y}^i)dy\notag\\
			&=: \roma+\romb+\romc\notag
		\end{align}
		
		For $\roma$, denote by $\Om_{i, i+1}=(\xi_i, \xi_{i+2})\times(-l_2, l_2)\times(-l_3, l_3)$. From Lemma~\ref{lemma_G} and \eqref{egamma} we have
		\begin{align}\label{bu_i+I}
			\abs{\roma}&\lesssim \norme{f_{i}+f_{i+1}}{\Om_{i, i+1}}'\norme{J_{i}(\cdot)G(\tilde{x}^i, \tilde{\cdot}^i)}{\Omega_{i, i+1}}\notag\\
			&\lesssim\flay\max_{y\in \Om_{i, i+1}}\big(k\abs{G(\tilde{x}^i, \tilde{y}^i)}+\abs{\nabla G(\tilde{x}^i, \tilde{y}^i)}\big)\notag\\
			&\lesssim\flay\max_{y\in \Om_{i, i+1}}\big(k\abs{\tilde{x}^i-\tilde{y}^i}^{-1}+\abs{\tilde{x}^i-\tilde{y}^i}^{-2}\big)e^{-k\im\rho(\tilde{x}^i, \tilde{y}^i))}\notag\\
			&\lesssim ke^{-k\ga\bar\si}\flay.
		\end{align} 

Next we estimate $\romb$.
		\begin{align*}
			\abs{\romb}&\lesssim \int_{\Omega_{i,t}^-}\abs{\bar{u}_{i-1}^+(y)}\abs{G(\tilde{x}^i, \tilde{y}^i)}dy + \int_{\Omega_{i,t}^+}\abs{\bar{u}_{i-1}^+(y)}\abs{G(\tilde{x}^i, \tilde{y}^i)}dy \\
			&:=\romb_1+\romb_2,
		\end{align*}
where $\Om_{i,t}^\pm$ are defined in \eqref{Omi+-}. From Lemma~\ref{rstab} and \eqref{barui+}, we have $\norm{\bar{u}_{i-1}^+(y)}_{L^2(\Om_{i,t}^-)}$ $\ls k\cs\flay$. Moreover,   by Lemma~\ref{lem-imrho} with $b_0=\min_{j=1,2,3}(L_j-l_j-t d_j)=\min_{j=1,2,3}((1-t)d_j)$ and $a_j=2\bar l_j$  we have 
\eq{\label{imrho1}\im\rho(\tilde{x}, \tilde{y})\geq \frac{\si_0b_0^2}{\big(\Sigma_{j=1}^{3}(2\bar l_j+b_0)^2\big)^\frac12}\ge (2-2t)^2\ga\bar{\si},\quad\text{for } x \in \pa\Om_{i}^\pml,  y \in \Om_{i,t}^- .} 
So, from Lemmas~\ref{rstab} and \ref{lem:csta}, we have
		\begin{align}\label{bu_i+II1}
			\abs{\romb_1}&\lesssim \norm{\bar{u}_{i-1}^+(y)}_{L^2(\Om_{i,t}^-)} \bigg(\int_{\Om_{i,t}^-}\abs{ G(\tilde{x}^i, \tilde{y}^i)}^2dy\bigg)^\frac12\notag\\
			&\lesssim  k\cs\flay\max_{y\in \Om_{i,t}^-}\abs{\tilde{x}^i-\tilde{y}^i}^{-1}e^{-k\im\rho(\tilde{x}^i, \tilde{y}^i))}\notag\\
			&\lesssim k e^{-(2-2t)^2 k\ga\bar\si}\flay.		
		\end{align}
		By Lemma \ref{lem3_}, Lemma \ref{Bound_G} , for $x \in \pa\Om_{i}^\pml $
		\begin{align}\label{II2}
			\abs{\romb_2}&\lesssim \max_{y\in \Om_{i,t}^+}\abs{\bar{u}_{i-1}^+(y)} \int_{\Om_{i,t}^+}\abs{G(\tilde{x}^i, \tilde{y}^i)}dy\notag\\
			&\lesssim ke^{-(2t-1)^2 k\ga\bar\si}\flay
		\end{align}
		So, taking $t=\frac34$, for $x \in \pa\Om_{i}^\pml $ we have $\abs{\romb}\lesssim k e^{-\frac14 k\ga\bar\si}\flay$.
		By a similar argument, we have $\abs{\romc}\lesssim k^2e^{-\frac14 k\ga\bar\si}\flay$. 
		
		The above estimates of $\roma$ -- $\romc$ show that
 $\abs{\bar{u}_{i}^+(x)}\lesssim k^2e^{-\frac14 k\ga\bar\si}\flay$ for $x \in \pa\Om_{i}^\pml $.
		Similarly,  we have $\abs{\nabla\bar{u}_{i}^+(x)}\lesssim k^3e^{-\frac14 k\ga\bar\si}\flay$ for $x\in \pa\Om_{i}^\pml $.
		This completes the proof by using \eqref{eq_barupm}.
	\end{proof}

	\begin{lemma}\label{lem6}
		Under the condition of \eqref{A2}, for sufficiently large $k\si_0d$,  $i=1, \cdots, N-1$,  we have
		\begin{align*}
			\norme{\bar{u}_{i}^+-\hat{u}_i^+}{\Omega_i^\pml} & \lesssim (k\mu_0)^{-i}k^4 e^{-\frac14k\ga\bar{\si}} \flay,  \\
			\norme{\bar{u}_{i}^--\hat{u}_i^-}{\Omega_i^\pml} & \lesssim (k\mu_0)^{-(N-i)}k^4 e^{-\frac14k\ga\bar{\si}} \flay.
		\end{align*}
	\end{lemma}
	
	\begin{proof} We only prove (i) since the proof of (ii) is similar. 
From Algorithms~\ref{alg1} and \ref{alg2}, we have
		\begin{align*}
	&\normef{\bar{f}_{i+1}^+-\hat{f}_{i+1}^+}{\Omega_{i+1}^\pml}=\normE{ \Psi_{i+1}^+(\bar{f}_{i}^+)-\hat\Psi_{i+1}^+(\hat{f}_{i}^+) }_{\Omega_{i+1}^\pml}'\lesssim k^{-1}	\norme{\bar{u}_i^+-\hat{u}_i^+}{\Omega_i^\pml}.
		\end{align*}	
Moreover, from Algorithm \ref{alg2} and \eqref{pmlalg1}--\eqref{pmlalg1F} we have
\eq{\label{eq:barfhatf}
-\na\cdot  (A_i\na(\bar{u}_i^+- \hat{u}_i^+)) - k^2J_i(\bar{u}_i^+- \hat{u}_i^+)  &= J_i(\bar{f}_{i}^+- \hat{f}_{i}^+), \text{ in }\Omega_i^\pml,\\
\bar{u}_{i}^+-\hat{u}_{i}^+ &= \bar{u}_{i}^+,\quad\text{ on } \pa \Omega_i^\pml.}
By combining the inf-sup condition \eqref{localinfsup}, the lifting result \cite[Lemma 4.22]{melenk2011}, and Lemma \ref{lem4},  we have
		\begin{align}\label{bu_i-hu_i}
			 \norme{\bar{u}_{i}^+-\hat{u}_{i}^+}{\Omega_i^\pml}&\lesssim \mu_0^{-1} \big(\normef{\bar{f}_{i}^+-\hat{f}_{i}^+}{\Omega_i^\pml} + \normE{\bar{u}_{i}}_{\frac12,\pa\Omega_{i}^\pml}\big)\\
			& \lesssim \mu_0^{-1} \normef{\bar{f}_{i}^+-\hat{f}_{i}^+}{\Omega_i^\pml} + \mu_0^{-1}k^3 e^{-\frac14k\ga\bar\si}\flay,\notag
		\end{align}
which together with \eqref{eq:barfhatf}	gives
\eqn{\norme{\bar{u}_{i}^+-\hat{u}_{i}^+}{\Omega_i^\pml}& \lesssim (k\mu_0)^{-1} \norme{\bar{u}_{i-1}^+-\hat{u}_{i-1}^+}{\Omega_{i-1}^\pml} + \mu_0^{-1}k^3 e^{-\frac14k\ga\bar\si}\flay,}
where $\bar{u}_0^+-\hat{u}_0^+=0$.	By the induction argument we have
		\begin{align*}
			\norme{\bar{u}_{i}^+-\hat{u}_{i}^+}{\Omega_i^\pml}\lesssim (k\mu_0)^{-i}k^4 e^{-\frac14k\ga\bar{\si}}\flay.	
		\end{align*}
		This completes the proof of the lemma.
	\end{proof}	

	\begin{theorem}\label{th4}
		Under the conditions  of \eqref{A1} and \eqref{A2}, for sufficiently large $k\si_0d$,  we have
		\begin{align}\label{eq4.1}
			\norme{\tilde{u}-\hat{u}_\p}{B_L}\lesssim (k\cs)^{N-1}k^4 e^{-\frac14k\ga\bar{\si}} \flay, 
		\end{align}
		where $\tilde u$ is the non-truncated PML solution to \eqref{eq_exactsol} and $\hat{u}_\p$ is the layer-wise PSTDDM solution given by  Algorithm~\ref{alg2}, $\ga$ and $\bar\si$ are defined in \eqref{egamma}.
	\end{theorem}
	\begin{proof} By \eqref{eup}, 
	\(  \tilde{u} - \hat{u}_\p \in H^1(B_L) \) is the solution to the following elliptic problem: $\tilde{u} - \hat{u}_\p =  \tilde{u}\;  \mathrm{ on }\ \partial B_L,$
	\begin{align*}
&(\nabla (	\tilde{u} - \hat{u}_\p ), \nabla v) + k^2((	\tilde{u} - \hat{u}_\p ), v)\\
& = \sum_{i = 1}^{N} \big((\nabla (\tilde{u} - \hat{v}), \nabla v)_{\Omega_i \cap B_L} + k^2(\tilde{u} - \hat{v}, v)_{\Omega_i \cap B_L}\big) \quad \forall v \in H_0^1(B_L).
	\end{align*}
By \cite[Lemma 4.22]{melenk2011}, there exists a lifting $\tilde w\in H^1(B_L)$ of $\tilde u|_{\partial B_L}$ satisfying $\normE{\tilde w}\ls \normE{\tilde u}_{\frac12,\pa B_L}$.
	Therefore, by taking $v=\tilde{u} - \hat{u}_\p-\tilde w$, we have
	\begin{align}
			\norme{\tilde{u} - \hat{u}_\p}{B_L} \ls & \bigg(  \sum_{i = 1}^{N} \norme{\tilde{u}-\hat{v}}{\Omega_i \cap	B_L}^2 \bigg)^{1/2} + \normE{\tilde{u}}_{\frac12,\partial B_L}. \label{eqxxxx1}
	\end{align}
Combining with Theorem \ref{th3} and \eqref{barui},  we have $\tilde{u}(x)=-(\bar{u}_{i-1}^+(x)+\bar{u}_{i}^-(x))$ for any $x\in\Omega_i, i=1, \cdots, N$,  where $\bar{u}_{0}^+=\bar{u}_{N}^-=0$ in $\R^3$. Then, by using Lemma \ref{lem6},  we can obtain
	\begin{align}
		\bigg( \sum_{i = 1}^{N} \norme{\tilde{u}-\hat{v}}{\Omega_i \cap	B_L}^2 \bigg)^{1/2} \lesssim (k\cs)^{N-1}k^4 e^{-\frac14k\ga\bar{\si}}\flay. \label{eqxxxx2}
	\end{align}
 which, together with	\eqref{eqxxxx1}--\eqref{eqxxxx2} and Lemma~\ref{lemma_tildeuBC} completes the proof.
	\end{proof}

	\begin{remark}\label{rem2.3} {\rm (i)}
		Theorem~\ref{th4} shows that the solution $\hat{u}_\p$ obtained by our PSTDDM provides a good approximation to the PML solution $\tilde u$ without truncation, and the associated error decays exponentially as $k\si_0d$ increases.
		
	
	{\rm (ii)} The hidden constant in \eqref{eq4.1} generally  depends polynomially on the PML medium parameter $\si_0$ and the number of layers $N$, due to the inf--sup condition and the induction argument used in the proofs.	

{\rm (iii)} If the stability constant $\cs\ls k^{-1}$, the error estimate becomes $\norme{\tilde{u}-\hat{u}_\p}{B_L}\lesssim k^4 e^{-\frac14k\ga\bar{\si}}\flay$.
\end{remark}	
	
	\section{Source transfer bar by bar} \label{sec3} 
	Recall that the PSTDDM in Algorithms~\ref{alg1} and \ref{alg2} are layer-wise in which the original PML problems are decomposed into small PML problems by doing source transfers layer by layer from both positive and negative $x_1$-directions. Obviously,  those small PML problems on strips may be solved by the same PSTDDM recursively but by doing source transfers along two $x_2$-directions,  and as a results,  they are decomposed further into small PML problems on bars. In this section we introduce such bar-wise PSTDDM by doing source transfers bar by bar (PSTDDMbar),  which further reduce the size of the subproblems and the computational costs to solve them. 
	\subsection{The PSTDDMbar for the PML problem in $\R^3$}
	We only show the PSTDDM for solving the local PML equations in Algorithm~\ref{alg1},  which is described in Algorithm~\ref{alg3} below. The PSTDDMbar for the original PML problem \eqref{eq_exactsol} is a combination of Algorithm~\ref{alg1} and Algorithm~\ref{alg3}. 
	We first introduce some notation. 
	By introducing $\eta_j=-l_2+(j-1)\De\eta$ where $\De\eta=2l_2/N$ and $1\le j\le N+1$, we divide the whole space into layers:
	\begin{align*}
		\Omega^1 & = \{x=(x_1, x_2, x_3)\in\R^3:x_2<\eta_{2}\},     \\
		\Omega^j & = \{x=(x_1, x_2, x_3)\in\R^3:\eta_j<x_2<\eta_{j+1}\}, \ j=2, \cdots, N-1,  \\
		\Omega^N & = \{x=(x_1, x_2, x_3)\in\R^3:\eta_N<x_2\}.
	\end{align*}
The intersections $\Omega_i\cap\Omega^j$ induce a  decomposition of whole space into bars, denoted by $\D_{\xi\eta}:=\set{\Omega_i\cap\Omega^j:\;i,j=1,\cdots,N}$.  Similar to \eqref{normDD1}, we introduce the following negative norm on $\D_{\xi\eta}$:
\eq{\label{normDD2}
\normE{f}'_{\D_{\xi\eta}}:=\Big(\sum_{i,j=1}^N\big(\normE{f}'_{\Omega_i\cap\Omega^j}\big)^2\Big)^\frac12.
}
It is clear that $\normE{f}'_{\D_{\xi}}\le\normE{f}'_{\D_{\xi\eta}}$
	
	Although $u_i^{\pm}$ in Algorithm~\ref{alg1} is defined in $\R^3$,  only its restriction in $\Omega_i\cup\Omega_{i+1}$ is useful (see also Theorem~\ref{th3}). 
	Since $\bar{u}_i^{\pm}$ defined in \eqref{pmlalg1} equals to $u_i^{\pm}$ in $\Omega_i\cup\Omega_{i+1}$,  the local PML problems in Algorithm~\ref{alg1} (see \eqref{eqs1a}--\eqref{eqs2b}) may be replaced correspondingly by \eqref{pmlalg1}.
	Let 
	\eq{\label{fij}
	f_{i, j}^\pm=F_i^\pm \text{ in } \Omega^j \text{ and } f_{i, j}^\pm=0 \text{ elsewhere for } j=1, \cdots, N,} 
where $F_i^\pm$ are defined in \eqref{pmlalg1F}. We take smooth functions ${\ga_j^+}(x_2)$ and ${\ga_j^-}(x_2)$ satisfying
    \begin{align}\label{gammapm}
\begin{cases}{\ga_j^+}=1, &\mathrm{as}\ x_2\leq\eta_j,   \\ {\ga_j^+}=0,  &\mathrm{as}\ x_2\geq\eta_{j+1},  \end{cases}
		\quad \big|{\ga_j^+}'\big| \lesssim (\De\eta)^{-1},\quad \ga_j^-=1-\ga_j^+.
	\end{align}
	Similar to Algorithm~\ref{alg1},  we have the PSTDDM as shown below in Algorithm~\ref{alg3} for solving \eqref{pmlalg1}.
	\begin{algorithm}
		\caption{Source Transfers for the $i^{th}$ local PML problems}
		\label{alg3}
		\begin{multicols}{2}			
			\textbf{Step 1.}
			
			\hrulefill
			
			1. Let ${\bar{f}}_{i, 1}^{\pm+}={f}_{i, 1}^\pm$;
			
			2. While $j=1, \cdots, N-2$ do
			
			\qquad$\bullet$ Find ${u}_{i, j}^{\pm+}\in H^1(\R^3)$, 
			\begin{align}\label{3.3}
				J_i^{-1}\na\cdot & (A_i\na{u}_{i, j}^{\pm+}) + k^2{u}_{i, j}^{\pm+} \\
				& =- {\bar{f}}_{i, j}^{\pm+} - {f}_{i, j+1}^\pm, \notag
			\end{align}
			\qquad$\bullet$ Compute ${\Psi}_{i, j+1}^{\pm+}({\bar{f}}_{i, j}^{\pm+})$ in $\Omega^{j+1}$, 
			\begin{align*}
				{\Psi}_{i, j+1}^{\pm+}({\bar{f}}_{i, j}^{\pm+}) =  J_i^{-1}\na\cdot\big(A_i\na \ga_{j+1}^+u_{i, j}^{\pm+}\big)\\+J_i^{-1}A_i\na \ga_{j+1}^+\cdot\na u_{i, j}^{\pm+}-\ga_{j+1}^+{f}_{i, j+1}^\pm . \end{align*} 
			\qquad$\bullet$ Set
			\begin{align*}
				{\bar{f}}_{i, j+1}^{\pm+}=\begin{cases}
					{f}_{i, j+1}^\pm+{\Psi}_{i, j+1}^{\pm+}({\bar{f}}_{i, j}^{\pm+}) \text{ in } \Omega^{j+1},  \\
					0 \qquad\qquad\text{ elsewhere. }
				\end{cases}
			\end{align*}
			\quad End while
			
			3. Find ${u}_{i, N-1}^{\pm+}\in H^1(\R^3)$, 
			\begin{align}
				J_i^{-1} \na\cdot & (A_i\na {u}_{i, N-1}^{\pm+}) + k^2{u}_{i, N-1}^{\pm+} \\
				& = - {\bar{f}}_{i, N-1}^{\pm+} - {f}_{i, N}^\pm.\notag
			\end{align}
			
			\columnbreak			
			\textbf{Step 2.}
			
			\hrulefill
			
			1. ${\bar{f}}_{i, N}^{\pm-}={f}_{i, N}^\pm$;
			
			2. While $j=N-1, \cdots, 2$, 
			
			\qquad$\bullet$ Find ${u}_{i, j}^{\pm-}\in H^1(\R^3)$, 
			\begin{align}
				J_i^{-1} \na\cdot & (A_i\na{u}_{i, j}^{\pm-}) + k^2{u}_{i, j}^{\pm-} \\
				& =-{\bar{f}}_{i, j+1}^{\pm-}, \notag
			\end{align}
			\qquad$\bullet$ Compute ${\Psi}_{i, j}^{\pm-}({\bar{f}}_{i, j+1}^{\pm-})$ in $\Omega^{j}$, 
			\begin{align*}
				{\Psi}_{i, j}^{\pm-}({\bar{f}}_{i, j+1}^{\pm-}) =  J_i^{-1}\na\cdot\big(A_i\na \ga_{j}^-u_{i, j}^{\pm-}\big) \\
				 + J_i^{-1}A_i\na \ga_{j}^-\cdot\na u_{i, j}^{\pm-} . 
			\end{align*}
			\qquad$\bullet$ Set
			\begin{align*}
				{\bar{f}}_{i, j}^{\pm-}=\begin{cases}
					{f}_{i, j}^{\pm} + {\Psi}_{i, j}^{\pm-}({\bar{f}}_{i, j+1}^{\pm-}) \text{ in } \Omega^{j},  \\
					0 \qquad\qquad \text{ elsewhere. }
				\end{cases}
			\end{align*}
			\quad End while
			
			3. Find ${u}_{i, 1}^{\pm-}\in H^1(\R^3)$, 
			\begin{align}\label{3.6}
				J_i^{-1}\na\cdot & (A_i\na u_{i, 1}^{\pm-}) + k^2(u_{i, 1}^{\pm-}) \\
				& =-{\bar{f}}_{i, 2}^{\pm-} - {f}_{i, 1}^\pm.\notag
			\end{align}		\end{multicols}
		\hrulefill
				
		\textbf{Step 3.}
				Output $\tilde{u}_i^\pm$:\;  $\tilde{u}_i^\pm=-\big({u}_{i, j-1}^{\pm+}+{u}_{i, j}^{\pm-}\big)$ in $ \Omega^j$ where ${u}_{i, 0}^{\pm+}\equiv{u}_{i, N}^{\pm-}\equiv0$.

	\end{algorithm}
	
	By following the proof of Theorem~\ref{th3},  we may prove the following lemma.
	\begin{lemma}\label{lem8}
		We have
${u}_i^\pm=\tilde u_i^\pm$ in $\Omega^j$ for  $1\le j\le N$,  where ${u}_i^\pm$ is from Algorithm~\ref{alg1} and $\tilde u_i^\pm$ is the output of Algorithm~\ref{alg3}.
	\end{lemma}
	
	By combining Theorem~\ref{th3} and Lemma~\ref{lem8} we have the following theorem which says that the solution $\tilde u$ to the original PML problem \eqref{eq_exactsol} can be expressed by $u_{i, j}^{\pm\pm}$.
	\begin{theorem}\label{th5}
		Define $u_{0, j}^{+\pm}= u_{i, 0}^{\pm+}=u_{N, j}^{-\pm}= u_{i, N}^{\pm-}\equiv0$. Then we have
		\begin{align*}
			\tilde{u}(x)={u}_{i-1, j-1}^{++}(x)+{u}_{i-1, j}^{+-}(x)+{u}_{i, j-1}^{-+}(x)+{u}_{i, j}^{--}(x)\quad\forall x\in \Omega_i\cap\Omega^j, \; 1\le i, j\le N.
		\end{align*}
	\end{theorem}

	\subsection{The PSTDDMbar for the truncated PML problem}
	For $i, j=1, \cdots, N-1$, let $\Omega_{i, j}^{\pml}=(\xi_i-d_1, \xi_{i+2}+d_1)\times(\eta_j-d_2, \eta_{j+2}+d_2)\times(-L_3, L_3)$  and let $\Omega_{i, j}^{\tru}=\Omega_{i, j}^\pml\cap\Omega^j$ and $\Omega_{i, N}^{\tru}=\Omega_{i, N-1}^\pml\cap\Omega^N$.
	
	In order to truncate the problems \eqref{4.3}--\eqref{4.6},  we define the PML complex coordinate stretching $\til{x}^{i, j}=(\til{x}_1^i(x_1), \til{x}_2^j(x_2), \til{x}_3(x_3))$ outside the domain $(\xi_i, \xi_{i+2})\times(\eta_j, \eta_{j+2})\times(-l_3, l_3)$,  where $\til{x}_1^i(x_1)$ is defined in \eqref{tx1} and
	\begin{equation}\label{tx2}
		\tilde{x}_2^j(x_2)=x_2+(\zeta+\ii)\int_{\eta_{j+1}}^{x_2}\si_2^j(t) dt, 
		\si_2^j(t):=
		\begin{cases}
			\si_2(t+\eta_{N+1}-\eta_{j+2}) & \text{if } t>\eta_{j+1},  \\
			\si_2(t-\eta_j+\eta_1)  & \text{if } t\le\eta_{j+1}.
		\end{cases}
	\end{equation}
	Then the PML equation's coefficients are defined as
	\begin{align*}
		&A_{i, j}(x)=\mathrm{diag}\left(\frac{\tilde{x}_2^j(x_2)'\tilde{x}_3(x_3)'}{\tilde{x}_1^i(x_1)'}, \frac{\tilde{x}_1^i(x_1)'\tilde{x}_3(x_3)'}{\tilde{x}_2^j(x_2)'}, \frac{\tilde{x}_1^i(x_1)'\tilde{x}_2^j(x_2)'}{\tilde{x}_3(x_3)'}\right), \\
		&J_{i, j}(x)=\tilde{x}_1^i(x_1)'\tilde{x}_2^j(x_2)'\tilde{x}_3 (x_3)'.
	\end{align*}

	Next we solve the truncated PML problems \eqref{eqs3a}--\eqref{eqs4b} in Algorithm~\ref{alg2} by using the PSTDDM in $x_2$ directions,  that is,  we introduce Algorithm~\ref{alg4} to solve the following type of truncated PML problem: Find $\hat{u}_i^\pm\in H_0^1(\Omega_i^{\pml})$ such that
	\begin{align}
		J_i^{-1}\na\cdot & (A_i\na {\hat u}_i^\pm) + k^2 {\hat u}_i^\pm = \check F_i^\pm, \quad i=1, \cdots, N-1,\label{checku}
	\end{align}
 where $\check F_i^\pm$ are given in Algorithm~\ref{alg5} below.	We remark that the output $\check u_i^\pm$ of Algorithm~\ref{alg4} is an approximation of ${\hat u}_i^\pm$. Set $\check f_{i, j}^\pm=\check F_i^\pm$ in $\Omega^j$ and $ \check f_{i, j}^{\pm}=0$ elsewhere.
	\begin{algorithm}
		\caption{: $\check u_i^\pm=$ PSTDDM$_i(\check F_i^\pm)$.
			Source Transfers for Local Problem \eqref{checku}.
		}
		\label{alg4}
		\begin{multicols}{2}
			
			\textbf{Step 1.}
			
			\hrulefill
			
			1. Let $\hat{f}_{i, 1}^{\pm+}=\check{f}_{i, 1}^{\pm}$;
			
			2. While $j=1, \cdots, N-2$ do
			
			\qquad$\bullet$ Find $\hat{u}_{i, j}^{\pm+}\in H_0^1(\Omega_{i, j}^{\pml})$, 
			\begin{align}\label{eqalg5-1}
				J_{i, j}^{-1}\na\cdot & (A_{i, j}\na \hat{u}_{i, j}^{\pm+}) + k^2\hat{u}_{i, j}^{\pm+} \\
				& =-\hat{f}_{i, j}^{\pm+}-\check{f}_{i, j+1}^{\pm}, \notag
			\end{align}
			\qquad$\bullet$ Compute $\hat{\Psi}_{i, j+1}^{\pm+}(\hat{f}_{i, j}^{\pm+})$ in $\Omega_{i, j+1}^\tru$,  
			\begin{align*}
				\hat{\Psi}_{i, j+1}^{\pm+}(\hat{f}_{i, j}^{\pm+}) =  J_{i, j}^{-1}\na\cdot\big(A_{i, j}\na \ga_{j+1}^+\hat{u}_{i, j}^{\pm+}\big) \\
 +J_{i, j}^{-1}A_{i, j}\na \ga_{j+1}^+\cdot\na\hat{u}_{i, j}^{\pm+}-\ga_{j+1}^+\check{f}_{i, j+1}^{\pm} . 
			\end{align*}
			\qquad$\bullet$ Set
			\begin{align*}
				\hat{f}_{i, j+1}^{\pm+}=\begin{cases}
					\check{f}_{i, j+1}^{\pm}+\hat{\Psi}_{i, j+1}^{\pm+}(\hat{f}_{i, j}^{\pm+}) \text{ in } \Omega_{i, j+1}^\tru,  \\
					0 \qquad\qquad\text{ elsewhere. }
				\end{cases}
			\end{align*}
			\quad End while;
			
			3. Find $\hat{u}_{i, N-1}^{\pm+}\in H_0^1(\Omega_{i, N-1}^{\pml})$, 
			\begin{align}\label{eqalg5-2}
				J_{i, N-1}^{-1}\na\cdot & (A_{i, N-1}\na \hat{u}_{i, N-1}^{\pm+}) + k^2\hat{u}_{i, N-1}^{\pm+} \\
				& =-\hat{f}_{i, N-1}^{\pm+} - \check{f}_{i, N}^{\pm}.\notag
			\end{align}
			
			\columnbreak
			
			\textbf{Step 2.}
			
			\hrulefill
			
			1. Let $\hat{f}_{i, N}^{\pm-}=\check{f}_{i, N}^{\pm}$;
			
			2. While $j=N-1, \cdots, 2$, 
			
			\qquad$\bullet$ Find $\hat{u}_{i, j}^{\pm-}\in H_0^1(\Omega_{i, j}^{\pml})$
			\begin{align}\label{eqalg6-1}
				J_{i, j}^{-1}\na\cdot & (A_{i, j}\na \hat{u}_{i, j}^{\pm-}) + k^2\hat{u}_{i, j}^{\pm-} \\
				& = - \hat{f}_{i, j+1}^{\pm-}, \notag
			\end{align}
			\qquad$\bullet$ Compute $\hat{\Psi}_{i, j}^{\pm-}(\hat{f}_{i, j+1}^{\pm-})$ in $\Omega_{i, j}^{\tru}$,  
			\begin{align*}
				\hat{\Psi}_{i, j}^{\pm-}(\hat{f}_{i, j+1}^{\pm-}) =  J_{i, j}^{-1}\na\cdot(A_{i, j}\na \ga_{j}^-\hat{u}_{i, j}^{\pm-}) \\
				 +J_{i, j}^{-1}A_{i, j}\na \ga_{j}^-\cdot\na\hat{u}_{i, j}^{\pm-} . 
			\end{align*}
			\qquad$\bullet$ Set
			\begin{align*}
				\hat{f}_{i, j}^{\pm-}=\begin{cases}
					\check{f}_{i, j}^{\pm}+\hat\Psi_{i, j}^{\pm-}(\hat{f}_{i, j+1}^{\pm-}) \text{ in } \Omega_{i, j}^{\tru},  \\
					0 \text{ elsewhere. }
				\end{cases}
			\end{align*}
			\quad End while;
			
			3. Find $\hat{u}_{i, 1}^{\pm-}\in H_0^1(\Omega_{i, 1}^{\pml})$, 
			\begin{align}\label{eqalg6-2}
				J_{i, 1}^{-1}\na\cdot & (A_{i, 1}\na \hat{u}_{i, 1}^{\pm-})+k^2\hat{u}_{i, 1}^{\pm-} \\
				& =-\hat{f}_{i, 2}^{\pm-}-\check{f}_{i, 1}^{\pm}.\notag
			\end{align}
		\end{multicols}
		\hrulefill
		
		\textbf{Step 3.}
		
		1. Let  $\hat u_{i, 0}^{\pm+}=\hat u_{i, N}^{\pm-}=0$ and $\check v_i^{\pm} =-(\hat u_{i, j-1}^{\pm+}+\hat u_{i, j}^{\pm-})$ in $ \Omega_{i, j}^\tru$
		
		2. Output $\check u_i^{\pm}\in H_0^1(\Omega_i^{\pml})$ by solving
		\begin{align*}
					(\nabla \check u_i^{\pm}, \nabla v)_{\Omega_i^{\pml}} + k^2 (\check u_i^{\pm}, v)_{\Omega_i^{\pml}} = \sum_{j=1}^{N} \big((\nabla \check v_i^{\pm}, \nabla v)_{\Omega_{i, j}^\tru} + k^2 (\check v_i^{\pm}, v)_{\Omega_{i, j}^\tru}\big), \; \forall v\in H_0^1(\Omega_i^{\pml}).
		\end{align*}
	\end{algorithm}
	
	By inserting Algorithm~\ref{alg4} into Algorithm~\ref{alg2},  we obtain the PSTDDMbar as described in Algorithm~\ref{alg5} for solving the truncated PML problem \eqref{eq_tpml}.
	\begin{algorithm}
		\caption{PSTDDMbar (Source Transfers bar by bar)}
		\label{alg5}
		\begin{multicols}{2}
			
			\textbf{Step 1.}
			
			\hrulefill
			
			1. Let $\check{f}_{1}^+=f_1$;
			
			2. While $i=1, \cdots, N-2$ do
			
			\qquad$\bullet$ Let $\check F_i^+:=- \check{f}_{i}^+ - f_{i+1}$.
			
			\hskip 28pt Call $\check u_i^+=$ PSTDDM$_i(\check F_i^+)$.
			
			\qquad$\bullet$ Compute $\check{\Psi}_{i+1}^+(\check{f}_{i}^+)$ in $\Omega_{i+1}\cap B_L$, 
			\begin{align*}
				\check{\Psi}_{i+1}^+(\check{f}_{i}^+) =  J_i^{-1}\na\cdot\big(A_i\na \be_{i+1}^+\check{u}_i^+\big) \\
+J_i^{-1}A_i\na \be_{i+1}^+\cdot\na\check{u}_i^+ - \be_{i+1}^+f_{i+1}.
			\end{align*}
			\qquad$\bullet$ Set
			\begin{align*}
				\check{f}_{i+1}^+=\begin{cases}
					f_{i+1}+\check{\Psi}_{i+1}^+(\check{f}_{i}^+) \text{ in } \Omega_{i+1}\cap B_L,  \\
					0 \text{ elsewhere. }
				\end{cases}
			\end{align*}
			\quad End while;
			
			3. For $i=N-1$, 
			
			\hskip 28pt Let $\check F_{N-1}^+:=- \check{f}_{N-1}^+ - f_N$.
			
			\hskip 28pt Call $\check u_{N-1}^+=$ PSTDDM$_{N-1}(\check F_{N-1}^+)$.
			\columnbreak
			
			\textbf{Step 2.}
			
			\hrulefill
			
			1. Let $\check{f}_{N}^- =f_N$;
			
			2. While $i=N-1, \cdots, 2$, 
			
			\qquad$\bullet$ Let $\check F_i^-:=- \check{f}_{i+1}^-$.
			
			\hskip 28pt Call $\check u_i^-=$ PSTDDM$_i(\check F_i^-)$.
			
			\qquad$\bullet$ Compute $\check{\Psi}_i^-(\check{f}_{i+1}^-)$ in $\Omega_i\cap B_L$, 
			\begin{align*}
				\check{\Psi}_i^-(\check{f}_{i+1}^-)=  J_i^{-1}\na\cdot(A_i\na \be_i^-\check{u}_i^-) \\
+J_i^{-1}A_i\na \be_i^-\cdot\na\check{u}_i^- .
			\end{align*}
			\qquad$\bullet$ Set
			\begin{align*}
				\check{f}_i^-=\begin{cases}
					f_i+\check\Psi_i^-(\check{f}_{i+1}^-) \text{ in } \Omega_i\cap B_L,  \\
					0 \text{ elsewhere. }
				\end{cases}
			\end{align*}
			\quad End while;
			
			3. For i=1, 
			
			\hskip 28pt Let $\check F_1^-:=- \check{f}_{2}^- - f_1$.
			
			\hskip 28pt Call $\check u_1^-=$ PSTDDM$_1(\check F_1^-)$.
		\end{multicols}
		\hrulefill

			\textbf{Step 3.}
		
		1. Let  $\check{u}_0^+=\check{u}_{N}^-=0$ and $\check{v}(x)=-(\check{u}_{i-1}^++\check{u}_{i}^-)$ in $ \Omega_i\cap B_L$, $i=1, \cdots, N$.
				
		2. Output $\check u \in H_0^1(B_L)$ by solving
\begin{align}\label{ecu}
(\nabla \check u, \nabla v) + k^2(\check u, v) = \sum_{i = 1}^{N} \big((\nabla \check{v}, \nabla v)_{\Omega_i \cap B_L} + k^2(\check{v}, v)_{\Omega_i \cap B_L}\big) \; \forall v \in H_0^1(B_L).
\end{align}
	\end{algorithm}

	For further analysis we first notice from Lemma~\ref{blstab} that the following local inf--sup condition holds if the assumptions \eqref{A1}--\eqref{A2} hold and $k\si_0d$ is sufficiently large.
		\begin{align}\label{lem6e}
			\sup_{\psi\in H_0^1(\Omega_{i, j}^\pml)} \frac{\abs{(A_{i, j}\na\phi, \na\psi)-k^2(J_{i, j}\phi, \psi)}}{ \norme{\psi}{\Oij} }\gtrsim \mu_0 \norme{\phi}{\Oij}\quad \phi\in H_0^1(\Omega_{i, j}^\pml), 
		\end{align}
		where $\mu_0^{-1}=   k^2\cs$.
Similarly, from \eqref{rinfsup}, there also holds the following inf--sup condition for the PML equation in $\R^3$, 
	\begin{align}
		\sup_{\psi\in H^1(\R^3)} \frac{\abs{(A_{i, j}\na\phi, \na\psi)-k^2(J_{i, j}\phi, \psi)}}{\norme{\psi}{\R^3}} \gtrsim \mu_0 \norme{\phi}{\R^3}  \quad \forall \phi\in H^1(\R^3). \label{eq_misir2}
	\end{align}

	Next we introduce the auxiliary functions $\bar{u}_{i, j}^{\pm\pm}\in H^1(\R^3)$,  $i, j=1, 2, \cdots, N-1$ which are solutions to the PML problems:
	\begin{align}\label{pmlalg3}
		J_{i, j}^{-1}\nabla \cdot (A_{i, j}\nabla \bar u_{i, j}^{\pm\pm})+k^2\bar u_{i, j}^{\pm\pm} = F_{i, j}^{\pm\pm}\quad \mathrm{in}\ \R^3,  \quad i=1, \cdots,  N-1, 
	\end{align}
	where
	\begin{align}\label{pmlalg3F}
		\begin{split}
			F_{i, j}^{\pm+}&=-\bar{f}_{i, j}^{\pm+}-f_{i, j+1}^{\pm},  \quad i=1, \cdots,  N-1, \\
			F_{i, j}^{\pm-}&=-\bar{f}_{i, j+1}^{\pm-}, \quad j=2,\ldots,N-1,\ i=1,\ldots,N-1, \quad F_{i, 1}^{\pm-}=-f_{i, 1}^{\pm}-\bar{f}_{i, 2}^{\pm-}.
		\end{split}
	\end{align}
	Clearly,  we have
	\begin{align}\label{baruij}
		\bar u_{i, j}^{\pm\pm}=u_{i, j}^{\pm\pm} \quad\text{in}\quad \Omega^j\cup\Omega^{j+1}
	\end{align}
 and $\hat u_{i, j}^{\pm\pm}$ in Algorithm~\ref{alg4} approximates $\bar u_{i, j}^{\pm\pm}$.
	
%
	
	To proceed, given $t\in (\frac12,1)$, we decompose $\Omega^j$ into two parts $\Omega^j=\Omega_{i,t}^{j+}\cup\Omega_{i,t}^{j-}$:
	\begin{align}\label{O+-}
		\begin{split}
		\Omega_{i,t}^{j+}&:=\Big\{x\in\Omega^j:\;|x_1-\xi_{i+1}|>\frac{2l_1}{N}+t  d_1 \ \text{or}\ |x_3|>l_3+t d_3\Big\},\\ \Omega_{i,t}^{j-}&:=\Omega^j\setminus \Omega_{i,t}^{j+}.
		\end{split}
	\end{align}
	
	The following two lemmas give some estimates of $\bar u_{i, j}^{\pm\pm}$.
	\begin{lemma}\label{lemma_omp}
		Under the conditions of \eqref{A1}--\eqref{A2}, if $k\si_0d$ is sufficiently large,   then we have for $j=1, \cdots, N-1$, 
		\begin{align}
			k\abs{\bar{u}_{i, j}^{\pm+}(x)} +\abs{\na\bar{u}_{i, j}^{\pm+}(x)} & \lesssim k^3 e^{- (t-\frac12)^2 k\ga\bar{\si}}\fbar\quad\forall x\in\Omega_{i,t}^{j+1,+}; \label{eq_omp_bupmp} \\
			k\abs{\bar{u}_{i, j}^{\pm-}(x)}+\abs{\na\bar{u}_{i, j}^{\pm-}(x)} & \lesssim k^3 e^{- (t-\frac12)^2 k\ga\bar{\si}}\fbar\quad\forall x\in\Omega_{i,t}^{j,+}.
		\end{align}
	\end{lemma}
	\begin{proof}
		We show only the proof of \eqref{eq_omp_bupmp} for $\bar{u}_{i, j}^{++}$. From \eqref{pmlalg1}--\eqref{pmlalg1F} and Algorithm~\ref{alg3},  similar to \eqref{barui+} and \eqref{bu_i+}, we have,  for $j=1, \cdots, N-1$, 
		\begin{align} \label{eq_upp}
			\bar{u}_{i, j}^{++}(x) & = - \int_{\cup_{p=1}^{j+1}\Omega^p} f_{i+1}(y)J_{i}(y)G(\tilde{x}^i, \tilde{y}^i)dy - \int_{\cup_{p=1}^{j+1}\Omega^p} \bar{f}_i^+(y)J_{i}(y)G(\tilde{x}^i, \tilde{y}^i)dy\notag \\
			&=\int_{\cup_{p=1}^{j+1}\Omega^p}\big((\be_i^+-1)f_{i}(y)-f_{i+1}(y)\big)J_i(y)G(\tilde{x}^i, \tilde{y}^i)dy\notag \\
			&\quad -\int_{\Omega_i\cap \cup_{p=1}^{j+1}\Omega^p}\na_y\cdot\left( A \na_y\be_{i}^+\right)\bar{u}_{i-1}^+(y)G(\tilde{x}^i, \tilde{y}^i)dy\notag\\
			&\quad-2\int_{\Omega_i\cap \cup_{p=1}^{j+1}\Omega^p}A {\na_y\be_{i}^+}\cdot \na_y u_{i-1}^+(y)G(\tilde{x}^i, \tilde{y}^i)dy\notag\\
			&:=\roma+\romb+\romc
		\end{align}
		
		First, for $x\in\Omega_{i,t}^{j+1,+}$ and $y\in(\Om_{i}\cup \Om_{i+1})\cap B_l$, from Lemma~\ref{lem-imrho}, similar to \eqref{imrho0}, we have  
$\im\rho(\tilde{x}^i, \tilde{y}^i)\geq (2t-1)^2\ga\bar{\si}$.
Therefore,  it follows from Lemma~\ref{lemma_G} and the definitions of the complex coordinate stretching $\tilde{x}^i$ and $\tilde{y}^i$ that,  for $x\in\Omega_{i,t}^{j+1,+}$
\begin{align}\label{eq_uppi}
			\abs{\roma}
			 &\lesssim \normef{f}{(\Om_i\cup\Om_{i+1})\cap B_l}\norme{J_{i}(\cdot)G(\tilde{x}^i, \tilde{\cdot}^i)}{(\Om_i\cup\Om_{i+1})\cap B_l}  \notag \\
			& \lesssim \fbar\max_{y\in(\Om_{i}\cup \Om_{i+1})\cap B_l}\big(k|G(\tilde{x}^i, \tilde{y}^i)| + |\nabla_yG(\tilde{x}^i, \tilde{y}^i)|  \big)\notag    \\
			 &\lesssim k\max_{y\in(\Om_{i}\cup \Om_{i+1})\cap B_l} e^{- k\im\rho(\tilde{x}^i, \tilde{y}^i)}\fbar \notag \\
			& \lesssim ke^{-(2t-1)^2\ga k\bar{\si}} \fbar.
		\end{align}
		
		Next we estimat $\romb$. Given $\frac12<t_1<t$, we have
\eqn{
\romb&\ls \int_{\Omega_i}|\bar{u}_{i-1}^+(y)||G(\tilde{x}^i, \tilde{y}^i)|dy=\int_{\Omega_i}|\bar{u}_{i-1}^+(y)||G(\tilde{x}^i, \tilde{y}^i)|dy\\
&=\int_{\Omega_{i,t_1}^-}|\bar{u}_{i-1}^+(y)||G(\tilde{x}^i, \tilde{y}^i)|dy+\int_{\Omega_{i,t_1}^+}|\bar{u}_{i-1}^+(y)||G(\tilde{x}^i, \tilde{y}^i)|dy\\
&=: \romb_1+II_2.}
 For $x\in\Omega_{i,t}^{j+1,+}$ and $y\in\Om_{i,t_1}^-$, from Lemma~\ref{lem-imrho} with $b_0=\min_{j=1,2,3}(l_j+t d_j$ $-l_j-t_1d_j)=\min_{j=1,2,3}((t-t_1)d_j)$ and $a_j=2(l_j+t_1d_j)$, similar to \eqref{imrho1}, we have  
\eq{\label{imrho2}\im\rho(\tilde{x}^i, \tilde{y}^i)> \frac{\si_0 b_0^2}{\big(\sum_{j=1}^3(2l_j+(t+t_1)d_j)^2\big)^\frac12}\geq (2t-2t_1)^2\ga\bar{\si}.}
So, following the proof of \eqref{bu_i+II1}, we have
\eqn{|II_1|\ls k e^{-(2t-2t_1)^2 k\ga\bar\si}\fbar.}
From Lemma~\ref{lem3_} and \eqref{elem2.5}, we have
\eqn{II_2\ls \max_{y\in \Omega_{i,t_1}^+}|\bar{u}_{i-1}^+(y)|\int_{\Omega_{i,t_1}^+}|G(\tilde{x}^i, \tilde{y}^i)|dy\ls k e^{- (2t_1-1)^2 k\ga\bar{\si}}\fbar.}
Taking $t_1=\frac{t}2+\frac14$ gives
\eq{\label{eq_uppii} \abs{\romb}\ls k e^{- (t-\frac12)^2 k\ga\bar{\si}}\fbar.}

Similarly,
\eq{\label{eq_uppiii} \abs{\romc}\ls k^2 e^{- (t-\frac12)^2 k\ga\bar{\si}}\fbar.}

		Combining \eqref{eq_upp},  \eqref{eq_uppi}, \eqref{eq_uppii}, and \eqref{eq_uppiii}, yields that,  
		\begin{align}\label{eq_bupp}
			\abs{\bar{u}_{i, j}^{++}(x)} \lesssim k^2 e^{- (t-\frac12)^2 k\ga\bar{\si}}\fbar\quad\forall x\in\Omega_{i,t}^{j+1,+}.
		\end{align}
		
		By similar arguments as above we have
		\begin{align}
			\abs{\na\bar{u}_{i, j}^{++}(x)} \lesssim k^3 e^{- (t-\frac12)^2 k\ga\bar{\si}}\fbar\quad\forall x\in\Omega_{i,t}^{j+1,+}. \label{eq_nbupp}
		\end{align}
		Finally,  \eqref{eq_omp_bupmp} follows by combining \eqref{eq_bupp}--\eqref{eq_nbupp}.
	\end{proof}

	\begin{lemma}\label{lem9}
Under the conditions of \eqref{A1}--\eqref{A2}, if $k\si_0d$ is sufficiently large, then for $i, j=1, \cdots, N-1$, we have
		\begin{align*}
			\normE{\bar{u}_{i, j}^{\pm+}}_{\frac12,\pa\Omega_{i, j}^\pml}&\lesssim  k^4 (k\mu_0)^{-j}e^{-\frac19 k\ga\bar\si}   \fbar,\\
			\normE{\bar{u}_{i, j}^{\pm-}}_{\frac12,\pa\Omega_{i, j}^\pml}&\lesssim k^4 (k\mu_0)^{-(N-j)}e^{-\frac19 k\ga\bar\si}  \fbar.
		\end{align*}
	\end{lemma}

	\begin{proof} 
		We only sketch the proof for $\bar u_{i, j}^{++}$. We cannot directly apply Lemma~\ref{lem4} to \eqref{pmlalg3} and Algorithm~\ref{alg3} for solving \eqref{pmlalg1}, since $F_i^+=-\bar{f}_{i}^+-f_{i+1}$ is not compactly supported in $\Omega_i\cup\Omega_{i+1}\cap B_l$, however, we can still follows the proof of Lemma~\ref{lem4}.

From \eqref{leq_vga},  we have
		\begin{align}
			& \normE{\bar{u}_{i, j}^{++}}_{\frac12,\pa\Omega_{i, j}^\pml} \label{eq_uijpml}                   \lesssim  \max_{x\in \pa \Om_{i,j}^\pml } \big(k^\frac12\abs{\bar{u}_{i,j}^{++}(x)}+\abs{\nabla\bar{u}_{i,j}^{++}(x)} \big).
		\end{align}
		
Analogous to \eqref{bu_i+}, using \eqref{pmlalg3}--\eqref{pmlalg3F}, and Algorithm~\ref{alg3}, we may derive that
		\begin{align}\label{3.38.1}
			\bar{u}_{i, j}^{++}(x) & = \int_{\mathbb{R}^3} \big((1-\ga_j^+) f_{i, j}^+(y)+f_{i, j+1}^+(y)\big) J_{i, j}(y) G(\tilde{x}^{i, j}, \tilde{y}^{i, j}) dy     \notag  \\
			& \quad+ \int_{\Omega^j} \na_y\cdot\left( A_i \na_y\ga_j^+\right)\bar{u}_{i, j-1}^{++}(y)  G(\tilde{x}^{i, j}, \tilde{y}^{i, j})dy \notag\\
			& \quad+ 2\int_{\Omega^j} A_i\na_y\ga_j^+\cdot \na_y u_{i, j-1}^{++}(y) G(\tilde{x}^{i, j}, \tilde{y}^{i, j}) dy \notag\\
			& := \roma + \romb + \romc.
		\end{align}
		
First, we estimate $\roma$. Let $\Om_{i,i+1}^{j,j+1}:=(\Om_i\cup \Om_{i+1})\cap(\Om^j\cup \Om^{j+1})\cap B_l$. Denote $\tilde f_i^+=\bar{f}_{i}^+ -(1-\be_i^+)f_i$ in $\Om_i$ and $\tilde f_i^+=0$ otherwise. From \eqref{fij} and \eqref{pmlalg1F}, we have 
\eq{\label{fij+}-f_{i,j}^+ =  (1-\be_i^+)f_i+f_{i+1}+\tilde f_i^+ \text{ in }\Om^j.} Therefore
\eqn{\abs{\roma} &=\bigg|\int_{\Om^j\cup\Om^{j+1}} (1-\ga_j^+)\big((1-\be_i^+)f_i+f_{i+1}+\tilde f_i^+\big) J_{i, j}(y) G(\tilde{x}^{i, j}, \tilde{y}^{i, j}) dy  \bigg|\\
&\lesssim  \norme{f}{\Om_{i,i+1}^{j,j+1}}'\norme{ G(\tilde{x}^{i, j}, \tilde{\cdot}^{i, j})}{\Om_{i,i+1}^{j,j+1}}\\
&\quad+\norme{\tilde f_i^+}{\Om_{i,t}^-\cap(\Om^j\cup \Om^{j+1})}'\norme{G(\tilde{x}^{i, j}, \tilde{\cdot}^{i, j})}{\Om_{i,t}^-\cap(\Om^j\cup \Om^{j+1})} \\
&\quad+ \int_{\Om_{i,t}^+\cap(\Om^j\cup \Om^{j+1})}|\tilde f_i^+(y)||G(\tilde{x}^{i, j}, \tilde{\cdot}^{i, j})| dy\\
&=: \roma_1+\roma_2+\roma_3.}
Similar to \eqref{bu_i+I}, we have			
        \begin{align*}
			\abs{\roma_1} 
			&\lesssim  ke^{-k\ga\bar\si} \fbar.
		\end{align*}
		To estimate $\roma_2$, we estimate $\normef{\tilde{f}_i^+}{ \Omega_i}$. By using Algorithm~\ref{alg1}, we conclude that
\eq{\label{tfi}\tilde{f}_i^+=J^{-1}\na\cdot\left( A \na\be_i^+  u_{i-1}^+\right) + J^{-1}A \na\be_i^+\cdot\na  u_{i-1}^+\quad\text{ in }\Om_i,}
and hence $\normef{\tilde{f}_i^+}{\Om_{i,t}^\pm\cap(\Om^j\cup \Om^{j+1})}  \lesssim  k^{-1}\norme{ u_{i-1}^+}{ \Omega_i}$, which together with
\eqref{u1ia} and \eqref{rinfsup} implies that
\eq{\label{eq_fip}
			\normef{\tilde{f}_i^+}{\Om_{i,t}^\pm\cap(\Om^j\cup \Om^{j+1})}
			& 
			 \lesssim (k\mu_0)^{-1}\fbar.
}
Following the proofs of \eqref{bu_i+I} and \eqref{imrho1}, we conclude that
\eqn{\abs{\roma_2}\ls \normef{\tilde{f}_i^+}{\Om_{i,t}^-\cap(\Om^j\cup \Om^{j+1})}\max_{y\in \Om_{i,t}^-\cap(\Om^j\cup \Om^{j+1})}k e^{-k\im\rho(\tilde{x}^{i,j}, \tilde{y}^{i,j})}\ls  \mu_0^{-1}e^{-(2-2t)^2 k\ga\bar\si} \fbar.}
Next, we estimate $\roma_3$. Using \eqref{tfi} and \eqref{barui}, similar to \eqref{II2}, we obtain
\eqn{
\abs{\roma_3}\ls k^2e^{-(2t-1)^2 k\ga\bar\si}\fbar.}
Therefore, taking $t=\frac34$, combining the estimates for $\roma_1$, $\roma_2$, and $\roma_3$, and using Lemma~\ref{lem:csta}, we have
		\begin{align}\label{eq_pmli} 
			\abs{\roma} 
			&\ls k^2 e^{-\frac14 k\ga\bar\si} \fbar.
		\end{align}

Next we estimate $\romb$. From \eqref{O+-},  it is clear that
		\begin{align*}
			|\romb| &\lesssim   \int_{\Omega_{i,t}^{j-}} \abs{\bar{u}_{i, j-1}^{++}(y)} \abs{ G(\tilde{x}^{i, j}, \tilde{y}^{i, j})} dy + \int_{\Omega_{i,t}^{j+}} \abs{\bar{u}_{i, j-1}^{++}(y)} \abs{ G(\tilde{x}^{i, j}, \tilde{y}^{i, j})} dy\\
			&=: \romb_1+\romb_2.
		\end{align*}
		By \eqref{pmlalg3}--\eqref{pmlalg3F}, the inf--sup condition \eqref{eq_misir2}, Algorithm~\ref{alg3}, and \eqref{eq_fip}, we conclude that
		\begin{align}\label{buij.1}
			\norme{\bar u_{i, j-1}^{++}}{\R^3} &\lesssim  \mu_0^{-1} \normef{\bar f_{i, j-1}^{++} + f_{i, j}^{+}}{\R^3}          \\
			&\lesssim  \mu_0^{-1} \big( \normef{\Psi_{i, j-1}^{++}({\bar f}_{i, j-2}^{++})}{\Omega^{j-1}} + \normef{F_{i}^{+}}{\Omega^{j-1}\cup\Omega^{j}} \big) \notag \\
			&\lesssim  (k\mu_0)^{-1} \norme{\bar u_{i, j-2}^{++}}{\Omega^{j-1}} + \mu_0^{-1}\normef{\bar f_i^++f_{i+1}}{\Omega^{j-1}\cup\Omega^{j}} \notag\\
			&\lesssim  (k\mu_0)^{-1} \norme{\bar u_{i, j-2}^{++}}{\Omega^{j-1}} + k^{-1}\mu_0^{-2}\fbar, \notag
		\end{align}
		where $\bar u_{i, 0}^{++}=0$.  By recursion,  we have
		\begin{align}\label{buij.2}
			\norme{\bar u_{i, j-1}^{++}}{\R^3}&\lesssim  k^{-1}\mu_0^{-2}\fbar \sum_{p=0}^{j-2} (k\mu_0)^{-p} 			\lesssim  k (k\mu_0)^{-j} \fbar.
		\end{align}

Following the proof of \eqref{bu_i+II1} and using \eqref{buij.2}, we can obtain that
\eqn{\romb_1\ls  (k\mu_0)^{-j}e^{-(2-2t)^2 k\ga\bar\si} \fbar.}
		By Lemmas~\ref{lemma_omp} and \ref{Bound_G} we have
		\begin{align*}
			\abs{\romb_2} & \lesssim  \max_{y\in\Omega_{i,t}^{j+}} \abs{\bar{u}_{i, j-1}^{++}(y)} \int_{\Omega_{i,t}^{j+}} \abs{G(\tilde{x}^{i, j}, \tilde{y}^{i, j})} dy   
\lesssim k^2 e^{- (t-\frac12)^2 k\ga\bar{\si}} \fbar.
		\end{align*}
		Therefore, taking $t=\frac56$, we have, for any $x\in\pa\Oij$
		\begin{align}\label{eq_pmlii}
			\abs{\romb} \lesssim k^2 (k\mu_0)^{-j}e^{-\frac19 k\ga\bar\si} \fbar.
		\end{align}				
		By similar arguments,  we have
		\begin{align}\label{eq_pmliii}
			\abs{\romc} \lesssim k^3 (k\mu_0)^{-j}e^{-\frac19 k\ga\bar\si}  \fbar.
		\end{align}
		
		By combining \eqref{3.38.1}--\eqref{eq_pmliii}, we obtain that,  for $x\in\partial\Omega_{i, j}^\pml$, 
		\begin{align}\label{eq_uijabs}
			\abs{\bar{u}_{i, j}^{++}(x)} \lesssim k^3 (k\mu_0)^{-j}e^{-\frac19 k\ga\bar\si}  \fbar.
		\end{align}
		
		By similar arguments,  we have for $x\in\partial\Omega_{i, j}^\pml$, 
		\begin{align}\label{eq_nauijabs}
			\abs{\nabla\bar{u}_{i, j}^{++}(x)} \lesssim k^4 (k\mu_0)^{-j}e^{-\frac19 k\ga\bar\si} \fbar.
		\end{align}
		
		Finally,  we complete the proof by combining \eqref{eq_uijpml},  \eqref{eq_uijabs}--\eqref{eq_nauijabs}.
	\end{proof}

	\begin{lemma}\label{lem7}
		Under the conditions of \eqref{A1}--\eqref{A2}, if $k\si_0d$ is sufficiently large, then for $i=1, 2, \cdots, N-1$,  we have
		\begin{align}\label{eq7a}
			\norme{\check{u}_i^\pm-\bar{u}_i^\pm}{\Oi} \lesssim  k^5 (k\mu_0)^{-N(N-1)}e^{-\frac19 k\ga\bar\si} \fbar.
		\end{align}
	\end{lemma}
	
	\begin{proof}
		We only prove \eqref{eq7a} for $\norme{\check{u}_i^+-\bar{u}_i^+}{\Oi}$. From \eqref{pmlalg1}--\eqref{barui},  Algorithms~\ref{alg1} and \ref{alg5}, we conclude that, for $i, j=1, \cdots, N-1, $
\begin{align}\label{Fi-cFi}
			\normef{F_{i}^+-\check{F}_{i}^+}{\Oij} \lesssim k^{-1} \norme{\bar{u}_{i-1}^+-\check{u}_{i-1}^+}{\Om_i\cap\Oij}\ls k^{-1} \norme{\bar{u}_{i-1}^+-\check{u}_{i-1}^+}{\Omega_{i-1}^\mathrm{PML}},
		\end{align}
		where $\bar{u}_0^+-\check{u}_0^+=0$.
		Similarly, from \eqref{baruij} and Algorithms \ref{alg3} and \ref{alg4}, we have, for $i, j=1, \cdots, N-1, $
		\begin{align*}
			\normef{\bar{f}_{i, j}^{++}-\hat{f}_{i, j}^{++}}{\Oij} &\lesssim k^{-1} \norme{\bar{u}_{i, j-1}^{++}-\hat{u}_{i, j-1}^{++}}{\Om_{i, j-1}^\pml} +\normef{F_{i}^+-\check{F}_{i}^+}{\Oij}  \\
			&\lesssim k^{-1}  \Big(\norme{\bar{u}_{i, j-1}^{++}-\hat{u}_{i, j-1}^{++}}{\Om_{i, j-1}^\pml}  +\norme{\bar{u}_{i-1}^+-\check{u}_{i-1}^+}{\Omega_{i-1}^\mathrm{PML}}\Big), 
		\end{align*}
		where $\bar{u}_{i, 0}^{++}-\hat{u}_{i, 0}^{++}=0$.
		
		Similar to \eqref{bu_i-hu_i}, by combining the local inf--sup condition~\eqref{lem6e},  the lifting result \cite[Lemma 4.22]{melenk2011},  and Lemma~\ref{lem9},  we obtain
		\begin{align}\label{eq7.1} 
			&\norme{\bar{u}_{i, j}^{++}-\hat{u}_{i, j}^{++}}{\Omega_{i, j}^\pml} \\ &\lesssim \mu_0^{-1} \big( \normef{(\bar{f}_{i, j}^{++}+{f}_{i, j+1}^+)-(\hat{f}_{i, j}^{++}+\check{f}_{i, j+1}^+)}{\Oij} + \normE{\bar{u}_{i, j}^{++}}_{\frac12,\pa\Omega_{i, j}^\pml} \notag  \big) \\
			& \lesssim (k\mu_0)^{-1}\big( \norme{\bar{u}_{i, j-1}^{++}-\hat{u}_{i, j-1}^{++}}{\Omega_{i, j-1}^\pml}+\norme{\bar{u}_{i-1}^+-\check{u}_{i-1}^+}{\Omega_{i-1}^\mathrm{PML}}\notag \\
			& \quad +k^5 (k\mu_0)^{-j}e^{-\frac19 k\ga\bar\si}   \fbar \big) \notag
		\end{align}
		for $i, j=1, \cdots, N-1$. By induction on $j$, 
		\begin{align}\label{eq7.2}
			&\norme{\bar{u}_{i, j}^{++}-\hat{u}_{i, j}^{++}}{\Oij}  \\
			&\lesssim (k\mu_0)^{-j} \big(  \norme{\bar{u}_{i-1}^+-\check{u}_{i-1}^+}{\Omega_{i-1}^\mathrm{PML}} +  k^5 (k\mu_0)^{-j}e^{-\frac19 k\ga\bar\si} \fbar  \big).\notag 
%
%
		\end{align}
		
		Similarly,  we have for $j=N-1, \cdots, 1$, 
		\begin{align}\label{eq7.3} 
			 &\norme{\bar{u}_{i, j}^{+-}-\hat{u}_{i, j}^{+-}}{\Oij}\\ 
			 &\lesssim (k\mu_0)^{-(N-j)} \big(  \norme{\bar{u}_{i-1}^+-\check{u}_{i-1}^+}{\Omega_{i-1}^\mathrm{PML}} +  k^5 (k\mu_0)^{-(N-j)}e^{-\frac19 k\ga\bar\si} \fbar  \big).\notag 
%
 	     \end{align}
		
		Combining \eqref{eq7.2},  \eqref{eq7.3}, \eqref{barui}, Algorithm~\ref{alg4},  Lemmas~\ref{lem8} and \ref{lem9},  we conclude that
		\begin{align}
			& 
			\sum_{j=1}^{N} \norme{\bar{u}_{i}^+ - \check v_i^+}{\Omega_{i, j}^\tru} \\
 \lesssim  &\sum_{j=1}^{N-1} \big((k\mu_0)^{-(N-1)} \norme{\bar{u}_{i-1}^+-\check{u}_{i-1}^+}{\Omega_{i-1}^\mathrm{PML}}+  k^5 (k\mu_0)^{-2(N-1)}e^{-\frac19 k\ga\bar\si} \fbar\big)     \notag \\
			\lesssim&\,   (k\mu_0)^{-(N-1)}  \norme{\bar{u}_{i-1}^+-\check{u}_{i-1}^+}{\Omega_{i-1}^\mathrm{PML}} +  k^5 (k\mu_0)^{-2(N-1)}e^{-\frac19 k\ga\bar\si} \fbar.\notag 
		\end{align}
		By similar arguments to Theorem~\ref{th4} and using Lemma~\ref{lem4}, we have
		\begin{align}
			\norme{\bar{u}_{i}^+-\check{u}_{i}^+}{\Omega_{i}^\pml} \lesssim &   (k\mu_0)^{-(N-1)}  \norme{\bar{u}_{i-1}^+-\check{u}_{i-1}^+}{\Omega_{i-1}^\mathrm{PML}}  \\
		& +  k^5 (k\mu_0)^{-2(N-1)}e^{-\frac19 k\ga\bar\si} \fbar.\notag 
		\end{align}
	Then the proof of the lemma is completed by induction.
	\end{proof}

	The following theorem is a direct consequence of Theorem~\ref{th3},  Lemma~\ref{lem7} and the fact that $u_i^\pm=\bar{u}_i^\pm$ in $\Omega_i\cup\Omega_{i+1}$.
	\begin{theorem} \label{th_final}
Under the conditions of \eqref{A1}--\eqref{A2}, if $k\si_0d$ is sufficiently large, then  we have
		\begin{align}\label{eqend1}
			\norme{\check{u}-\til{u}}{B_L} & \lesssim (k\cs)^{N(N-1)} k^5 e^{-\frac19 k\ga\bar\si} \fbar,
		\end{align}
where $\tilde u$ is the non-truncated PML solution to \eqref{eq_exactsol} and $\check{u}$ is the bar-wise PSTDDM solution given by  Algorithm~\ref{alg5}.	
    \end{theorem}
    \begin{remark}\label{rem3.1}
    	{\rm (i)}Theorem~\ref{th_final} shows that the solution $\check u$ from our bar-wise PSTDDM well approximates the PML solution $\tilde u$ without truncation, with the corresponding error decaying exponentially as $k\si_0d$ increases.
    	
    	{\rm (ii)}The implicit constant in \eqref{eqend1} may depend on the PML medium parameter $\si_0$ and the layer number $N$, owing to the local inf–sup estimates and the induction arguments in the two coordinate directions. A polynomial bound for this dependence is not quantified here.
    	
    	{\rm (iii)} If the stability constant $\cs\ls k^{-1}$, then $(k\cs)^{N(N-1)}\ls1$ for fixed $N$, and the error estimate becomes
    	$
    	\norme{\check u-\tilde u}{B_L}\ls k^5e^{-\frac19k\ga\bar\si}\fbar.
    	$
    	\end{remark}
	
	%

	\section{Source transfer block by block} \label{sec4}
	Recall that the PSTDDMbar in Algorithms~\ref{alg3}-- \ref{alg5} are bar-wise in which the original PML problems are decomposed into small PML problems by doing source transfers bar by bar along $x_1$ and $x_2$-directions.
	Further more, those small PML problems may be solved by the same PSTDDM recursively but by doing source transfers from both positive and negative $x_3$-directions,  and as a results,  they are decomposed further into small PML problems on blocks. In this section we introduce such block-wise PSTDDM by doing source transfers block by block (PSTDDMblk),  which further reduce the size of the subproblems and the computational costs to solve them.

	\subsection{The PSTDDMblk for the PML problem in $\R^3$}
	We only show the PSTDDM for solving the local PML equations in Algorithm~\ref{alg3}, which is described in Algorithm~\ref{alg41}
	below. The PSTDDMblk for the original PML problem \eqref{eq_exactsol} is a combination of Algorithm~\ref{alg1},  Algorithm~\ref{alg3} and Algorithm~\ref{alg41}. As last section, we introducing $\tau_s=-l_3+(s-1)\De\tau$ where $\De\tau=2l_3/N$ and $1\le s\le N+1$, we divide the whole space into layers: 
	\begin{align*}
		^1\Omega & = \{x=(x_1, x_2, x_3)\in\R^3:x_3<\tau_{2}\}, \\
		^s\Omega & = \{x=(x_1, x_2, x_3)\in\R^3:\tau_s<x_3<\tau_{s+1}\}, \ s=2, \cdots, N-1,  \\
		^N\Omega & = \{x=(x_1, x_2, x_3)\in\R^3: x_3>\tau_N\}.
	\end{align*}
The intersections $\Omega_i\cap\Omega^j\cap ^s\Omega$ induce a  decomposition of whole space into blocks, denoted by $\D_{\xi\eta\tau}:=\set{\Omega_i\cap\Omega^j\cap ^s\Omega:\;i,j,s=1,\cdots,N}$.  Similar to \eqref{normDD1}, we introduce the following negative norm on $\D_{\xi\eta\tau}$:
\eq{\label{normDD3}
\normE{f}'_{\D_{\xi\eta\tau}}:=\Big(\sum_{i,j,s=1}^N\big(\normE{f}'_{\Omega_i\cap\Omega^j\cap ^s\Omega}\big)^2\Big)^\frac12
}
It is clear that $\normE{f}'_{\D_{\xi}}\le\normE{f}'_{\D_{\xi\eta}}\le\normE{f}'_{\D_{\xi\eta\tau}}$.
	
	Although $\uij$ in Algorithm~\ref{alg3} is defined in $\R^3$,  only its restriction in $(\Omega_i\cup\Omega_{i+1}) \cap (\Omega_j\cup\Omega_{j+1})$ is useful (see Theorem~\ref{th5}). Since $\uijj$ defined in \eqref{pmlalg3} equals to $\uij$ in $(\Omega_i\cup\Omega_{i+1}) \cap (\Omega_j\cup\Omega_{j+1})$,  the local PML problems in Algorithm~\ref{alg3} (see  \eqref{3.3}--\eqref{3.6}) may be replaced by \eqref{pmlalg3}.
%
%
	Let $f_{i, j, s}^{\pm\pm}=F_{i, j}^{\pm\pm}$ in $ ^s\Om$ and $ f_{i, j, s}^{\pm\pm}=0$ elsewhere for $s=1, \cdots, N$. We take smooth functions ${\chi_s^+}(x_3)$ and ${\chi_s^-}(x_3)$ satisfying
	\begin{align}\label{epspm}
		\begin{cases}{\chi_s^+}=1, &\mathrm{as}\ x_3\leq\tau_s,   \\ {\chi_s^+}=0,  &\mathrm{as}\ x_3\geq\tau_{s+1},  \end{cases}
		\quad \big|{\chi_s^+}'\big| \lesssim (\De\tau)^{-1},\quad \chi_s^-=1-\chi_s^+.
	\end{align}	
	Similar to Algorithm~\ref{alg3},  we have the PSTDDM as shown below in Algorithm~\ref{alg41} for solving \eqref{pmlalg3}.
	
	\begin{algorithm}
		\caption{Source Transfers for the $(ij)^{th}$ local PML problems}
		\label{alg41}
		\begin{multicols}{2}			
			\textbf{Step 1.}
			
			\hrulefill
			
			1. Let ${\bar{f}}_{i, j, 1}^{\pm\pm+}={f}_{i, j, 1}^{\pm\pm}$;
			
			2. While $s=1, \cdots, N-2$ do
			
			\quad$\bullet$ Find ${u}_{i, j, s}^{\pm\pm+}\in H^1(\R^3)$, 
			\begin{align}\label{4.3}
				J_{i, j}^{-1}\na\cdot & (A_{i, j}\na{u}_{i, j, s}^{\pm\pm+}) + k^2{u}_{i, j, s}^{\pm\pm+} \\
				& =- {\bar{f}}_{i, j, s}^{\pm\pm+} - {f}_{i, j, s+1}^{\pm\pm}, \notag
			\end{align}
			\quad$\bullet$ Compute ${\Psi}_{i, j, s+1}^{\pm\pm+}({\bar{f}}_{i, j, s}^{\pm\pm+})$ in $^{s+1}\Omega$,
			 \begin{align*}
				\hskip -8pt{\Psi}_{i,j,s+1}^{\pm\pm+}({\bar f}_{i,j,s}^{\pm\pm+})= J_{i,j}^{-1}\na\cdot\big(A_{i,j}
				\na\chi_{s+1}^+\,u_{i,j,s}^{\pm\pm+}\big)\\
				 +J_{i,j}^{-1}A_{i,j}\na\chi_{s+1}^+\cdot\na u_{i,j,s}^{\pm\pm+} -\chi_{s+1}^+f_{i,j,s+1}^{\pm\pm}.
			\end{align*}
			\quad$\bullet$ Set
			\begin{align*}
				\hskip -6pt{\bar{f}}_{i, j, s+1}^{\pm\pm+}\hskip -2pt=\hskip -2pt \begin{cases}
					{f}_{i, j, s+1}^{\pm\pm}
					\hskip -2pt + \hskip -2pt {\Psi}_{i, j, s+1}^{\pm\pm+}({\bar{f}}_{i, j, s}^{\pm\pm+})\\
					\hskip 65pt \text{ in } \hskip -2pt^{s+1}\Omega,  \\
					0 \hskip 50pt\text{ elsewhere. }
				\end{cases}
			\end{align*}
			\quad End while
			
			3. Find ${u}_{i, j, N-1}^{\pm\pm+}\in H^1(\R^3)$, 
			\begin{align}
				\hskip -6ptJ_{i, j}^{-1}\na\cdot & (A_{i, j}\na{u}_{i, j, N-1}^{\pm\pm+}) + k^2{u}_{i, j, N-1}^{\pm\pm+}\notag \\
				& =- {\bar{f}}_{i, j, N-1}^{\pm\pm+} - {f}_{i, j, N}^{\pm\pm}.
			\end{align}
			
			\columnbreak			
			\textbf{Step 2.}
			
			\hrulefill
			
			1. ${\bar{f}}_{i, j, N}^{\pm\pm-}={f}_{i, j, N}^{\pm\pm}$;
			
			2. While $s=N-1, \cdots, 2$, 
			
			\quad$\bullet$ Find ${u}_{i, j, s}^{\pm\pm-}\in H^1(\R^3)$, 
			\begin{align}
				J_{i, j}^{-1}\na\cdot & (A_{i, j}\na{u}_{i, j, s}^{\pm\pm-}) + k^2{u}_{i, j, s}^{\pm\pm-} \\
				& =-{\bar{f}}_{i, j, s+1}^{\pm\pm-}, \notag
			\end{align}
			\quad$\bullet$ Compute ${\Psi}_{i, j, s}^{\pm\pm-}({\bar{f}}_{i, j, s+1}^{\pm\pm-})$ in $^s\Omega$,
			           \begin{align*}
				\hskip -3pt{\Psi}_{i,j,s}^{\pm\pm-}({\bar f}_{i,j,s+1}^{\pm\pm-})= J_{i,j}^{-1}\na\cdot\big(A_{i,j}
				\na\chi_s^-\,u_{i,j,s}^{\pm\pm-}\big)\\
				 +J_{i,j}^{-1}A_{i,j}\na\chi_s^-\cdot\na u_{i,j,s}^{\pm\pm-}.
			\end{align*}	
			\quad$\bullet$ Set
			\begin{align*}
				{\bar{f}}_{i, j, s}^{\pm\pm-}=\begin{cases}
					{f}_{i, j, s}^{\pm\pm} + 	{\Psi}_{i, j, s}^{\pm\pm-}({\bar{f}}_{i, j, s+1}^{\pm\pm-})\\
			\hskip 65pt \text{ in }  ^{s}\Omega,  \\
					0 \hskip 45pt \text{ elsewhere. }
				\end{cases}
			\end{align*}
			\quad End while
			
			3. Find ${u}_{i, j, 1}^{\pm\pm-}\in H^1(\R^3)$, 
			\begin{align}\label{4.6}
				J_{i, j}^{-1}\na\cdot & (A_{i, j}\na{u}_{i, j, 1}^{\pm\pm-}) + k^2{u}_{i, j, 1}^{\pm\pm-} \notag\\
				& =-{\bar{f}}_{i, j, 2}^{\pm\pm-} - {f}_{i, j, 1}^{\pm\pm}.
			\end{align}			
		\end{multicols}
		\hrulefill
		
		\textbf{Step 3.}
		Output $\tilde{u}_{i,j}^{\pm\pm}$:\;  $\tilde{u}_{i,j}^{\pm\pm}=-\big({u}_{i, j,s-1}^{\pm\pm+}+{u}_{i, j,s}^{\pm\pm-}\big)$ in $ ^s\Omega$ where ${u}_{i, j,0}^{\pm\pm+}\equiv{u}_{i,j, N}^{\pm\pm-}\equiv0$.
		
	\end{algorithm}
	Similar to Lemma~\ref{lem8}, by following the proof of Theorem~\ref{th3},  we may prove the following lemma.
	\begin{lemma}\label{lem4.1}		
		We have
		${u}_{i,j}^{\pm\pm}=\tilde u_{i,j}^{\pm\pm}$ in $^s\Omega$ for  $1\le s\le N$,  where ${u}_{i,j}^{\pm\pm}$ is from Algorithm~\ref{alg3} and $\tilde u_{i,j}^{\pm\pm}$ is the output of Algorithm~\ref{alg41}.
	\end{lemma}
	
	By combining Theorem~\ref{th5} and Lemma~\ref{lem4.1} we have the following theorem which says that the solution $\tilde u$ to the original PML problem \eqref{eq_exactsol} can be expressed by $u_{i, j, s}^{\pm\pm\pm}$.
	\begin{theorem}\label{th4.2}
		Define $u_{0, j, s}^{+\pm\pm}= u_{i, 0, s}^{\pm+\pm}=u_{i, j, 0}^{\pm\pm+}=u_{N, j, s}^{-\pm\pm}= u_{i, N, s}^{\pm-\pm}=u_{i, j, N}^{\pm\pm-}\equiv0$. Then for any $x\in \Omega_i\cap \Omega^j\cap\hskip -1pt ^s\Omega,  1\le i, j, s\le N$, we have
		\begin{align*}
			\tilde{u}(x)=-\Big(& {u}_{i-1, j-1, s-1}^{+++}(x)+{u}_{i-1, j-1, s}^{++-}(x)+
			{u}_{i-1, j, s-1}^{+-+}(x)+{u}_{i-1, j, s}^{+--}(x)\\
			&\quad +
			{u}_{i, j-1, s-1}^{-++}(x)+{u}_{i, j-1, s}^{-+-}(x)+
			{u}_{i, j, s-1}^{--+}(x)+{u}_{i, j, s}^{---}(x)\Big).
		\end{align*}
	\end{theorem}
	
	\subsection{The PSTDDMblk for the truncated PML problem}
	 For $i, j, s=1, \cdots, N-1$, let $\Omega_{i, j, s}^{\pml}=(\xi_i-d_1, \xi_{i+2}+d_1)\times(\eta_j-d_2, \eta_{j+2}+d_2)\times(\tau_s-d_3, \tau_{s+2}+d_3)$  and let $\Omega_{i, j, s}^{\tru}=\Omega_{i, j, s}^\pml\cap\hskip -1pt ^s\Omega$ and $\Omega_{i, j, N}^{\tru}=\Omega_{i, j, N-1}^\pml\cap\hskip -1pt ^N\Omega$.

	In order to truncate the problems \eqref{3.3}--\eqref{3.6},  we define the PML complex coordinate stretching $\til{x}^{i, j, s}=(\til{x}_1^i(x_1), \til{x}_2^j(x_2), \til{x}_3^s(x_3))$ outside the domain $(\xi_i, \xi_{i+2})\times(\eta_j, \eta_{j+2})\times(\tau_s, \tau_{s+2})$,  where $\til{x}_1^i(x_1)$, $\til{x}_2^j(x_2)$ is defined in \eqref{tx1}, \eqref{tx2} and
	\begin{equation}\label{tx3}
		\tilde{x}_3^s(x_3)=x_3+(\zeta+\ii)\int_{\tau_{s+1}}^{x_3}\si_3^s(t) dt, 
		\;\si_3^s(t):=
		\begin{cases}
			\si_3(t+\tau_{N+1}-\tau_{s+2}) & \text{if } t>\tau_{s+1},  \\
			\si_3(t-\tau_s+\tau_1)  & \text{if } t\le\tau_{s+1}.
		\end{cases}
	\end{equation}
	Then the PML equation's coefficients are defined as
	\begin{align*}
		&A_{i, j, s}(x)=\mathrm{diag}\left(\frac{\tilde{x}_2^j(x_2)'\tilde{x}_3^s(x_3)'}{\tilde{x}_1^i(x_1)'}, \frac{\tilde{x}_1^i(x_1)'\tilde{x}_3^s(x_3)'}{\tilde{x}_2^j(x_2)'}, \frac{\tilde{x}_1^i(x_1)'\tilde{x}_2^j(x_2)'}{\tilde{x}_3^s(x_3)'}\right), \\
		&J_{i, j,s}(x)=\tilde{x}_1^i(x_1)'\tilde{x}_2^j(x_2)'\tilde{x}_3^s(x_3)'.
	\end{align*}
	
	Next we solve the truncated PML problems \eqref{eqalg5-1}--\eqref{eqalg6-2} in Algorithm~\ref{alg4} by using the PSTDDM in $x_3$ directions,  that is,  we introduce Algorithm~\ref{alg42} to solve the following type of truncated PML problem: Find $\ttu_{ij}^{\pm\pm}\in H_0^1(\Omega_{ij}^{\pml})$ such that
	\begin{align}
		J_{i, j}^{-1}\nabla \cdot (A_{i, j}\nabla \ttu_{ij}^{\pm\pm})+k^2\ttu_{ij}^{\pm\pm} = \ttF_{i, j}^{\pm\pm},  \quad i, j=1, \cdots,  N-1,\label{4.8}
	\end{align}
	where $\ttF_{i, j}^{\pm\pm}$ are given in Algorithm~\ref{alg43} and depend on $\tF_{i}^{\pm}$ provided in Algorithm~\ref{alg44}. In Algorithm~\ref{alg42}, we set $\ttf_{i, j, s}^{\pm\pm}=\ttF_{i, j}^{\pm\pm}$ in $ ^s\Omega$ and $\ttf_{i, j, s}^{\pm\pm}=0$ elsewhere.  We remark that the output $\ttu_{i, j}^{\pm\pm}$ of Algorithm~\ref{alg42} is an approximation of $\hat{u}_{ij}^{\pm\pm}$ in the problems \eqref{eqalg5-1}--\eqref{eqalg6-2}. 
	
	\begin{algorithm}
		\caption{: $\ttu_{i, j}^{\pm\pm}=$ PSTDDM$_{ij}(\ttF_{i, j}^{\pm\pm})$.
			Source Transfers for Local Problem \eqref{4.8}.
		}
		\label{alg42}
		\begin{multicols}{2}
			
			\textbf{Step 1.}
			
			\hrulefill
			
			1. Let $\ff_{i, j, 1}^{\pm\pm+}=\ttf_{i, j, 1}^{\pm\pm}$;
			
			2. While $s=1, \cdots, N-2$ do
			
			\quad$\bullet$ Find $\uu_{i, j, s}^{\pm\pm+}\in H_0^1(\Omega_{i, j, s}^{\pml})$, 
			\begin{align}\label{4.9}
				J_{i, j, s}^{-1}\na\cdot & (A_{i, j, s}\na \uu_{i, j, s}^{\pm\pm+}) + k^2\uu_{i, j, s}^{\pm\pm+} \\
				& =-\ff_{i, j, s}^{\pm\pm+}-\ttf_{i, j, s+1}^{\pm\pm}, \notag
			\end{align}
			\quad$\bullet$ Compute $\PPsi_{i, j, s+1}^{\pm\pm+}(\ff_{i, j, s}^{\pm\pm+})$ in $\Omega_{i,j,s+1}^{\tru}$,
			\begin{align*}
			    &\PPsi_{i,j,s+1}^{\pm\pm+}(\ff_{i,j,s}^{\pm\pm+})\\
				=& J_{i,j,s}^{-1}\na\cdot\big(A_{i,j,s}
				\na\chi_{s+1}^+\,\uu_{i,j,s}^{\pm\pm+}\big)\\
				 +\,&J_{i,j,s}^{-1}A_{i,j,s}\na\chi_{s+1}^+\cdot\na\uu_{i,j,s}^{\pm\pm+} -\chi_{s+1}^+\ttf_{i,j,s+1}^{\pm\pm}.
			\end{align*}
			\quad$\bullet$ Set
			\begin{align*}
				\ff_{i, j, s+1}^{\pm\pm+}\hskip -3pt=\hskip -3pt\begin{cases}
					\ttf_{i, j, s+1}^{\pm\pm} +
					\PPsi_{i, j, s+1}^{\pm\pm+}(\ff_{i, j, s}^{\pm\pm+})\\ \hskip 60pt\text{ in } \Omega_{i, j, s+1}^\tru,  \\
					0 \hskip 60pt\text{ elsewhere. }
				\end{cases}
			\end{align*}
			\quad End while;
			
			3. Find $\uu_{i, j, N-1}^{\pm\pm+}\in H_0^1(\Omega_{i, j, N-1}^{\pml})$, 
			\begin{align}\label{4.10}
				\hskip -5pt J_{i, j, N-1}^{-1}\na\cdot & (A_{i, j, N-1}\na \uu_{i, j, N-1}^{\pm\pm+}) + k^2\uu_{i, j, N-1}^{\pm\pm+}\notag \\
				& =-\ff_{i, j, N-1}^{\pm\pm+}-\ttf_{i, j, N}^{\pm\pm}.
			\end{align}
			
			\columnbreak
			
			\textbf{Step 2.}
			
			\hrulefill
			
			1. Let $\ff_{i, j, N}^{\pm\pm-}=\ttf_{i, j, N}^{\pm\pm}$;
			
			2. While $s=N-1, \cdots, 2$, 
			
			\quad$\bullet$ Find $\uu_{i, j, s}^{\pm\pm-}\in H_0^1(\Omega_{i, j, s}^{\pml})$
			\begin{align}\label{4.11}
				J_{i, j, s}^{-1}\na\cdot & (A_{i, j, s}\na \uu_{i, j, s}^{\pm\pm-}) + k^2\uu_{i, j, s}^{\pm\pm-} \\
				&=-\ff_{i, j, s+1}^{\pm\pm-}, \notag
			\end{align}
			\quad$\bullet$ Compute
			$\PPsi_{i, j, s}^{\pm\pm-}(\ff_{i, j, s+1}^{\pm\pm-})$ in $\Omega_{i,j,s}^{\tru}$,		
			 \begin{align*}
			 	&\PPsi_{i,j,s}^{\pm\pm-}(\ff_{i,j,s+1}^{\pm\pm-})\\
				&= J_{i,j,s}^{-1}\na\cdot\big(A_{i,j,s}
				\na\chi_s^-\,\uu_{i,j,s}^{\pm\pm-}\big)\\
				 &\quad+J_{i,j,s}^{-1}A_{i,j,s}\na\chi_s^-\cdot\na\uu_{i,j,s}^{\pm\pm-}.
			\end{align*}

%
			\quad$\bullet$ Set
			\begin{align*}
				\ff_{i, j, s}^{\pm\pm-}=\begin{cases}
					\ttf_{i, j, s}^{\pm\pm}+\PPsi_{i, j, s}^{\pm\pm-}(\ff_{i, j, s+1}^{\pm\pm-})\\
					\hskip 60pt \text{ in } \Omega_{i, j, s}^\tru,  \\
					0  \hskip 55pt\text{ elsewhere. }
				\end{cases}
			\end{align*}
			\quad End while;
			
			3. Find $\uu_{i, j, 1}^{\pm\pm-}\in H_0^1(\Omega_{i, j, 1}^{\pml})$, 
			\begin{align}\label{4.12}
				J_{i, j, 1}^{-1}\na\cdot & (A_{i, j, 1}\na \uu_{i, j, 1}^{\pm\pm-}) + k^2\uu_{i, j, 1}^{\pm\pm-}\notag \\
				& =-\ff_{i, j, 2}^{\pm\pm-}-\ttf_{i, j, 1}^{\pm\pm}.
			\end{align}
		\end{multicols}
		\hrulefill
		
		\textbf{Step 3.}
		
		1. Let  $\uu_{i, j, 0}^{\pm\pm+}=\uu_{i, j, N}^{\pm\pm-}=0$ and $\vv_{i, j}^{\pm\pm} =-(\uu_{i, j, s-1}^{\pm\pm+}+\uu_{i, j, s}^{\pm\pm-})$ in $ \Omega_{i, j, s}^\tru$
		
		2. Output $\ttu_{i, j}^{\pm\pm} \in H_0^1(\Omega_{ij}^{\pml})$ by solving:  \vskip -20pt
		\begin{align*}
(\nabla \ttu_{i, j}^{\pm\pm}, \nabla v)_{\Omega_{ij}^{\pml}} + k^2 (\ttu_{i, j}^{\pm\pm}, v)_{\Omega_{ij}^{\pml}} = \sum_{s=1}^{N} \big((\nabla \vv_{i, j}^{\pm\pm}, \nabla v)_{\Omega_{i, j, s}^\tru} + k^2 (\vv_{i, j}^{\pm\pm}, v)_{\Omega_{i, j, s}^\tru}\big)\\
					\forall v\in H_0^1(\Omega_{ij}^{\pml}).
		\end{align*}
		\vskip -5pt
	\end{algorithm}
	
	By inserting Algorithm~\ref{alg42} into Algorithm~\ref{alg4},  
	we obtain the Algorithm~\ref{alg43}, whose output $\tu_{i}^{\pm}$ is an approximation of the corresponding output $\hat{u}_{i}^{\pm}$ of  Algorithm~\ref{alg4}. We set $\tf_{i, j}^{\pm}=\tF_{i}^{\pm}$ in $\Omega^j$ and $\tf_{i, j}^{\pm}=0$ elsewhere.
	
	\begin{algorithm}
		\caption{: $\tu_{i}^{\pm}=$ PSTDDM$_{\breve{i}}(\tF_{i}^{\pm})$.
		}	
		\label{alg43}
		\begin{multicols}{2}
			
			\textbf{Step 1.}
			
			\hrulefill
			
			1. Let $\ttf_{i, 1}^{\pm+}=\tf_{i, 1}^{\pm}$;
			
			2. While $j=1, \cdots, N-2$ do
			
			\qquad$\bullet$ Let $\ttF_{i, j}^{\pm+}:=
			-\ttf_{i, j}^{\pm+}-\tf_{i, j+1}^{\pm}$.
			
			\hskip 28pt Call $\ttu_{i, j}^{\pm+}=$ PSTDDM$_{ij}(\ttF_{i, j}^{\pm+})$.
			
			\qquad$\bullet$ Compute $\ttPsi_{i, j+1}^{\pm+}(\ttf_{i, j}^{\pm+}) $ in $\Omega_{i,j+1}^{\tru}$,
			    \begin{align*}
			    	&\ttPsi_{i,j+1}^{\pm+}(\ttf_{i,j}^{\pm+})= J_{i,j}^{-1}\na\cdot\big(A_{i,j}
	\na\ga_{j+1}^+\,\ttu_{i,j}^{\pm+}\big)\\
				&\quad +J_{i,j}^{-1}A_{i,j}\na\ga_{j+1}^+\cdot\na\ttu_{i,j}^{\pm+}-\ga_{j+1}^+\tf_{i,j+1}^{\pm}.
			\end{align*}
			\qquad$\bullet$ Set
			\begin{align*}
				\ttf_{i, j+1}^{\pm+}=\begin{cases}\tf_{i, j+1}^{\pm} + \ttPsi_{i, j+1}^{\pm+}(\ttf_{i, j}^{\pm+}) \text{ in } \Omega_{i, j+1}^\tru,  \\
					0 \text{ elsewhere. }
				\end{cases}
			\end{align*}
			\quad End while;
			
			3.  For $j=N-1$, 
			
			\hskip 28pt Let 
			$\ttF_{i, N-1}^{\pm+}:=
			-\ttf_{i, N-1}^{\pm+}-\tf_{i, N}^{\pm}$.
			
			\hskip 28pt Call $\ttu_{i, N-1}^{\pm+}=$ PSTDDM$_{ij}(\ttF_{i, N-1}^{\pm+})$.
			
			\columnbreak
			
			\textbf{Step 2.}
			
			\hrulefill
			
			1. Let $\ttf_{i, N}^{\pm-}=\tf_{i, N}^{\pm}$;
			
			2. While $j=N-1, \cdots, 2$ do
			
			\qquad$\bullet$ Let $\ttF_{i, j}^{\pm-}:=
			-\ttf_{i, j+1}^{\pm-}$.
			
			\hskip 28pt Call $\ttu_{i, j}^{\pm-}=$ PSTDDM$_{ij}(\ttF_{i, j}^{\pm-})$.
			
			\qquad$\bullet$ Compute $\ttPsi_{i, j}^{\pm-}(\ttf_{i, j+1}^{\pm-})$ in $\Omega_{i,j}^{\tru}$,
			 \begin{align*}
			    \ttPsi_{i,j}^{\pm-}(\ttf_{i,j+1}^{\pm-})=& J_{i,j}^{-1}\na\cdot\big(A_{i,j}
			\na\ga_j^-\,\ttu_{i,j}^{\pm-}\big)\\
				& +J_{i,j}^{-1}A_{i,j}\na\ga_j^-\cdot\na\ttu_{i,j}^{\pm-}.
			\end{align*}			
			\qquad$\bullet$ Set
			\begin{align*}
				\ttf_{i, j}^{\pm-}=\begin{cases}\tf_{i, j}^{\pm} + 	\ttPsi_{i, j}^{\pm-}(\ttf_{i, j+1}^{\pm-})) \text{ in } \Omega_{i, j}^\tru,  \\
					0 \text{ elsewhere. }
				\end{cases}
			\end{align*}
			\quad End while;
			
			3.  For $j=1$, 
			
			\hskip 28pt Let 
			$\ttF_{i, 1}^{\pm-}:=
			-\ttf_{i, 2}^{\pm-}-\tf_{i, 1}^{\pm}$.
			
			\hskip 28pt Call $\ttu_{i, 1}^{\pm-}=$ PSTDDM$_{i1}(\ttF_{i, 1}^{\pm-})$.
			
		\end{multicols}
		\hrulefill
		

		\textbf{Step 3.}
		
		1. Let  $\ttu_{i, 0}^{\pm+}=\ttu_{i, N}^{\pm-}=0$ and $ \tv_i^{\pm} =-(\ttu_{i, j-1}^{\pm+} + \ttu_{i, j}^{\pm-}))$ in $ \Omega_{i, j}^\tru$.
		
2. Output $\tu_i^{\pm}\in H_0^1(\Omega_i^{\pml})$ by solving: \vskip -20pt
		\begin{align*}		(\nabla \tu_i^{\pm}, \nabla v)_{\Omega_i^{\pml}} + k^2 (\tu_i^{\pm}, v)_{\Omega_i^{\pml}} = \sum_{j=1}^{N} \big( (\nabla \tv_i^{\pm}, \nabla v)_{\Omega_{i, j}^\tru} + k^2 (\tv_i^{\pm}, v)_{\Omega_{i, j}^\tru} 
\big)\\ 
\forall   v\in H_0^1(\Omega_i^{\pml}).
		\end{align*}
		\vskip -5pt
	\end{algorithm}
	
	By inserting Algorithm~\ref{alg43} into Algorithm~\ref{alg5}, we obtain the PSTDDMblk as described in Algorithm~\ref{alg44}
	for solving the truncated PML problem \eqref{eq_tpml}.
	
	\begin{algorithm}
		\caption{PSTDDMblk (Source Transfers block by block)}
		\label{alg44}
		\begin{multicols}{2}
			
			\textbf{Step 1.}
			
			\hrulefill
			
			1. Let $\tf_{1}^+=f_1$;
			
			2. While $i=1, \cdots, N-2$ do
			
			\qquad$\bullet$ Let $\tF_i^+:=- \tf_{i}^+ - f_{i+1}$.
			
			\hskip 28pt Call $\tu_i^+=$ PSTDDM$_{\widetilde{i}}(\tF_i^+)$.
			
			\qquad$\bullet$ Compute $\tPsi_{i+1}^+(\tf_{i}^+)$ in $\Omega_{i+1}\cap B_L$,
           \begin{align*}
           	\tPsi_{i+1}^+(\tf_i^+)=
	       J_i^{-1}\na\cdot\big(A_i
	       \na\be_{i+1}^+\,\tu_i^+\big)\\
	        +J_i^{-1}A_i\na\be_{i+1}^+\cdot\na\tu_i^+-\be_{i+1}^+f_{i+1}.
        \end{align*}			
			
			\qquad$\bullet$ Set
			\begin{align*}
				\tf_{i+1}^+=\begin{cases}
					f_{i+1}+\tPsi_{i+1}^+(\tf_{i}^+) \text{ in } \Omega_{i+1}\cap B_L,  \\
					0 \text{ elsewhere. }
				\end{cases}
			\end{align*}
			\quad End while;
			
			3. For $i=N-1$, 
			
			Let $\tF_{N-1}^+:=- \tf_{N-1}^+ - f_{N}$.
			
			\hskip 28pt Call $\tu_{N-1}^+=$ PSTDDM$_{\check{N-1}}(\tF_{N-1}^+)$.
			
			\columnbreak
			
			\textbf{Step 2.}
			
			\hrulefill
			
			1. Let $\tf_{N}^- =f_N$;
			
			2. While $i=N-1, \cdots, 2$, 
			
			\qquad$\bullet$ Let $\tF_i^-:=- \tf_{i+1}^-$.
			
			\hskip 28pt Call $\tu_i^-=$ PSTDDM$_{\widetilde{i}}(\tF_i^-)$.
			
			\qquad$\bullet$ Compute $\tPsi_i^-(\tf_{i+1}^-)$ in $\Omega_i\cap B_L$,
           \begin{align*}
	           \tPsi_i^-(\tf_{i+1}^-)= J_i^{-1}\na\cdot\big(A_i
	           \na\be_i^-\,\tu_i^-\big)\\
	            +J_i^{-1}A_i\na\be_i^-\cdot\na\tu_i^-.
            \end{align*}

			\qquad$\bullet$ Set
			\begin{align*}
				\tf_i^-=\begin{cases}
					f_i+\tPsi_i^-(\tf_{i+1}^-) \text{ in } \Omega_i\cap B_L,  \\
					0 \text{ elsewhere. }
				\end{cases}
			\end{align*}
			\quad End while;
			
			3. For i=1, 
			
			\hskip 28pt Let $\tF_1^-:=- \tf_{2}^- - f_1$.
			
			\hskip 28pt Call $\tu_1^-=$ PSTDDM$_{\check{1}}(\tF_1^-)$.
		\end{multicols}
		
		\hrulefill
		
		\textbf{Step 3.}
		
		1. Let  $\tu_0^+=\tu_{N}^-=0$ and  $\tv(x)=-(\tu_{i-1}^++\tu_i^-)$ in $ \Omega_i\cap B_L$, $i=1, \cdots, N$.
		
		2. Output $\tu \in H_0^1(B_L)$ by solving
		\begin{align}
			(\nabla \tu, \nabla v) + k^2(\tu, v) = \sum_{i = 1}^{N} \big((\nabla \tv, \nabla v)_{\Omega_i \cap B_L} + k^2(\tv, v)_{\Omega_i \cap B_L}\big) \; \forall v \in H_0^1(B_L).
		\end{align}
			
	\end{algorithm}
	
	For further analysis we first notice that,  similar to \eqref{lem6e}, \eqref{eq_misir2}, and Lemma~\ref{blstab}, we have the following inf--sup conditions hold for $i,j,s=1,2,...,N-1$  under the conditions  \eqref{A1}--\eqref{A2}:		\begin{align}
		\sup_{\psi\in H^1(\R^3)} \frac{\abs{(A_{i, j, s}\na\phi, \na\psi)-k^2(J_{i, j, s}\phi, \psi)}}{\norme{\psi}{\R^3}} &\gtrsim \mu_0 \norme{\phi}{\R^3}, \quad \phi\in H_0^1(\R^3), \label{4.17}
		\end{align}
where $\mu_0^{-1}= k^2\cs$, and
	\begin{align}\label{lem43}
			\sup_{\psi\in H_0^1(\Omega_{i, j, s}^\pml)} \frac{\abs{(A_{i, j, s}\na\phi, \na\psi)-k^2(J_{i, j, s}\phi, \psi)}}{ \norme{\psi}{\Oijs} }&\gtrsim  \mu_0 \norme{\phi}{\Oijs}, \quad \phi\in H_0^1(\Omega_{i, j, s}^\pml),		\end{align}
provided further that  $k\si_0d$ is sufficiently large.
	
	Next we introduce the auxiliary functions $\bar{u}_{i, j, s}^{\pm\pm\pm}\in H^1(\R^3)$,  $i, j, s=1, 2, \cdots, N-1$ which are solutions to the PML problems:
	\begin{align}
		J_{i, j, s}^{-1}\na\cdot(A_{i, j, s}\na\bar{u}_{i, j, s}^{\pm\pm+}) + k^2\bar{u}_{i, j, s}^{\pm\pm+} & =- {\bar{f}}_{i, j, s}^{\pm\pm+} - {f}_{i, j, s+1}^{\pm\pm}, \quad s=1, \cdots, N-1, \label{4.18} \\
		J_{i, j, s}^{-1}\na\cdot(A_{i, j, s}\na\bar{u}_{i, j, s}^{\pm\pm-}) + k^2\bar{u}_{i, j, s}^{\pm\pm-} & =- {\bar{f}}_{i, j, s+1}^{\pm\pm-}, \quad s=2, \cdots, N-1,     \\
		J_{i, j, 1}^{-1}\na\cdot(A_{i, j, 1}\na\bar{u}_{i, j, 1}^{\pm\pm-}) + k^2\bar{u}_{i, j, 1}^{\pm\pm-} & = -\bar{f}_{i, j, 2}^{\pm\pm-}  -f_{i,j,1}^{\pm\pm}.\label{4.20}
	\end{align}
	Clearly,  $\bar{u}_{i, j, s}^{\pm\pm\pm}=u_{i, j, s}^{\pm\pm\pm}$ in $ ^s\Omega\cup\hskip -1pt ^{s+1}\Omega$ and $\uu_{i, j, s}^{\pm\pm\pm}$ in Algorithm~\ref{alg42} approximates $\bar{u}_{i, j, s}^{\pm\pm\pm}$.
	
	To proceed, we decompose ${}^s\Omega$ into two parts
	${}^s\Omega={}^s\Omega_i^{j+}\cup{}^s\Omega_i^{j-}$:
	\begin{align}\label{4.21}
		\begin{split}
			{}^s\Omega_i^{j+}&:=\left\{x\in{}^s\Omega:\ |x_1-\xi_{i+1}|>\frac{2l_1}{N}+\frac78 d_1
			\ \text{or}\ |x_2-\eta_{j+1}|>\frac{2l_2}{N}+\frac78 d_2\right\},\\
			{}^s\Omega_i^{j-}&:={}^s\Omega\setminus{}^s\Omega_i^{j+}.
		\end{split}
	\end{align}
	
	%
	%
	The following two lemmas give some estimates of $\bar {u}_{i, j, s}^{\pm\pm\pm}$.
	\begin{lemma}\label{lem45}
		Under the condition of \eqref{A1}--\eqref{A2}, if $k\si_0d$ is sufficiently large, then for $s=1, \cdots, N-1$, we have 
		\begin{align}
			k\abs{\bar u_{i,j,s}^{\pm\pm+}(x)}+
			\abs{\na\bar u_{i,j,s}^{\pm\pm+}(x)}
			&\lesssim k^4(k\mu_0)^{-(N-1)} e^{-\frac1{16}k\ga\bar\si}\fblk,
			\;x\in{}^{s+1}\Omega_i^{j+},\label{4.22}\\
			k\abs{\bar u_{i,j,s}^{\pm\pm-}(x)}+
			\abs{\na\bar u_{i,j,s}^{\pm\pm-}(x)}
			&\lesssim k^4(k\mu_0)^{-(N-1)} e^{-\frac1{16}k\ga\bar\si}\fblk,
			\;x\in{}^{s}\Omega_i^{j+}.
		\end{align}
		%
	\end{lemma}
	\begin{proof}
		We show only the proof of \eqref{4.22} for $\bar{u}_{i, j, s}^{+++}$. Put
		$D_s=\bigcup_{p=1}^{s+1}{}^p\Omega$. From \eqref{pmlalg3}--\eqref{pmlalg3F} and Algorithm~\ref{alg41},  similar to \eqref{3.38.1}, we have,  for $s=1, \cdots, N-1$, 
		\begin{align} \label{uijs123}
			\bar{u}_{i, j, s}^{+++} & = - \int_{D_s} f_{i, j+1}^+(y)J_{i, j}(y)G(\tilde{x}^{i, j}, \tilde{y}^{i, j})dy - \int_{D_s} \bar{f}_{i, j}^{++}(y)J_{i, j}(y)G(\tilde{x}^{i, j}, \tilde{y}^{i, j})dy\notag \\
			&=-\int_{D_s}\big((1-\ga_j^+)f_{i,j}^+(y)+f_{i,j+1}^+(y)\big)
			J_{i,j}(y)G(\tilde x^{i,j},\tilde y^{i,j})\,dy\notag\\
			&\quad-\int_{D_s\cap\Omega^j}\na_y\!\cdot\!\big(A_i\na_y\ga_j^+\big)
			\bar u_{i,j-1}^{++}(y)G(\tilde x^{i,j},\tilde y^{i,j})\,dy\notag\\
			&\quad-2\int_{D_s\cap\Omega^j}A_i\na_y\ga_j^+\!\cdot\!\na_y\bar u_{i,j-1}^{++}(y)
			G(\tilde x^{i,j},\tilde y^{i,j})\,dy\notag\\
			&=:\roma+\romb+\romc .
		\end{align}
		
Similar to \eqref{imrho0}, by applying Lemma~\ref{lem-imrho}, we have, for $x\in{}^{s+1}\Omega_i^{j+}$,
\eqn{\im\rho(\tilde x^{i,j},\tilde y^{i,j})&\ge \frac9{16}\ga\bar\si\quad\forall y\in \Om_{i,i+1}^{j,j+1}=(\Om_i\cup \Om_{i+1})\cap(\Om^j\cup \Om^{j+1})\cap B_l,\\
\im\rho(\tilde x^{i,j},\tilde y^{i,j})&\ge \frac{9}{64}\ga\bar\si\quad\forall y\in \Om_{i,\frac{11}{16}}^-\cap(\Om^j\cup \Om^{j+1}.)			} 
Following the derivation of \eqref{eq_pmli} (with the intermediate steps omitted), we obtain, 
\eq{\label{uijs1}
\abs{\roma}\ls k^2e^{-\frac{9}{64}k\ga\bar\si} \fblk}.
		%

Next we estimate $\romb$. By applying Lemma~\ref{lem-imrho}, we conclude that, for $x\in{}^{s+1}\Omega_i^{j+}$,
		\eqn{
\im\rho(\tilde x^{i,j},\tilde y^{i,j})\geq \frac1{16}\ga\bar{\si}\quad\forall y\in \Omega_{i,\frac34}^{j-}.}
Following the proof of \eqref{eq_pmlii}, we obtain
		\begin{align}\label{uijs2}
			\abs{\romb}
            \lesssim&\int_{D_s\cap\Omega_{i,\frac34}^{j-}}
			\abs{\bar u_{i,j-1}^{++}(y)}\abs{G(\tilde x^{i,j},\tilde y^{i,j})}\,dy+\int_{D_s\cap\Omega_{i,\frac34}^{j+}}
			\abs{\bar u_{i,j-1}^{++}(y)}\abs{G(\tilde x^{i,j},\tilde y^{i,j})}\,dy\notag\\
			\ls& k^2 (k\mu_0)^{-(N-1)}e^{-\frac1{16} k\ga\bar\si} \fblk.
		\end{align}
		
		Similarly,
		\begin{align}\label{uijs3}
			\abs{\romc}\lesssim k^3(k\mu_0)^{-(N-1)} e^{-\frac1{16}k\ga\bar\si}\fblk.
		\end{align}
		
		Combining
		\eqref{uijs123}--\eqref{uijs3} gives
		\begin{align}\label{uijs+}
			\abs{\bar u_{i,j,s}^{+++}(x)}
			\lesssim k^3(k\mu_0)^{-(N-1)}e^{-\frac1{16}k\ga\bar\si}\fblk\quad \forall x\in {}^{s+1}\Omega_i^{j,+}.
		\end{align}
		
		By similar arguments as above we have
		\begin{align}\label{nuijs+}
			\abs{\na\bar{u}_{i,j,s}^{+++}(x)} \lesssim k^4(k\mu_0)^{-(N-1)} e^{-\frac1{16}k\ga\bar\si}\fblk,\quad \forall x\in{}^{s+1}\Omega_i^{j,+}.
		\end{align}
		Finally,  \eqref{4.22} follows by combining \eqref{uijs+}--\eqref{nuijs+}. 	\end{proof}

	\begin{lemma}\label{lem46}
		Under the conditions of \eqref{A1}--\eqref{A2}, if $k\si_0d$ is sufficiently large, then for  $i,j,s=1,\ldots,N-1$, we have
		\begin{align*}
			\normE{\bar u_{i,j,s}^{\pm\pm\pm}}_{\frac12,\pa\Omega_{i,j,s}^\pml}
			&\lesssim k^5(k\mu_0)^{-(2N-3)}
			e^{-\frac1{16}k\ga\bar\si}\fblk.
		\end{align*}
	\end{lemma}
	\begin{proof}
		We only prove for $\bar u_{i,j,s}^{+++}$. From \eqref{leq_vga}, we have
		\begin{align} \label{4.37}  
			\normE{\bar{u}_{i, j, s}^{+++}}_{H^{\frac12}(\pa\Omega_{i, j, s}^\pml)}
			\lesssim  \max_{x\in \pa \Om_{i,j,s}^\pml } \big(k^\frac12\abs{\bar{u}_{i,j,s}^{+++}(x)}+\abs{\nabla\bar{u}_{i,j,s}^{+++}(x)} \big).		
		\end{align}		
		From Algorithm~\ref{alg41}, direct calculations gives
		\begin{align}\label{4.38}
			\Psi_{i,j,s}^{+++}(\bar f_{i,j,s-1}^{+++})
			&=J_{i,j}^{-1}\na\!\cdot\!\big(A_{i,j}\na\chi_s^+\big)
			\bar u_{i,j,s-1}^{+++}\notag\\
			&\quad+2J_{i,j}^{-1}A_{i,j}\na\chi_s^+\!\cdot\!
			\na\bar u_{i,j,s-1}^{+++}-\chi_s^+f_{i,j,s}^{++}
			\quad\text{in }{}^s\Omega.
		\end{align}
		Thus, analogously to \eqref{3.38.1}, for $x\in\pa\Omega_{i,j,s}^\pml$,
		\begin{align}\label{uijs_abc}
			\bar u_{i,j,s}^{+++}(x)
			&=\int_{\R^3}\big((1-\chi_s^+)f_{i,j,s}^{++}(y)+f_{i,j,s+1}^{++}(y)\big)
			J_{i,j,s}G(\tilde x^{i,j,s},\tilde y^{i,j,s})\,dy\notag\\
			&\quad+\int_{{}^s\Omega}\na_y\!\cdot\!\big(A_{i,j}\na_y\chi_s^+\big)
			\bar u_{i,j,s-1}^{+++}(y)G(\tilde x^{i,j,s},\tilde y^{i,j,s})\,dy\notag\\
			&\quad+2\int_{{}^s\Omega}A_{i,j}\na_y\chi_s^+\!\cdot\!
			\na_y\bar u_{i,j,s-1}^{+++}(y)G(\tilde x^{i,j,s},\tilde y^{i,j,s})\,dy\notag\\
			&=: \roma+\romb+\romc.
		\end{align}
		
        First, we estimate $\roma$. Denote  $^{s,s+1}\Om_{i,i+1}^{j,j+1}:=(\Om_i\cup \Om_{i+1})\cap(\Om^j\cup \Om^{j+1})\cap(^s\Om\cup {}^{s+1}\Om)\cap B_l$. Since $f_{i,j,s}^{++}=F_{i,j}^{++}$ in ${}^s\Omega$ and is zero elsewhere,  and $\chi_s^+=0$ on ${}^{s+1}\Omega$, by \eqref{pmlalg3F},  Algorithm~\ref{alg3}, and \eqref{fij+}, similar to \eqref{3.38.1} and its estimation, we have
		\begin{align*}
		\abs{\roma}
		&=\bigg|\int_{(\Om^j\cup \Om^{j+1})\cap({}^s\Omega\cup{}^{s+1}\Omega)}(1-\chi_s^+)F_{i,j}^{++}(y)
		J_{i,j,s}(y)G(\tilde x^{i,j,s},\tilde y^{i,j,s})\,dy\bigg|\\
		&=\bigg|\int_{(\Om^j\cup \Om^{j+1})\cap({}^s\Omega\cup{}^{s+1}\Omega)}(1-\chi_s^+)\Big[  (1-\ga_j^+)\big((1-\be_i^+)f_i+f_{i+1}+\tilde f_i^+\big) \\
		&\qquad - J_i^{-1}\na\!\cdot\!(A_i\na\ga_j^+)\bar u_{i,j-1}^{++}
		-2J_i^{-1}A_i\na\ga_j^+\!\cdot\!\na\bar u_{i,j-1}^{++} \Big]
		J_{i,j,s}(y)G(\tilde x^{i,j,s},\tilde y^{i,j,s})\,dy\bigg|\\
		&\lesssim
		\normef{f}{^{s,s+1}\Om_{i,i+1}^{j,j+1}}
		\norme{G(\tilde x^{i,j,s},\tilde\cdot^{i,j,s})}
		{^{s,s+1}\Om_{i,i+1}^{j,j+1}}
		\\
		&\quad+
		\normef{\tilde f_i^+}
		{\Omega_{i,\frac34}^-\cap(\Omega^j\cup\Omega^{j+1})\cap({}^s\Omega\cup{}^{s+1}\Omega)}
		\norme{G(\tilde x^{i,j,s},\tilde\cdot^{i,j,s})}
		{\Omega_{i,\frac34}^-\cap(\Omega^j\cup\Omega^{j+1})\cap({}^s\Omega\cup{}^{s+1}\Omega)}
		\\
		&\quad+
		\int_{\Omega_{i,\frac34}^+\cap(\Omega^j\cup\Omega^{j+1})\cap({}^s\Omega\cup{}^{s+1}\Omega)}
		|\tilde f_i^+(y)|\,|G(\tilde x^{i,j,s},\tilde y^{i,j,s})|\,dy
		\\
		&\quad+
		\norme{\bar u_{i,j-1}^{++}}
		{\Omega_{i,\frac34}^{j-}\cap({}^s\Omega\cup{}^{s+1}\Omega)}
		\norm{G(\tilde x^{i,j,s},\tilde\cdot^{i,j,s})}_{L^2\big(\Omega_{i,\frac34}^{j-}\cap({}^s\Omega\cup{}^{s+1}\Omega)\big)} 
		\\
		&\quad+
		\int_{\Omega_{i,\frac34}^{j+}\cap({}^s\Omega\cup{}^{s+1}\Omega)}
		\big(  \abs{\bar u_{i,j-1}^{++}}
		+\abs{\na\bar u_{i,j-1}^{++}}\big)
		|G(\tilde x^{i,j,s},\tilde y^{i,j,s})|\,dy
		\\
		&=: \roma_1+\roma_2+\roma_3+\roma_4+\roma_5.
		\end{align*}
		Similar to \eqref{bu_i+I}, we have			\begin{align*}
			\abs{\roma_1} 
			&\lesssim  ke^{-k\ga\bar\si} \fblk.
		\end{align*}
		For $\roma_2$, similar to the estimate of the $I_2$ in Lemma~\ref{lem9}, by \eqref{eq_fip}, we have 
		\begin{align*}
		\abs{\roma_2}
		&\lesssim \mu_0^{-1} e^{-\frac14k\ga\bar\si}\fblk.
		\end{align*}
		For $\roma_3$, similar to the estimate of the $I_3$ in Lemma~\ref{lem9}, \eqref{tfi} and \eqref{barui} yield		
		\begin{align*}
		\abs{\roma_3}
		&\lesssim k^2e^{-\frac14k\ga\bar\si}\fblk.
		\end{align*}
		Next we estimate $\roma_4$. Using \eqref{buij.2} and following the proof of \eqref{bu_i+II1}, we have 
		\begin{align*}
		\abs{\roma_4}
		&\lesssim k(k\mu_0)^{-j}\fblk\max_{y\in \Omega_{i,\frac34}^{j-}\cap({}^s\Omega\cup{}^{s+1}\Omega)}  e^{-k\im\rho(\tilde x^{i,j,s},\tilde y^{i,j,s})}\\
		&\lesssim k(k\mu_0)^{-j} e^{-\frac14k\ga\bar\si}\fblk.
		\end{align*}
		Finally, for $\roma_5$, Lemma~\ref{lemma_omp} with $t=\frac34$ gives on $\Omega_{i,\frac34}^{j+}$
		\begin{align*}
		\abs{\bar u_{i,j-1}^{++}}
		+\abs{\na\bar u_{i,j-1}^{++}}	\lesssim k^3e^{-\frac1{16}k\ga\bar\si}\fblk.
		\end{align*}
		Using Lemma~\ref{Bound_G}, we conclude that
		\begin{align*}
		\abs{\roma_5}
		&\lesssim k^3e^{-\frac1{16}k\ga\bar\si}\fblk.
		\end{align*}
		Combining the estimate for $\roma_1$--$\roma_5$, 
		we have
		\begin{align}\label{uijsa}
		\abs{\roma}
		&\lesssim k^3(k\mu_0)^{-j} e^{-\frac1{16}k\ga\bar\si}\fblk.
		\end{align}

		Next, we estimate $\romb$. For later use in the recursion, we estimate $\normef{\bar{f}_{i,j}^{++}}{ {}^{s-1}\Omega\cup{}^s\Omega}$ . By using Algorithm~\ref{alg3}, \eqref{fij+}, \eqref{eq_fip}, and \eqref{buij.2}, we have
		\begin{align}\label{barfij}
			\normef{\bar{f}_{i,j}^{++}}{ {}^{s-1}\Omega\cup{}^s\Omega}
			&\le\normef{f_{i,j}^+}{ {}^{s-1}\Omega\cup{}^s\Omega}
			+\normef{\Psi_{i,j}^{++}(\bar f_{i,j-1}^{++})}{ {}^{s-1}\Omega\cup{}^s\Omega}\notag\\
			&\lesssim\normef{f_{i,j}^+}{ {}^{s-1}\Omega\cup{}^s\Omega}
			+k^{-1}\norme{\bar u_{i,j-1}^{++}}{\R^3}\notag\\
			&\lesssim(k\mu_0)^{-j}\fblk.
		\end{align}
	
From \eqref{4.21}, it is clear that
		\begin{align*}
			|\romb| &\lesssim \int_{^s\Omega_i^{j-}} \abs{\bar{u}_{i, j,s-1}^{+++}(y)} \abs{ G(\tilde{x}^{i, j,s}, \tilde{y}^{i, j,s})} dy + \int_{^s\Omega_i^{j+}} \abs{\bar{u}_{i, j,s-1}^{+++}(y)} \abs{ G(\tilde{x}^{i, j,s}, \tilde{y}^{i, j,s})} dy\\
			&=: \romb_1+\romb_2.
		\end{align*}
		By \eqref{4.18}, the inf--sup condition \eqref{4.17}, \eqref{4.38}, \eqref{pmlalg3F}, \eqref{barfij}, \eqref{fij+}, and \eqref{eq_fip}, we have
		\begin{align}\label{baruijs_1}
			\norme{\bar u_{i, j,s-1}^{+++}}{\R^3} &\lesssim  \mu_0^{-1} \normef{\bar f_{i, j,s-1}^{+++} + f_{i, j,s}^{++}}{\R^3}          \\
			&\lesssim  (k\mu_0)^{-1} \norme{\bar u_{i, j,s-2}^{+++}}{\R^3} + \mu_0^{-1}\normef{F_{i,j}^{++}}{({}^{s-1}\Omega\cup{}^s\Omega)} \notag\\
			&\lesssim  (k\mu_0)^{-1} \norme{\bar u_{i, j,s-2}^{+++}}{\R^3} + \mu_0^{-1}\normef{\bar f_{i,j}^{++}+f_{i,j+1}^+}{({}^{s-1}\Omega\cup{}^s\Omega)} \notag\\
			&\lesssim  (k\mu_0)^{-1}
			\norme{\bar u_{i,j,s-2}^{+++}}{\R^3}
			+ \mu_0^{-1}(k\mu_0)^{-j}\fblk , \notag
		\end{align}
		where $\bar u_{i,j, 0}^{+++}=0$. By recursion,  we have
		\begin{align}\label{baruijs_2}
			\norme{\bar u_{i,j,s-1}^{+++}}{\R^3}
			&\lesssim k(k\mu_0)^{-j-1}
			\sum_{p=0}^{s-2}(k\mu_0)^{-p}\fblk\notag\\
			&\lesssim k(k\mu_0)^{-(j+s-1)}\fblk.
		\end{align}
		Following the proof of \eqref{bu_i+II1} and using \eqref{baruijs_2}, we can obtain that
		\begin{align}
			\abs{\romb_1}
			&\lesssim
			\norm{\bar u_{i,j,s-1}^{+++}}_{L^2({}^s\Omega_i^{j-})}
			\left(\int_{{}^s\Omega_i^{j-}}\abs{G(\tilde{x}^{i, j,s}, \tilde{y}^{i, j,s})}^2dy\right)^{1/2}\notag\\
			&\lesssim (k\mu_0)^{-(j+s-1)} 
			e^{-\frac1{16}k\ga\bar\si}\fblk,
		\end{align}
		By Lemmas~\ref{lem45} and \ref{Bound_G} we have
		\begin{align}
			\abs{\romb_2}
			&\lesssim
			\max_{{}^s\Omega_i^{j+}}|\bar u_{i,j,s-1}^{+++}|
			\int_{{}^s\Omega_i^{j+}}\abs{G(\tilde{x}^{i, j,s}, \tilde{y}^{i, j,s})}dy\notag\\
			&\lesssim k^3(k\mu_0)^{-(N-1)}
			e^{-\frac1{16}k\ga\bar\si}\fblk.
		\end{align}
	    Hence, for any $x\in\pa\Oijs$,
		\begin{align}\label{uijsb}
			\abs{\romb}\lesssim k^3(k\mu_0)^{-(2N-3)}
			e^{-\frac1{16}k\ga\bar\si}\fblk.
		\end{align}
		
		Similarly,
		\begin{align}\label{uijsc}
			\abs{\romc}\lesssim k^4(k\mu_0)^{-(2N-3)}
			e^{-\frac1{16}k\ga\bar\si}\fblk.
		\end{align}
		
		Combining \eqref{uijs_abc}--\eqref{uijsc}, we obtain that,  for $x\in\partial\Omega_{i, j,s}^\pml$, 
		\begin{align}\label{baruij.}
			\abs{\bar{u}_{i, j,s}^{+++}(x)} \lesssim k^4(k\mu_0)^{-(2N-3)}
			e^{-\frac1{16}k\ga\bar\si}\fblk.
		\end{align}
		
		By similar arguments,  we have for $x\in\partial\Omega_{i,j,s}^\pml$,
		\begin{align}\label{nbaruij.}
			\abs{\nabla\bar{u}_{i, j,s}^{+++}(x)} \lesssim k^5(k\mu_0)^{-(2N-3)}
			e^{-\frac1{16}k\ga\bar\si}\fblk.
		\end{align}

	\end{proof}	
	
	\begin{lemma}\label{lem47}
		Under the condition of \eqref{A1}--\eqref{A2}, if $k\si_0d$ is sufficiently large, then for $i=1, 2, \cdots, N-1$,  we have
		\begin{align}\label{4.44} 
			\norme{\tu_i^\pm-\bar{u}_i^\pm}{\Oi} \lesssim k^6(k\mu_0)^{-N(N-1)^2}e^{-\frac1{16}k\ga\bar\si}\fblk.
		\end{align}
	\end{lemma}
	
	\begin{proof}	
		 We only prove \eqref{4.44} for $\norme{\tu_i^+-\bar{u}_i^+}{\Oi}$. From \eqref{pmlalg1F}, \eqref{barui}, and  Algorithms~\ref{alg1} and \ref{alg44}, similar to \eqref{Fi-cFi}, we have,  for $i,j,s=1, \cdots, N-1$,
		\begin{align*}
			\normef{F_{i}^+-\tF_{i}^+}{\Oijs} \lesssim k^{-1} \norme{\bar{u}_{i-1}^+-\tu_{i-1}^+}{\Om_{i-1}^{\pml}},
		\end{align*}
		where $\bar{u}_0^+-\tu_0^+=0$.
		Similarly, from Algorithms~\ref{alg43} and \ref{alg3}, \eqref{pmlalg3F}, and the above two estimate, we have, for $ i, j=1, \cdots, N-1$
		where $\bar{u}_{i, 0}^{++}-\ttu_{i, 0}^{++}=0$. Combining the above two estimates and \eqref{pmlalg3F}, we have
		\begin{align*}
			&\normef{F_{i, j}^{++}-\ttF_{i, j}^{++}}{\Oijs} =\normef{(\bar{f}_{i, j}^{++}+{f}_{i, j+1}^+)-(\ttf_{i, j}^{++}+\tf_{i, j+1}^+)}{\Oijs}\\	
			&\quad \lesssim \normef{ ({f}_{i, j}^++{f}_{i, j+1}^+) - (\tf_{i, j}^++\tf_{i, j+1}^+)}{\Oijs} \\
&\qquad+ k^{-1} \norme{\bar{u}_{i, j-1}^{++}-\ttu_{i, j-1}^{++}}{\Om_{i, j-1}^{\pml}}\hskip -5pt+\normef{f_{i,j}^+-\tf_{i,j}^+}{\Oijs} \\
			&\quad \lesssim \normef{F_i^+-\tF_i^+}{\Oijs}
			+k^{-1}\norme{\bar u_{i,j-1}^{++}-\ttu_{i,j-1}^{++}}{\Om_{i,j-1}^{\pml}}\\
			&\quad\lesssim k^{-1}\Big(  \norme{\bar{u}_{i-1}^+-\tu_{i-1}^+}{\Om_{i-1}^{\pml}} +  \norme{ \bar{u}_{i, j-1}^{++}-\ttu_{i, j-1}^{++}}{\Om_{i, j-1}^{\pml}}  \Big),
		\end{align*}
		which together with Algorithms~\ref{alg41} and \ref{alg42} implies that, for $i, j, s=1, \cdots, N-1$, 
		\begin{align*}
			&\normef{\bar{f}_{i, j, s}^{+++}-\ff_{i, j, s}^{+++}}{\Oijs}  \lesssim k^{-1} \norme{\bar{u}_{i, j, s-1}^{+++}-\uu_{i, j, s-1}^{+++}}{\Om_{i, j, s-1}^{\pml}} +   \normef{F_{i, j}^{++}-\ttF_{i, j}^{++}}{\Oijs}\\
			&\quad\lesssim k^{-1} \norme{\bar{u}_{i, j, s-1}^{+++}-\uu_{i, j, s-1}^{+++}}{\Om_{i, j, s-1}^{\pml}} + k^{-1} \norme{\bar{u}_{i-1}^+-\tu_{i-1}^+}{\Om_{i-1}^{\pml}}\\
			&\qquad+ k^{-1} \norme{ \bar{u}_{i, j-1}^{++}-\ttu_{i, j-1}^{++}}{\Om_{i, j-1}^{\pml}},
		\end{align*}
		where $\bar{u}_{i, j, 0}^{+++}-\ttu_{i, j, 0}^{+++}=0$.
		By combining, \eqref{4.18}, \eqref{4.9}, the local inf--sup condition~\eqref{lem43}, the lifting result \cite[Lemma 4.22]{melenk2011}, and Lemma~\ref{lem46}, similar to \eqref{eq7.1}, we obtain
		\begin{align*}
			&\norme{\bar{u}_{i, j, s}^{+++}-\uu_{i, j, s}^{+++}}{\Oijs} \\
			& \lesssim \mu_0^{-1} \Big( \normef{(\bar{f}_{i, j, s}^{+++}+{f}_{i, j, s+1}^{++}) - (\ff_{i, j, s}^{+++}+\ttf_{i, j, s+1}^{++})}{\Oijs} +  \normE{\bar{u}_{i, j, s}^{+++}}_{\frac12,\pa\Omega_{i, j, s}^\pml}  \Big)\notag\\
			& \lesssim (k\mu_0)^{-1}   \Big(   \norme{\bar{u}_{i, j, s-1}^{+++}-\uu_{i, j, s-1}^{+++}}{\Om_{i, j, s-1}^{\pml}}  +   \norme{\bar{u}_{i-1}^+-\tu_{i-1}^+}{\Om_{i-1}^{\pml}}  \notag\\
			& + \norme{ \bar{u}_{i, j-1}^{++}-\ttu_{i, j-1}^{++}}{\Om_{i, j-1}^{\pml}}+k^6(k\mu_0)^{-(2N-3)}e^{-\frac1{16}k\ga\bar\si}\fblk
			\Big),  \; i, j, s=1, \cdots, N-1.  \notag 	
		\end{align*}
		By induction, 
		\begin{align}\label{4.47}
			\norme{\bar{u}_{i, j, s}^{+++}-\uu_{i, j, s}^{+++}}{\Oijs} &\lesssim (k\mu_0)^{-s} \Big(    \norme{\bar{u}_{i-1}^+-\tu_{i-1}^+}{\Om_{i-1}^{\pml}}+ \norme{ \bar{u}_{i, j-1}^{++}-\ttu_{i, j-1}^{++}}{\Om_{i, j-1}^{\pml}}\notag \\ 
			 &\hskip 50pt+ k^6(k\mu_0)^{-(2N-3)}e^{-\frac1{16}k\ga\bar\si}\fblk
			\Big).    	
		\end{align}
		Similarly, we have for $s=N-1,\ldots,1$,
		\begin{align}\label{4.48}
			\norme{\bar{u}_{i, j, s}^{++-}-\uu_{i, j, s}^{++-}}{\Oijs} \notag\lesssim (k\mu_0)^{-(N-s)}\Big(&    \norme{\bar{u}_{i-1}^+-\tu_{i-1}^+}{\Om_{i-1}^{\pml}} + \norme{ \bar{u}_{i, j-1}^{++}-\ttu_{i, j-1}^{++}}{\Om_{i, j-1}^{\pml}} \notag\\
			& + k^6(k\mu_0)^{-(2N-3)}e^{-\frac1{16}k\ga\bar\si}\fblk
			\Big).		\end{align}
	    Combining \eqref{4.47},  \eqref{4.48},  Algorithm~\ref{alg42},  Lemma~\ref{lem4.1}, the lifting result \cite[Lemma 4.22]{melenk2011}, and  Lemma~\ref{lem9}, by similar arguments to Theorem~\ref{th4}, we conclude that
		\begin{align}
			& \norme{\bar{u}_{i, j}^{++}-\ttu_{i, j}^{++}}{\Omega_{i, j}^\pml}  \\
			&\lesssim   \sum_{s = 1}^{N} \norme{\bar{u}_{i, j}^{++}-\vv_{i, j}^{++}}{\Omega_{i, j, s}^\tru} + \normE{\bar{u}_{i, j}^{++}}_{\frac12,\pa\Omega_{i, j}^\pml}  \notag\\
			&\lesssim  \sum_{s=1}^{N-1} (k\mu_0)^{-(N-1)} \Big(    \norme{\bar{u}_{i-1}^+-\tu_{i-1}^+}{\Om_{i-1}^{\pml}} + \norme{ \bar{u}_{i, j-1}^{++}-\ttu_{i, j-1}^{++}}{\Om_{i, j-1}^{\pml}} \notag\\
			&\qquad + k^6(k\mu_0)^{-(2N-3)}e^{-\frac1{16}k\ga\bar\si}\fblk
			\Big).\notag	\\     	
			&\lesssim (k\mu_0)^{-(N-1)}\Big(    \norme{\bar{u}_{i-1}^+-\tu_{i-1}^+}{\Om_{i-1}^{\pml}} + \norme{ \bar{u}_{i, j-1}^{++}-\ttu_{i, j-1}^{++}}{\Om_{i, j-1}^{\pml}} \notag\\
			&\qquad + k^6(k\mu_0)^{-(2N-3)}e^{-\frac1{16}k\ga\bar\si}\fblk
			\Big),\quad i, j=1, \cdots, N-1.\notag	
		\end{align}
		 By induction, 
		\begin{align}
			\norme{\bar{u}_{i, j}^{++}-\ttu_{i, j}^{++}}{\Omega_{i, j}^\pml} &\lesssim (k\mu_0)^{-(N-1)j} \Big( \norme{\bar{u}_{i-1}^+-\tu_{i-1}^+}{\Om_{i-1}^{\pml}} \\
			&\quad+  k^6(k\mu_0)^{-(2N-3)}e^{-\frac1{16}k\ga\bar\si}\fblk \Big). \notag
		\end{align}
		
	    Similarly, we have for $j=N-1,\ldots,1$,
		\begin{align}
			\norme{\bar{u}_{i, j}^{+-}-\ttu_{i, j}^{+-}}{\Omega_{i, j}^\pml}  &\lesssim (k\mu_0)^{-(N-1)(N-j)} \Big( \norme{\bar{u}_{i-1}^+-\tu_{i-1}^+}{\Om_{i-1}^{\pml}} \\
			&\quad+  k^6(k\mu_0)^{-(2N-3)}e^{-\frac1{16}k\ga\bar\si}\fblk \Big). \notag
		\end{align}
		Combining above inequalities,  Algorithm~\ref{alg43},  Lemma~\ref{lem8},  Lemma~\ref{lem4}, we have	
		\begin{align}
			&\norme{\bar{u}_{i}^+-\tu_{i}^+}{\Omega_{i}^\pml} 
			\lesssim \sum_{j=1}^{N} \norme{\bar{u}_{i}^+-\tv_i^+}{\Omega_{i,j}^\tru} + \normE{\bar{u}_{i}^+}_{\frac12,\pa\Omega_{i}^\pml}   \notag\\
			&\lesssim (k\mu_0)^{-(N-1)^2} \Big( \norme{\bar{u}_{i-1}^+-\tu_{i-1}^+}{\Om_{i-1}^{\pml}} +  k^6(k\mu_0)^{-(2N-3)}e^{-\frac1{16}k\ga\bar\si}\fblk \Big)\notag,  
		\end{align}
		for $i=1, \cdots, N-1$. By induction, 
		\begin{align}
			\norme{\bar{u}_{i}^+-\tu_{i}^+}{\Omega_{i}^\pml} 
			&\lesssim (k\mu_0)^{-(N-1)^3}  k^6(k\mu_0)^{-(2N-3)}e^{-\frac1{16}k\ga\bar\si}\fblk  \notag\\
			&\lesssim k^6(k\mu_0)^{-(N^3-3N^2+5N-4)}e^{-\frac1{16}k\ga \bar{\si}}\fblk.  \notag
		\end{align}
		This together with $k\cs>1$ and $N^3-3N^2+5N-4\le N(N-1)^2$ for $N\ge2$, completes the proof.
	\end{proof}	
	
	Note that we have slightly enlarged the exponent in Lemma~\ref{lem47} for simplicity.
	The following theorem is a direct consequence of Theorem~\ref{th3},  Lemma~\ref{lem47} and the fact that $u_i^\pm=\bar{u}_i^\pm$ in $\Omega_i\cup\Omega_{i+1}$.
	\begin{theorem} \label{th48}
		Under the condition of \eqref{A1}--\eqref{A2} , if $k\si_0d$ is sufficiently large,  then we have
		\begin{align}\label{4.50}		
			\norme{\tu-\til{u}}{B_L} & \lesssim   (k\cs)^{N(N-1)^2} k^6 e^{-\frac1{16}k\ga \bar{\si}}\fblk.
		\end{align}
		where $\tilde u$ is the non-truncated PML solution to \eqref{eq_exactsol} and $\tu$ is the block-wise PSTDDM solution given by  Algorithm~\ref{alg44}.	
	\end{theorem}

	 \begin{remark}
		{\rm (i)}Theorem~\ref{th48} shows that the solution $\tu$ from our block-wise PSTDDM well approximates PML solution $\tilde u$ without truncation, with the corresponding error decaying exponentially as $k\sigma_0d$ increases.
		
		{\rm (ii)}The implicit constant in \eqref{4.50} may depend polynomially on the PML medium parameter $\si_0$ and the layer number $N$, owing to the local inf–sup estimates and the induction arguments in the three coordinate directions.
				
		{\rm (iii)} If the stability constant $\cs\ls k^{-1}$, then $(k\cs)^{N(N-1)^2}\ls1$  for fixed $N$, then the error estimate becomes
	$
	\norme{\tu-\tilde u}{B_L}\ls k^6e^{-\frac1{16}k\ga\bar\si}\fblk.
	$ 
	This error bound for the block-wise PSTDDM is larger than that of the bar-wise PSTDDM in Theorem~\ref{th_final} (see Remark~\ref{rem3.1}(iii)) and even larger than that of the layer-wise PSTDDM in Theorem~\ref{th4} (see Remark~\ref{rem2.3}). This may be caused by more intercommunication and accumulation between overlapping subdomains involved in the block-wise PSTDDM. 
		\end{remark}

	\section{Numerical examples} \label{sec5} 
	In this section, we present numerical experiments to evaluate the performance of the layer-wise and block-wise PSTDDMs as direct solvers and as preconditioners for GMRES applied to the truncated PML problem. All computations are performed in MATLAB on a computer equipped with an Intel(R) Xeon(R) Silver 4309Y processor running at 2.8 GHz and 512 GB of memory.
	
	We set the PML parameters to $\sigma_j(x_j)=\sigma_0$ and $\zeta=0$. The PML widths in the $x_1$, $x_2$, and $x_3$ directions are equal, $d_1=d_2=d_3=d$, where $d$ denotes the PML thickness. In the following examples, $\sigma_0$ is chosen so that the exponential decay factor satisfies $e^{-kd\sigma_0}\leq 10^{-3}$. The functions $\be_i^\pm(x_1)$, for $i=2,\ldots,N-1$, used in the source-transfer algorithm are defined as in \cite{pstddm}; specifically,
\begin{equation*}
	{\be_i^+}(x_1)=\begin{cases}
		1, & \ze_i\leq x_1<\ze_i+\De\ze/4, \\
		\eta_i(x_1), & \ze_i+\De\ze/4\leq x_1< \ze_i+3\De\ze/4, \\
		0, & \ze_i+3\De\ze/4\leq x_1\leq \ze_{i+1}
	\end{cases}
	\quad \text{ and }\quad {\be_i^-}=1-{\be_i^+},
\end{equation*}
 where $	\eta_i(x_1)=1+\left(\frac{x_1-(\ze_i+\De\ze/4)}{\De\ze/2}\right)^4-2\left(\frac{x_1-(\ze_i+\De\ze/4)}{\De\ze/2}\right)^2$. \( \chi_{j}^\pm \) and \( \gamma_j^\pm \) are defined similarly.

Layer-wise PDTDDM~\cref{alg2} and block-wise PSTDDM~\cref{alg44} are discretized using the finite element method (FEM). We solve the truncated PML problems in these algorithms using trilinear finite elements on a Cartesian mesh of cubes with side length $h$. The discrete source-transfer operators are computed by replacing the truncated PML solutions with their finite element approximations. In the error definitions below, $u_p$ denotes the direct numerical solution obtained using the layer-wise PSTDDM in~\cref{alg2} or the block-wise PSTDDM in~\cref{alg44}.

Let $u_f$ be the trilinear FEM solution of the original truncated PML problem~\eqref{eq_pmlBL}--\eqref{eq_pmlBLBC}, and let $u_i$ be the trilinear interpolant of the exact solution $u$ on the Cartesian mesh. We measure the relative errors of these numerical approximations in the $H^1$ seminorm:
\begin{align*}
	e_p = & \frac{|u-u_p|_{H^1(B_l)}}{|u|_{H^1(B_l)}},\quad e_f = \frac{|u-u_f|_{H^1(B_l)}}{|u|_{H^1(B_l)}},\quad e_i = \frac{|u-u_i|_{H^1(B_l)}}{|u|_{H^1(B_l)}}.
\end{align*}

We consider the PML problem~\eqref{eq_pmlBL}--\eqref{eq_pmlBLBC} with constant wavenumber, choosing $f$ so that the exact solution is
\begin{equation*}
	u=\left\{ \begin{aligned}
		 & 64r^3 (64r^3 + 48r^2 - 48r + 9) \frac{ e^{ikr} }{ 4\pi r }, & r\le 0.25, \\
		 & \frac{ e^{ikr} }{ 4\pi r }, & r>0.25.
	\end{aligned} \right.
\end{equation*}
Clearly, $u\in C^2(\R^3)$ and $\supp f\subseteq B_l := \set{x=(x_1,x_2,x_3)\in\R^3: |x_j|\leq 0.5,\ j=1,2,3}$. Here we set $l_1=l_2=l_3=l=0.5$.

We first test the layer-wise PSTDDM in~\cref{alg2}, setting $d_1=d_2=d_3=d=1/N$. The left graph of \cref{fig:layer1} shows the relative errors of the finite element solutions, the interpolants, and the discrete PSTDDM solutions for $k=40,80,120$ and $N=6$. The errors of the FEM solution $u_f$ and the discrete PSTDDM solution $u_p$ are nearly identical in all cases. For the FEM computation with mesh size $h<1/96$, MATLAB's backslash solver runs out of memory because of the large number of degrees of freedom.

Next, we test the layer-wise PSTDDM as a preconditioner for unrestarted GMRES applied to the linear system obtained from the FEM discretization of the original truncated PML problem~\eqref{eq_pmlBL}--\eqref{eq_pmlBLBC}. We use a maximum of 50 iterations and stop when the relative preconditioned residual is at most $10^{-4}$. The right graph of \cref{fig:layer1} shows the number of preconditioned GMRES iterations. The iteration count decreases as $h$ decreases and, for sufficiently fine meshes, appears to remain nearly bounded independently of the wavenumber.

To study how the layer-wise PSTDDM behaves as the number of layers increases, we solve the problem with fixed mesh size $h=1/96$ and wavenumbers $k=40,80,120$. The left of \cref{fig:layer2} shows the relative errors of the direct PSTDDM solutions as a function of $N$. The errors increase slightly with $N$ for $k=40$ and $80$, but decrease for $k=120$, yielding better behavior than the theoretical estimate predicts. We also use the layer-wise PSTDDM as a preconditioner and plot the number of preconditioned GMRES iterations against $N$ in the right of \cref{fig:layer2}. Although the iteration count increases with the number of layers, it remains small (at most 10), even when the layer width is less than one wavelength.

\begin{figure}
	\includegraphics[width=0.45\linewidth]{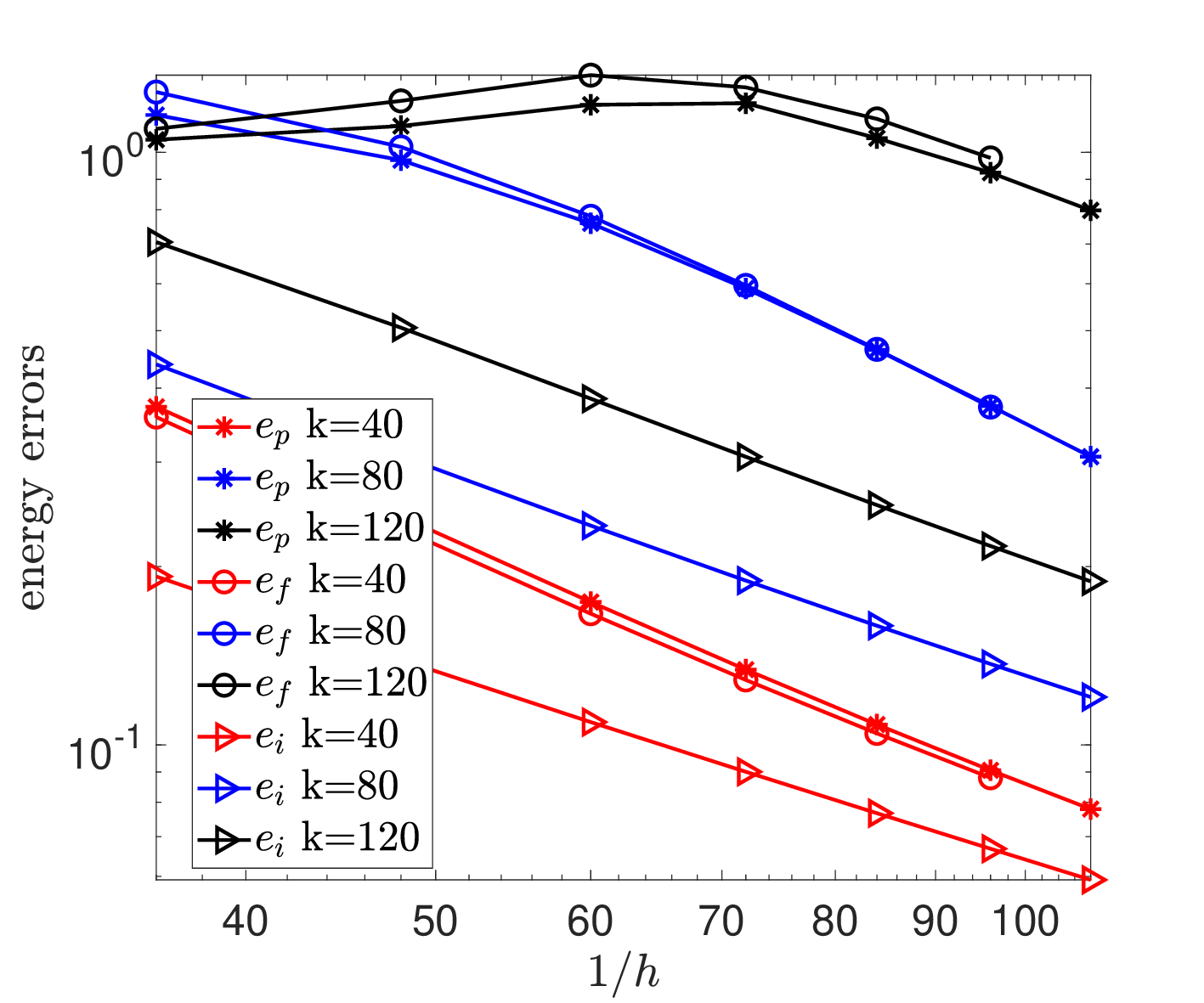}
	\includegraphics[width=0.45\linewidth]{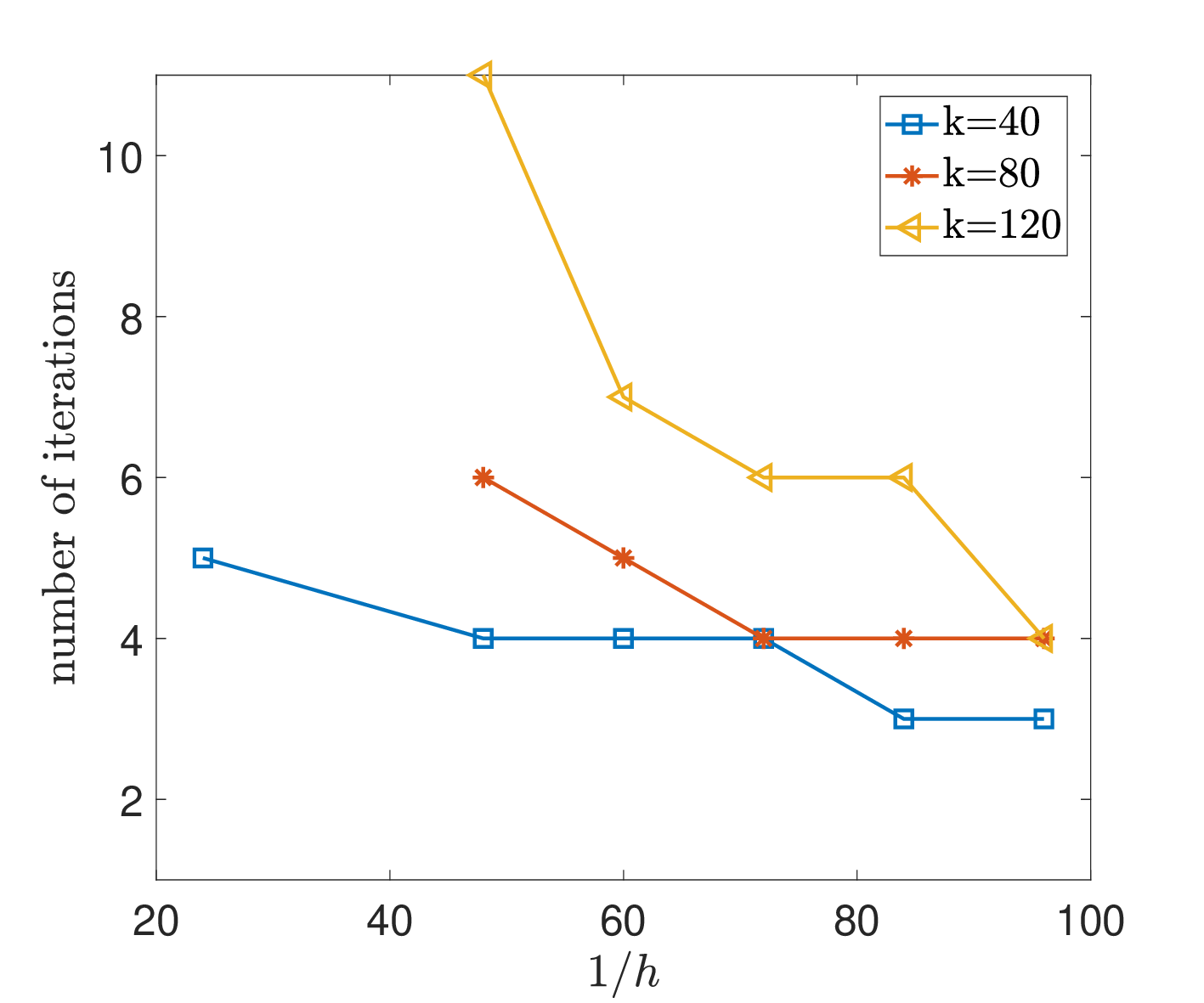}
	\caption{Left: Log-log plot of the relative errors $e_f$, $e_p$, and $e_i$ versus $1/h$ for $N=6$ and $k=40,80,120$. Right: Number of preconditioned GMRES iterations when the layer-wise PSTDDM is used as a preconditioner.}
	\label{fig:layer1}
\end{figure}

\begin{figure}
	\includegraphics[width=0.45\linewidth]{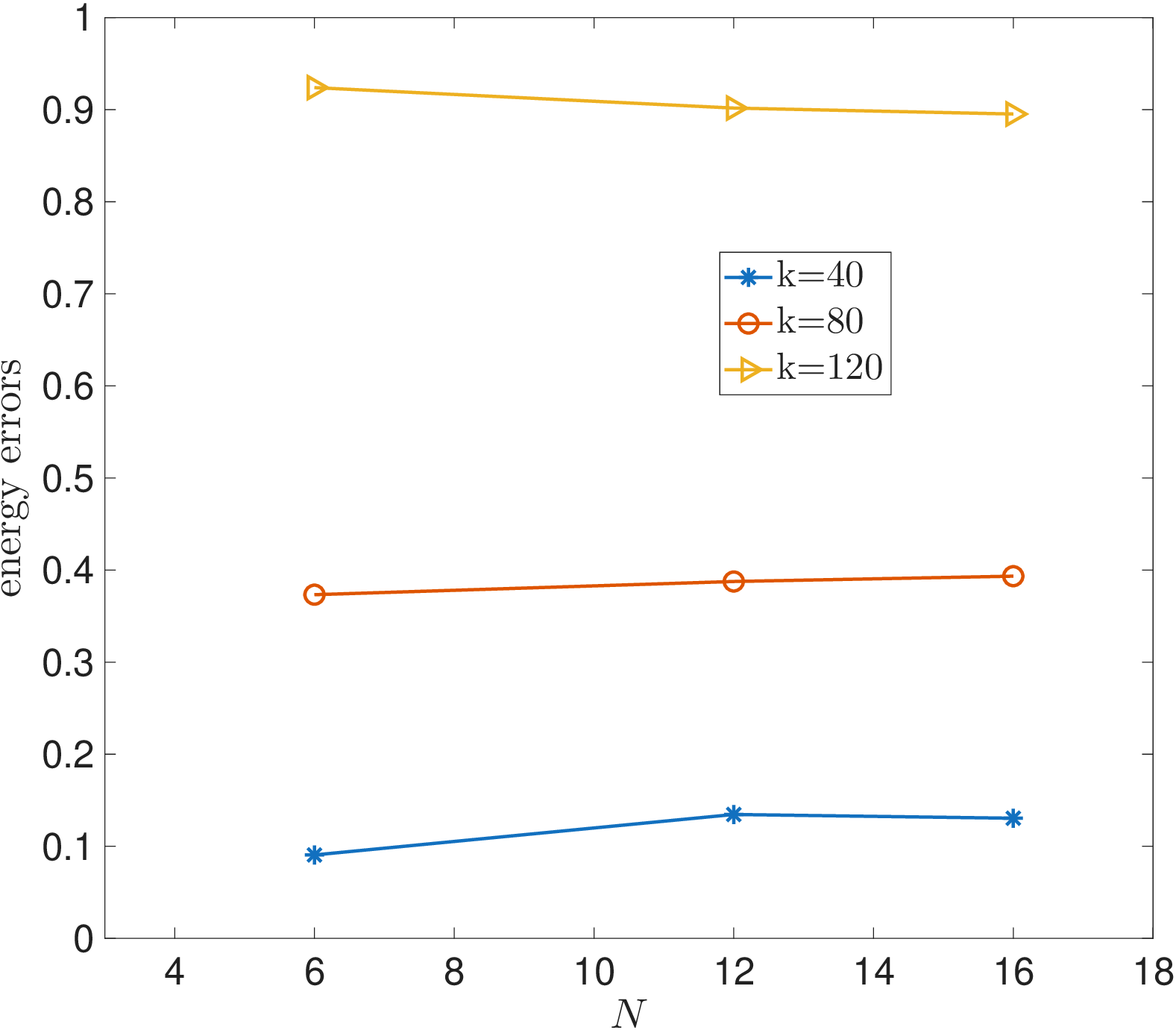}
	\includegraphics[width=0.45\linewidth]{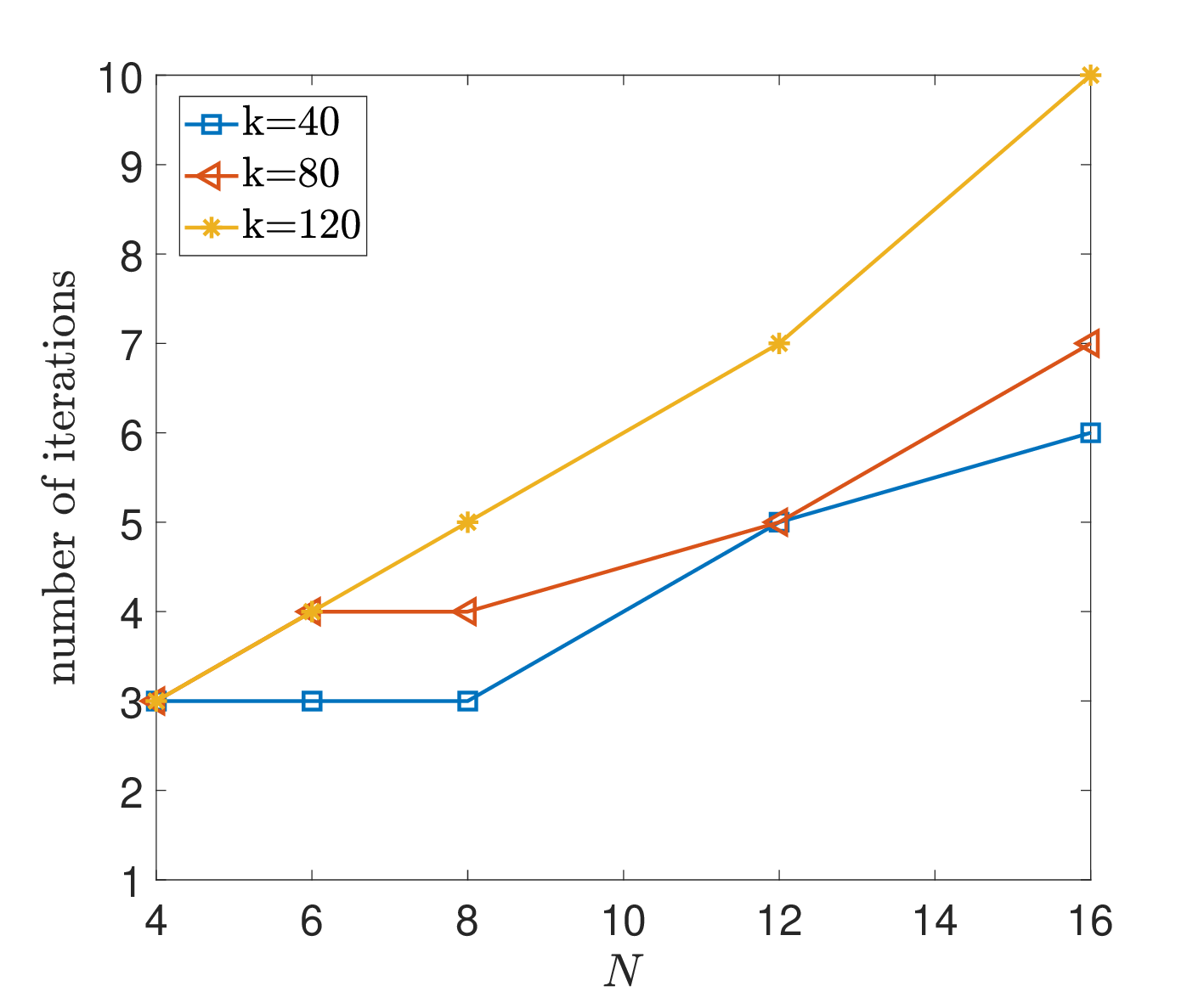}
	\caption{Left: Relative errors $e_p$ of the layer-wise PSTDDM solutions (see~\cref{alg2}) for fixed $h=1/96$. Right: Number of preconditioned GMRES iterations when the PSTDDM is used as a preconditioner, also for fixed $h=1/96$.}
	\label{fig:layer2}
\end{figure}

We next investigate the block-wise PSTDDM (see~\cref{alg44}). We set $d_1=d_2=d_3=d=1/N$ and retain the PML parameters specified at the beginning of this section.

The left graph of \cref{fig:block1} shows the relative errors of the numerical solutions. Similar to the layered PSTDDM, the errors \( e_f \) and \( e_p \) are nearly identical in all cases. However, the block-based PSTDDM can be implemented at a mesh size of \( h = 1/180 \), whereas the finite element method can only be implemented at \( h = 1/96 \). We also use the block-wise PSTDDM as a preconditioner for GMRES with the same unrestarted iteration limit and relative preconditioned residual tolerance as above. The right graph of \cref{fig:block1} shows the number of preconditioned GMRES iterations for $k=40,80,120$. Unlike the layered PDTDDM, when \( k \) is small, the number of iterations may increase as the mesh size decreases; however, when the mesh size is sufficiently small, the method performs better for large \( k = 120 \).

Finally, in~\cref{fig:block2}, we plot the relative errors of the block-wise PSTDDM solutions (left) and the number of preconditioned GMRES iterations (right) against $1/N$ for fixed $h=1/96$. As \( N \) increases, the relative error of the block-wise PSTDDM increases relatively significantly when \(k\) is small, but remains relatively stable at higher wave numbers. The GMRES iteration count increases more rapidly with $N$ than it does for the layer-wise PSTDDM. This behavior is consistent with the theoretical results.

\begin{figure}
	\includegraphics[width=0.45\linewidth]{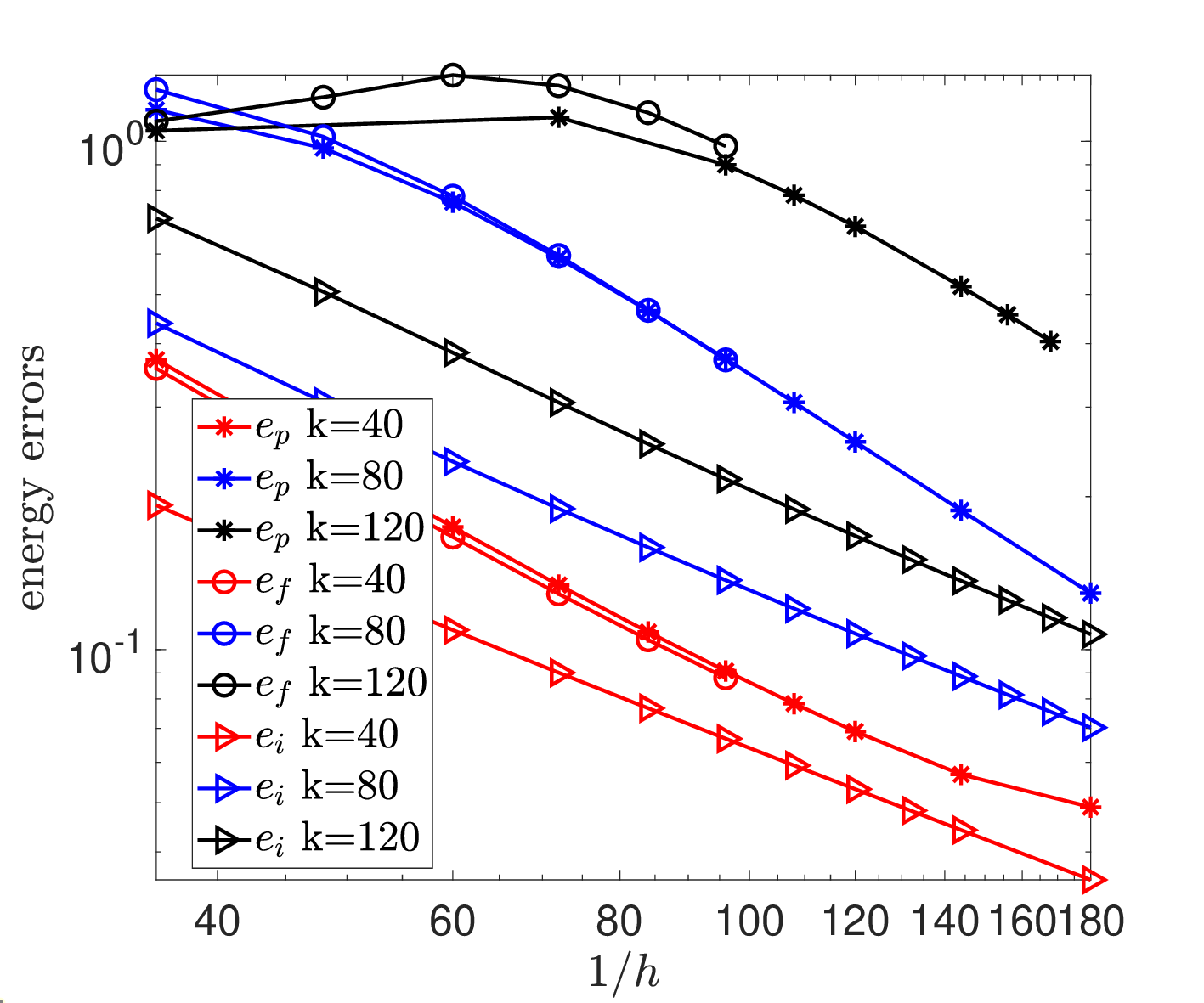}
	\includegraphics[width=0.45\linewidth]{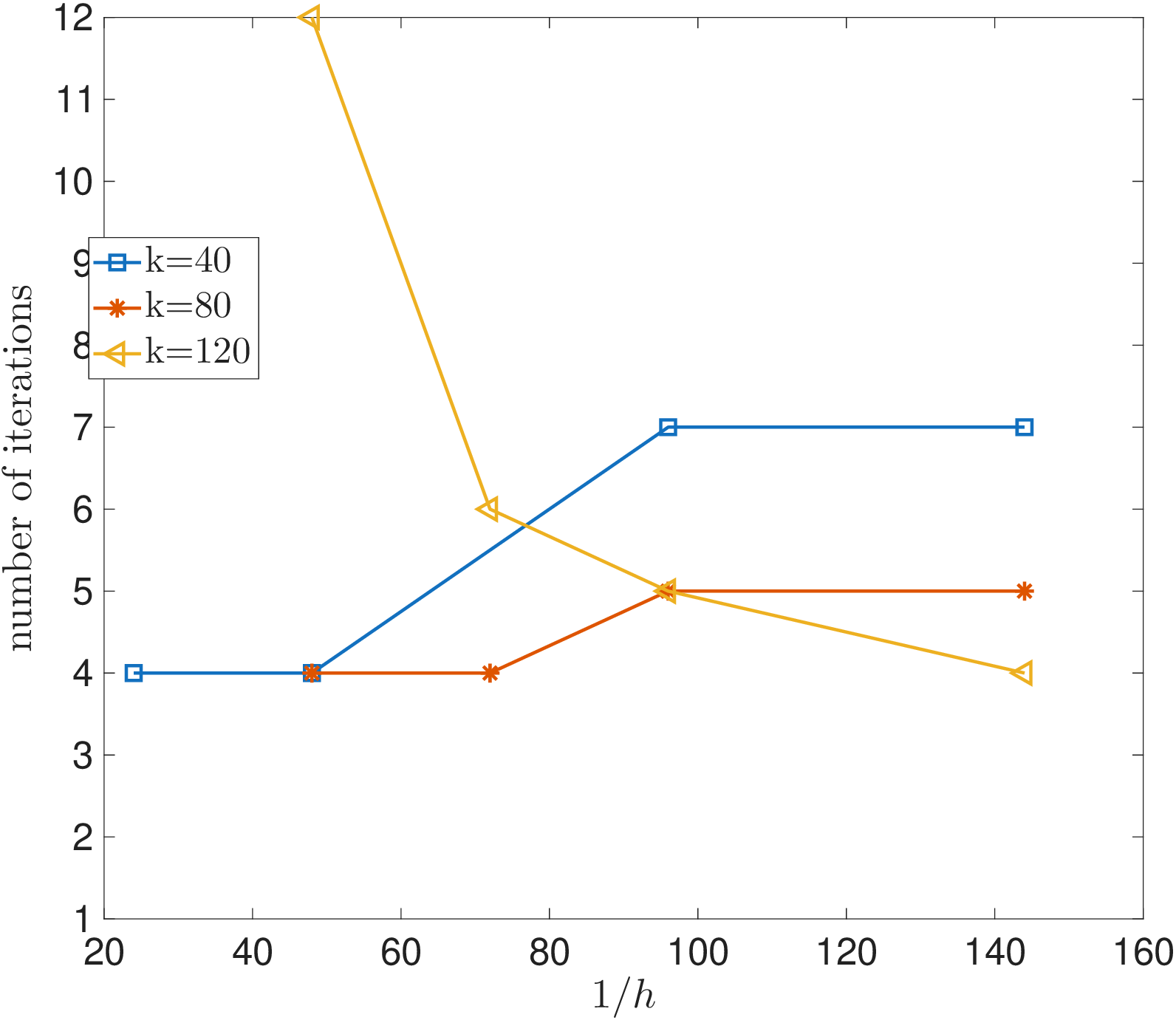}
	\caption{Left: Relative errors of the interpolant and the numerical solutions obtained using the block-wise PSTDDM (see~\cref{alg44}) as a direct solver, with $N=6$. Right: Number of preconditioned GMRES iterations when the block-wise PSTDDM is used as a preconditioner.}
	\label{fig:block1}
\end{figure}

\begin{figure}
	\includegraphics[width=0.45\linewidth]{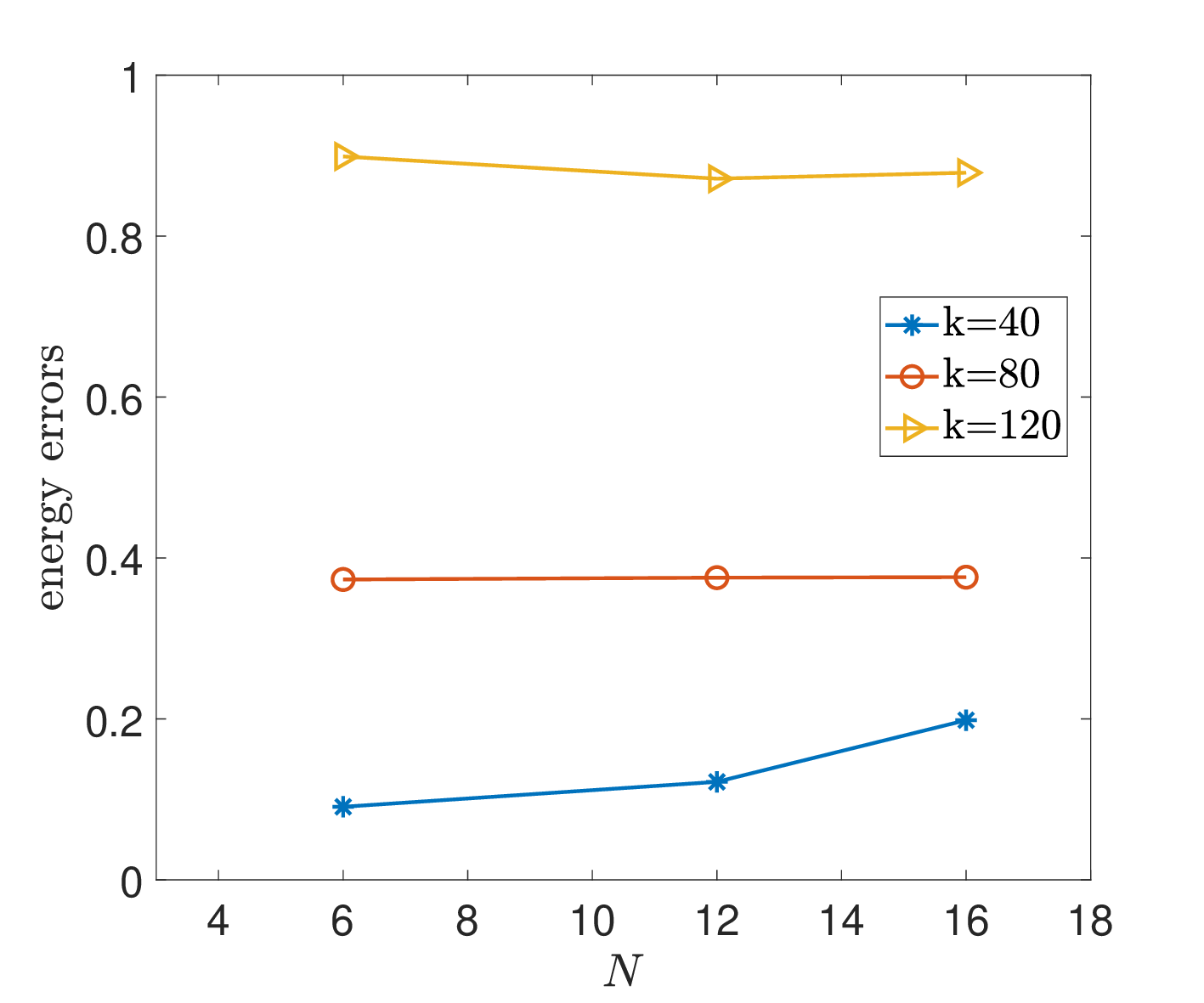}
	\includegraphics[width=0.45\linewidth]{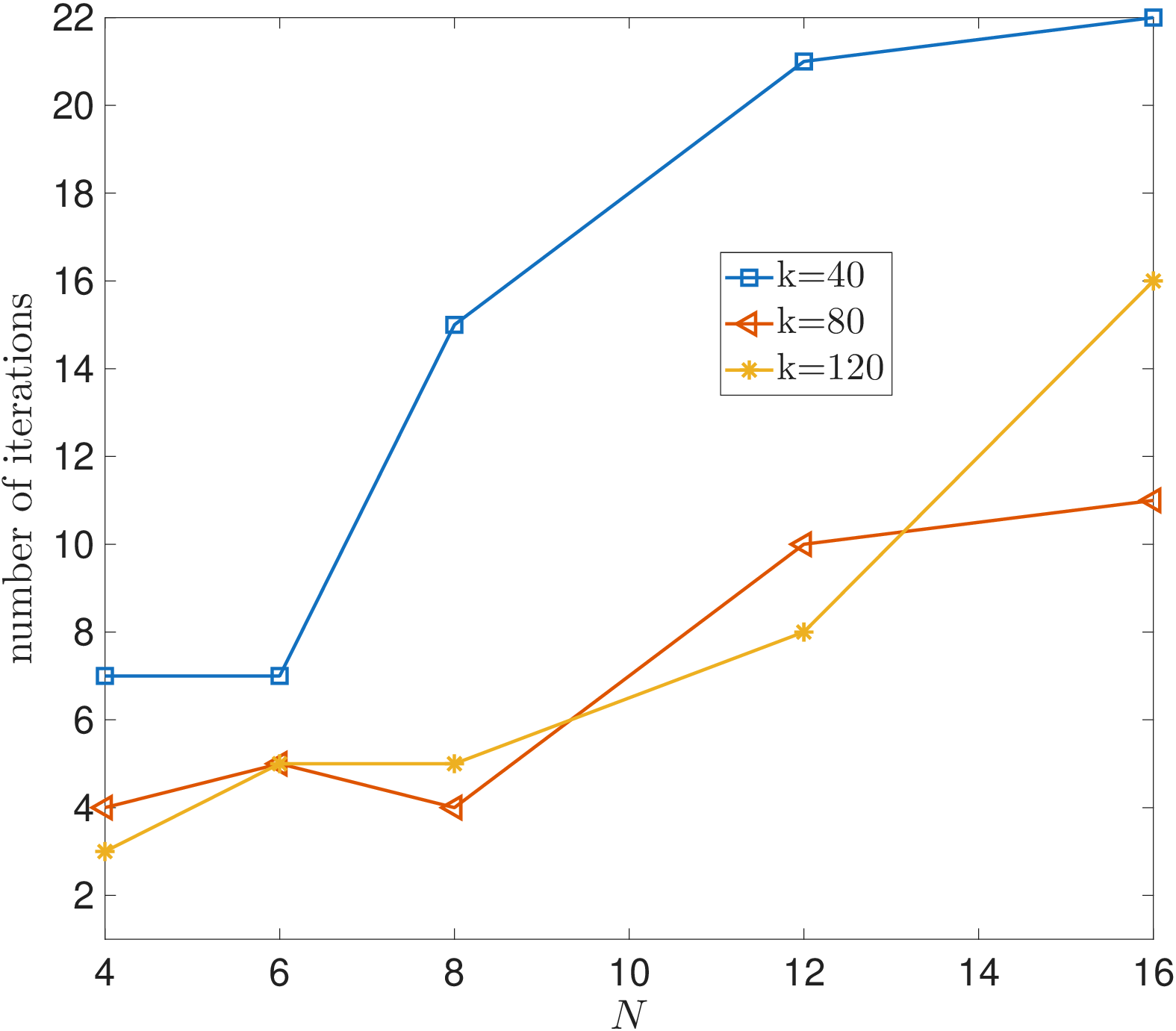}
	\caption{Left: Relative errors $e_p$ of the block-wise PSTDDM solutions (see~\cref{alg44}) for $h=1/96$. Right: Number of preconditioned GMRES iterations when the block-wise PSTDDM is used as a preconditioner, also for $h=1/96$.}
	\label{fig:block2}
\end{figure}

	\appendix 
	\section{The Proof of Lemma \ref{decayex} }	
	%
	%
	\begin{proof}
		Since the proofs of $u_i^+$ and $u_i^-$ are quite similar, we only prove $u_i^+$.
		
		By \ref{rstab} we have
		\begin{align*}
			&\norme{u_1^+}{\R^3} \lesssim \mu_0^{-1}\normef{f|_{\Omega_{1}\cup\Omega_2}}{\R^3},\\
			&\norme{u_i^+}{\R^3} \lesssim \mu_0^{-1}\normef{f|_{\Omega_i\cup\Omega_{i+1}} + \Psi_{i}^+}{\R^3},i\geq 2
		\end{align*}
		where
		\begin{align*}
			\normef{\Psi_{i}^+}{\R^3} = \sup_{v\in H^1(R^3)} \frac{ (\Psi_{i}^+, v) }{ \norme{v}{\R^3}} \lesssim \frac{ \norme{u_{i-1}^+}{\R^3} \cdot \frac{1}{k} \norme{v}{\R^3} }{ \norme{v}{\R^3} } \lesssim \frac{1}{k} \norme{u_{i-1}^+}{\R^3}
		\end{align*}
		So we have
		\begin{align*}
			\norme{u_i^+}{\R^3} \lesssim &\mu_0^{-1}\normef{f|_{\Omega_{i}\cup\Omega_{i+1}}}{\R^3} + \frac{1}{k} \mu_0^{-1} \norme{u_{i-1}^+}{\R^3} \\
			\lesssim  & \mu_0^{-1}\normef{f|_{\Omega_{i}\cup\Omega_{i+1}}}{\R^3} + \frac{1}{k} \mu_0^{-1} \Big( \mu_0^{-1}\normef{f|_{\Omega_{i-1}\cup\Omega_{i}}}{\R^3} + \frac{1}{k} \mu_0^{-1} \norme{u_{i-2}^+}{\R^3} \Big).
		\end{align*}
		By induction,we have
		\begin{align*}	    
			\norme{u_i^+}{\R^3} &\lesssim (\frac{1}{k}\mu_0^{-1})^{i-1}\mu_0^{-1} \flay \lesssim k(k\cs)^i \flay	    	
		\end{align*}
		
		As for the exponential decay of $\abs{u_i^{\pm}(x)}$ and $\abs{\nabla u_i^{\pm}(x)}$, we will use the induction argument to prove.     
		
		First, we consider the case when $i=1$. From Lemma~\ref{Im_roub}, we have $\im\rho(\tilde{x}, \tilde{y})\ge \sigma_0 \abs{x-y}/4$ for $\abs{x-y}\geq 2\sqrt{3} \beta\overline{l}_{\max}$.  Since $f_1$ is supported in $B_l$, by Lemma~\ref{lemma_G}, and the fact that $\abs{x}\geq 2\abs{y}$, $\abs{x-y}\geq 2\sqrt{3} \beta \overline{l}_{\max}$ for $\abs{x}\geq M_1$ and $y\in B_l$, we have
		\begin{align*}
			\abs{u_1^+(x)} \lesssim & \normE{f_1 + f_2}_{B_l}' \cdot \mathop{\max}_{y\in B_l} \big( k\abs{G(\tilde{x}, \tilde{y})} + \abs{\nabla_xG(\tilde{x}, \tilde{y})} \big)\\ 		
			\lesssim & k \max_{y\in B_l} e^{-k\im\rho(\tilde{x}, \tilde{y})} \normE{f_1 + f_2}_{B_l}'\\
			\lesssim & k \max_{y\in B_l} e^{-k\sigma_0 \abs{x-y}/4} \normE{f_1 + f_2}_{B_l}'\\
			\lesssim & k e^{-k\sigma_0 \abs{x}/8} \normE{f_1 + f_2}_{B_l}'.
		\end{align*}
		Similarly, $\abs{\nabla u_1^+(x)} \lesssim k^2e^{-k\sigma_0 \abs{x}/8} \normE{f_1 + f_2}_{B_l}'.$
		Then we have for any $\abs{x}\geq M_1$,
		\begin{align*}
			k\abs{u_1^+(x)}+\abs{\nabla u_1^+(x)} \lesssim k^2 e^{-\frac{1}{8}k\si_0\abs{x}}(\normE{f_1+f_2}_{B_l}') \lesssim k^2  e^{-\frac{1}{8}k\si_0\abs{x}} \flay.
		\end{align*}
		
		Next we consider the case $i\geq 2$. By the definition of $u_i^+$ and \eqref{psi1}, we have
		\begin{align*}
			\abs{u_i^+(x)} \lesssim & \abs{\int_{\Omega_i \cup \Omega_{i+1}} ( f_i(y) + f_{i+1}(y) ) \tilde{G}(x, y)dy}\\
			& +\int_{\Omega_i} \Big(\abs{u_{i-1}^+(y)}  + \abs{\nabla u_{i-1}^+(y)} \Big) \abs{\tilde{G}(x, y)} dy \\
			:= & \roma +\romb.
		\end{align*}
		By the same argument above, we have  \( \roma \lesssim k e^{-k \im\rho(\tilde{x}, \tilde{y}) } \normE{f_i + f_{i+1}}_{B_l}' \). 
		To estimate $\romb$, we split the integral into two parts
		\begin{align*}
			\romb = & \int_{\Omega_i\cap B(M_{i-1}) } \Big(\abs{u_{i-1}^+(y)}  + \abs{\nabla u_{i-1}^+(y)} \Big) \abs{\tilde{G}(x, y)}dy \\
			& + \int_{\Omega_i \setminus \overline{B(M_{i-1})}} \Big(\abs{u_{i-1}^+(y)}  + \abs{\nabla u_{i-1}^+(y)} \Big) \abs{\tilde{G}(x, y)}dy \\
			:= & \romb_1 +\romb_2,
		\end{align*} where $B(M_{i-1})$ is the circle of radius $M_{i-1}$. We first give the estimate of $\romb_1$. For any $\abs{x}\geq M_i$, $\abs{y}\leq M_{i-1}$, we have $\abs{x}\geq 2\abs{y}$ and $\abs{x-y}\geq 2\sqrt{3} \beta \overline{l}_{\max}$. By the above estimate, we have
		\begin{align*}
			\romb_1 \lesssim & \norme{ u_{i-1}^+}{\R^3}  \bigg( \int_{\Omega_i\cap B(M_{i-1}) }  \big(  \abs{G(\tilde{x}, \tilde{y})}  \big)^2 dy  \bigg)^{\frac12} \\
			\lesssim & k(k\cs)^{i-1}  \max_{y\leq M_{i-1}} e^{-k \im \rho(\tilde{x}, \tilde{y}) } \bigg( \int_{\Omega_i\cap B(M_{i-1}) }   \frac{ 1 }{ |x-y|^2 }   dy  \bigg)^{\frac12} \flay \\
			\lesssim &  k(k\cs)^{i-1} e^{-k\sigma_0 \abs{x}/8} \flay,
		\end{align*}
		where we have used $(1-\be^{-1})^2\frac{|x-y|}{3}\ge(1-\be^{-1})^2\frac{|x|-|x|/2}{3}\ge\frac{|x|}{8}$ to derive the last inequality.	
		For $\romb_2$, by the induction assumption
		\begin{align*}
			k\abs{u_{i-1}(y)} + \abs{\nabla u_{i-1}(y)} \lesssim k^2(k\cs)^{i-2} e^{-\frac{1}{8}k\si_0\abs{y}} \flay,\quad \forall \abs{y}\leq M_{i-1},
		\end{align*}
		we have
		\begin{align*}
			\romb_2 \lesssim & \int_{\Omega_i \backslash \overline{B(M_{i-1})}} \bigg( k\abs{u_{i-1}(y)} + \abs{\nabla u_{i-1}(y)} \bigg) \abs{\tilde{G}(x, y)}  dy \\
			\lesssim & k^2(k\cs)^{i-2} \flay \int_{\Omega_i \backslash \overline{B(M_{i-1})}} e^{-\frac{1}{8}k \si_0\abs{y}-k \im \rho(\tilde{x}, \tilde{y})}  \frac{1}{\abs{x-y}}  dy \\ 
			\lesssim & k^2(k\cs)^{i-2} e^{-\frac{1}{8}k\si_0\abs{x}} \flay \int_{\Omega_i\backslash\overline{B(M_{i-1})}}e^{-\frac{1}{8}k\si_0\abs{x-y}}   \frac{1}{\abs{x-y}} dy, 
		\end{align*}where we have used the triangle inequality $\abs{x-y}+\abs{y}\geq \abs{x}$. By simply calculation we have 		
		\begin{align*}
			&\int_{\Omega_i\backslash\overline{B(M_{i-1})}}e^{-\frac{1}{8}k\si_0\abs{x-y}}  \frac{1}{\abs{x-y}} dy\\
			&=\int_{0}^{\pi}\int_{0}^{2\pi}\int_{0}^{\infty}  \frac{ 1 }{ r }  e^{-c kr} r^2 \sin \phi d r d\phi d\theta\\
			&\lesssim  \int_{0}^{\infty}   \frac{ 1 }{ r }  e^{-c kr} r^2  dr  \\
			&\lesssim  \frac{ 1 }{ k^2 }.
		\end{align*}
		Therefore, we have 
		\begin{align*}
			\romb_2 \lesssim & (k\cs)^{i-2} e^{-\frac{1}{8}k\si_0\abs{x}} \flay.
		\end{align*}
		This implies for the case $i\geq 2$, we have $k\abs{u_{i}(y)} + \abs{\nabla u_{i}(y)} \lesssim k^2(k\cs)^{i-1} e^{-\frac{1}{8}k\si_0\abs{x}} \flay, \forall \abs{x}\geq M_{i}$.
		
		Then the proof is completed.

	\end{proof}

	\bibliographystyle{siam}
	\bibliography{reference_AMS.bib}

@misc{FuGongLiWang24arxiv,
	author        = {Fu, Shubin and Gong, Shihua and Li, Guanglian and Wang, Yueqi},
	title         = {On Edge Multiscale Space based Hybrid {Schwarz} Preconditioner for {Helmholtz} Problems with Large Wavenumbers},
	year          = {2024},
	eprint        = {2408.08198},
	archivePrefix = {arXiv},
	primaryClass  = {math.NA},
	doi           = {10.48550/arXiv.2408.08198},
	url           = {https://arxiv.org/abs/2408.08198},
	note          = {arXiv preprint}
}

@article {LuXuZhengZou25,
	AUTHOR = {Lu, Peipei and Xu, Xuejun and Zheng, Bowen and Zou, Jun},
	TITLE = {Two-level hybrid {S}chwarz preconditioners for the {H}elmholtz
	equation with high wave number},
	JOURNAL = {SIAM J. Numer. Anal.},
	FJOURNAL = {SIAM Journal on Numerical Analysis},
	VOLUME = {63},
	YEAR = {2025},
	NUMBER = {6},
	PAGES = {2187--2220},
	ISSN = {0036-1429,1095-7170},
	MRCLASS = {65N55 (65F08 65F10)},
	MRNUMBER = {4980418},
	DOI = {10.1137/24M168533X},
	URL = {https://doi.org/10.1137/24M168533X},
}

@article {GrahamSpenceZou20,
	AUTHOR = {Graham, Ivan G. and Spence, Euan A. and Zou, Jun},
	TITLE = {Domain decomposition with local impedance conditions for the
	{H}elmholtz equation with absorption},
	JOURNAL = {SIAM J. Numer. Anal.},
	FJOURNAL = {SIAM Journal on Numerical Analysis},
	VOLUME = {58},
	YEAR = {2020},
	NUMBER = {5},
	PAGES = {2515--2543},
	ISSN = {0036-1429,1095-7170},
	MRCLASS = {65N55 (65F08 65N22 65N30)},
	MRNUMBER = {4150274},
	MRREVIEWER = {Brian\ Bradie},
	DOI = {10.1137/19M1272512},
	URL = {https://doi.org/10.1137/19M1272512},
}

@article {MaChupeng25,
	AUTHOR = {Ma, Chupeng and Alber, Christian and Scheichl, Robert and
	Zhang, Yongwei},
	TITLE = {Two-level restricted additive {S}chwarz preconditioner based
	on multiscale spectral generalized {FEM} for heterogeneous
	{H}elmholtz problems},
	JOURNAL = {J. Sci. Comput.},
	FJOURNAL = {Journal of Scientific Computing},
	VOLUME = {105},
	YEAR = {2025},
	NUMBER = {3},
	PAGES = {Paper No. 99, 28},
	ISSN = {0885-7474,1573-7691},
	MRCLASS = {65N30 (65F10 65N55)},
	MRNUMBER = {4992072},
	DOI = {10.1007/s10915-025-03138-y},
	URL = {https://doi.org/10.1007/s10915-025-03138-y},
}

@article {llw,
	AUTHOR = {Yang, Zhiguo and Wang, Li-Lian and Gao, Yang},
	TITLE = {A truly exact perfect absorbing layer for time-harmonic
	acoustic wave scattering problems},
	JOURNAL = {SIAM J. Sci. Comput.},
	FJOURNAL = {SIAM Journal on Scientific Computing},
	VOLUME = {43},
	YEAR = {2021},
	NUMBER = {2},
	PAGES = {A1027--A1061},
	ISSN = {1064-8275,1095-7197},
	MRCLASS = {65N35 (35J05 65F05 65N22 76Q05)},
	MRNUMBER = {4230432},
	MRREVIEWER = {B\"ulent\ Karas\"ozen},
	DOI = {10.1137/19M1294071},
	URL = {https://doi.org/10.1137/19M1294071},
}

@article{WangZheng25,
  author     = {Wang, Yuhao and Zheng, Weiying},
  title      = {On polynomial {PML}s for {H}elmholtz scattering problems with high wave-numbers},
  journal    = {J. Comput. Math.},
  fjournal   = {Journal of Computational Mathematics},
  volume     = {43},
  number     = {6},
  pages      = {1469--1487},
  year       = {2025},
  doi        = {10.4208/jcm.2510-m2025-0072},
  issn       = {0254-9409,1991-7139},
  mrclass    = {65N30 (65N12 78A40)},
  mrnumber   = {4979520},
}

@article {FreletNicaise2022,
	AUTHOR = {Chaumont-Frelet, Th\'{e}ophile and Gallistl, Dietmar and
	Nicaise, Serge and Tomezyk, J\'{e}r\^{o}me},
	TITLE = {Wavenumber-explicit convergence analysis for finite element
	discretizations of time-harmonic wave propagation problems
	with perfectly matched layers},
	JOURNAL = {Commun. Math. Sci.},
	FJOURNAL = {Communications in Mathematical Sciences},
	VOLUME = {20},
	YEAR = {2022},
	NUMBER = {1},
	PAGES = {1--52},
	ISSN = {1539-6746,1945-0796},
	MRCLASS = {65N30 (35J05 65N12 78A45 78M10)},
	MRNUMBER = {4359249},
	MRREVIEWER = {Seydi\ Battal Gazi Karakoc},
	DOI = {10.4310/CMS.2022.v20.n1.a1},
	URL = {https://doi.org/10.4310/CMS.2022.v20.n1.a1},
}

@book{brenner2008mathematical,
  author     = {Brenner, Susanne C. and Scott, L. Ridgway},
  title      = {The mathematical theory of finite element methods},
  series     = {Texts in Applied Mathematics},
  volume     = {15},
  edition    = {Third},
  pages      = {xviii+397},
  year       = {2008},
  publisher  = {Springer, New York},
  doi        = {10.1007/978-0-387-75934-0},
  isbn       = {978-0-387-75933-3},
  mrclass    = {65-01 (65-02)},
  mrnumber   = {2373954},
}

@article{pstddm,
  author     = {Du, Yu and Wu, Haijun},
  title      = {A pure source transfer domain decomposition method for {H}elmholtz equations in unbounded domain},
  journal    = {J. Sci. Comput.},
  fjournal   = {Journal of Scientific Computing},
  volume     = {83},
  number     = {3},
  pages      = {Paper No. 68, 29},
  year       = {2020},
  doi        = {10.1007/s10915-020-01249-2},
  issn       = {0885-7474,1573-7691},
  mrclass    = {65N30 (65N12 65N15 65N55 78A40)},
  mrnumber   = {4113348},
  mrreviewer = {Jinbiao\ Wu},
}

@article{stddm,
  author     = {Chen, Zhiming and Xiang, Xueshuang},
  title      = {A source transfer domain decomposition method for {H}elmholtz equations in unbounded domain},
  journal    = {SIAM J. Numer. Anal.},
  fjournal   = {SIAM Journal on Numerical Analysis},
  volume     = {51},
  number     = {4},
  pages      = {2331--2356},
  year       = {2013},
  doi        = {10.1137/130917144},
  issn       = {0036-1429,1095-7170},
  mrclass    = {65N06 (35J05 65N55)},
  mrnumber   = {3085122},
  mrreviewer = {Ivana\ Pultarov\'{a}},
}

@article{stddm2,
  author     = {Chen, Zhiming and Xiang, Xueshuang},
  title      = {A source transfer domain decomposition method for {H}elmholtz equations in unbounded domain {P}art {II}: {E}xtensions},
  journal    = {Numer. Math. Theory Methods Appl.},
  fjournal   = {Numerical Mathematics. Theory, Methods and Applications},
  volume     = {6},
  number     = {3},
  pages      = {538--555},
  year       = {2013},
  doi        = {10.4208/nmtma.2013.1217nm},
  issn       = {1004-8979,2079-7338},
  mrclass    = {65N06 (35J05 65N55)},
  mrnumber   = {3081789},
  mrreviewer = {Ivana\ Pultarov\'{a}},
}

@article{BABUSKA1995325,
  author     = {Babu\v{s}ka, Ivo and Ihlenburg, Frank and Paik, Ellen T. and Sauter, Stefan A.},
  title      = {A generalized finite element method for solving the {H}elmholtz equation in two dimensions with minimal pollution},
  journal    = {Comput. Methods Appl. Mech. Engrg.},
  fjournal   = {Computer Methods in Applied Mechanics and Engineering},
  volume     = {128},
  number     = {3-4},
  pages      = {325--359},
  year       = {1995},
  doi        = {10.1016/0045-7825(95)00890-X},
  issn       = {0045-7825,1879-2138},
  mrclass    = {65N30 (65N15)},
  mrnumber   = {1368049},
  mrreviewer = {G.\ W.\ Hedstrom},
}

@article{d3,
  author     = {Babu\v{s}ka, Ivo M. and Sauter, Stefan A.},
  title      = {Is the pollution effect of the {FEM} avoidable for the {H}elmholtz equation considering high wave numbers?},
  journal    = {SIAM J. Numer. Anal.},
  fjournal   = {SIAM Journal on Numerical Analysis},
  volume     = {34},
  number     = {6},
  pages      = {2392--2423},
  year       = {1997},
  doi        = {10.1137/S0036142994269186},
  issn       = {0036-1429},
  mrclass    = {65N30},
  mrnumber   = {1480387},
  mrreviewer = {Mark\ Ainsworth},
}

@article{D12,
  author     = {Cand\`es, Emmanuel and Demanet, Laurent and Ying, Lexing},
  title      = {A fast butterfly algorithm for the computation of {F}ourier integral operators},
  journal    = {Multiscale Model. Simul.},
  fjournal   = {Multiscale Modeling \& Simulation. A SIAM Interdisciplinary Journal},
  volume     = {7},
  number     = {4},
  pages      = {1727--1750},
  year       = {2009},
  doi        = {10.1137/080734339},
  issn       = {1540-3459,1540-3467},
  mrclass    = {65T50},
  mrnumber   = {2539196},
  mrreviewer = {Joanna\ E.\ Kapusta},
}

@article{2015Multilevel,
  author     = {Chen, Huangxin and Wu, Haijun and Xu, Xuejun},
  title      = {Multilevel preconditioner with stable coarse grid corrections for the {H}elmholtz equation},
  journal    = {SIAM J. Sci. Comput.},
  fjournal   = {SIAM Journal on Scientific Computing},
  volume     = {37},
  number     = {1},
  pages      = {A221--A244},
  year       = {2015},
  doi        = {10.1137/13091840X},
  issn       = {1064-8275,1095-7197},
  mrclass    = {65F08 (35J05 65F10 65N22 65N30)},
  mrnumber   = {3302598},
  mrreviewer = {Melina\ A.\ Freitag},
}

@article{Multikry2001,
  author     = {Elman, Howard C. and Ernst, Oliver G. and O'Leary, Dianne P.},
  title      = {A multigrid method enhanced by {K}rylov subspace iteration for discrete {H}elmholtz equations},
  journal    = {SIAM J. Sci. Comput.},
  fjournal   = {SIAM Journal on Scientific Computing},
  volume     = {23},
  number     = {4},
  pages      = {1291--1315},
  year       = {2001},
  doi        = {10.1137/S1064827501357190},
  issn       = {1064-8275,1095-7197},
  mrclass    = {65N22 (65N55 76M25 76Q05 78A45)},
  mrnumber   = {1885602},
}

@incollection{1998D21,
  author     = {Chevalier, Philippe and Nataf, Fr\'{e}d\'{e}ric},
  title      = {Symmetrized method with optimized second-order conditions for the {H}elmholtz equation},
  booktitle  = {Domain decomposition methods, 10 ({B}oulder, {CO}, 1997)},
  series     = {Contemp. Math.},
  volume     = {218},
  pages      = {400--407},
  year       = {1998},
  publisher  = {Amer. Math. Soc., Providence, RI},
  doi        = {10.1090/conm/218/03035},
  isbn       = {0-8218-0988-1},
  mrclass    = {65N55},
  mrnumber   = {1649636},
}

@article{2008D29,
  author     = {Gander, Martin J.},
  title      = {{Schwarz} methods over the course of time},
  journal    = {Electron. Trans. Numer. Anal.},
  fjournal   = {Electronic Transactions on Numerical Analysis},
  volume     = {31},
  pages      = {228--255},
  year       = {2008},
  issn       = {1068-9613},
  mrclass    = {65N55 (65F10)},
  mrnumber   = {2569603},
  mrreviewer = {Marco\ Donatelli},
}

@article{2006OptiSchNoOverlap,
  author     = {Gander, Martin J. and Magoul\`es, Fr\'{e}d\'{e}ric and Nataf, Fr\'{e}d\'{e}ric},
  title      = {Optimized {S}chwarz methods without overlap for the {H}elmholtz equation},
  journal    = {SIAM J. Sci. Comput.},
  fjournal   = {SIAM Journal on Scientific Computing},
  volume     = {24},
  number     = {1},
  pages      = {38--60},
  year       = {2002},
  doi        = {10.1137/S1064827501387012},
  issn       = {1064-8275,1095-7197},
  mrclass    = {65N55 (65N22)},
  mrnumber   = {1924414},
  mrreviewer = {E.\ G.\ D\cprime yakonov},
}

@article{2016OptiSchOverlap,
  author     = {Gander, Martin J. and Zhang, Hui},
  title      = {Optimized {S}chwarz methods with overlap for the {H}elmholtz equation},
  journal    = {SIAM J. Sci. Comput.},
  fjournal   = {SIAM Journal on Scientific Computing},
  volume     = {38},
  number     = {5},
  pages      = {A3195--A3219},
  year       = {2016},
  doi        = {10.1137/15M1021659},
  issn       = {1064-8275,1095-7197},
  mrclass    = {65N55 (65F10 65N22)},
  mrnumber   = {3554882},
  mrreviewer = {Dietrich\ Braess},
}

@article{2019D42,
  author     = {Liu, Yongxiang and Xu, Xuejun},
  title      = {An optimized {S}chwarz method with relaxation for the {H}elmholtz equation: the negative impact of overlap},
  journal    = {ESAIM Math. Model. Numer. Anal.},
  fjournal   = {ESAIM. Mathematical Modelling and Numerical Analysis},
  volume     = {53},
  number     = {1},
  pages      = {249--268},
  year       = {2019},
  doi        = {10.1051/m2an/2018061},
  issn       = {2822-7840,2804-7214},
  mrclass    = {65N55 (65N30)},
  mrnumber   = {3937351},
  mrreviewer = {Pierre\ Gosselet},
}

@article{2008D51,
  author     = {St-Cyr, A. and Gander, M. J. and Thomas, S. J.},
  title      = {Optimized multiplicative, additive, and restricted additive {S}chwarz preconditioning},
  journal    = {SIAM J. Sci. Comput.},
  fjournal   = {SIAM Journal on Scientific Computing},
  volume     = {29},
  number     = {6},
  pages      = {2402--2425},
  year       = {2007},
  doi        = {10.1137/060652610},
  issn       = {1064-8275,1095-7197},
  mrclass    = {65F10},
  mrnumber   = {2357620},
  mrreviewer = {Vladimir\ B.\ Larin},
}

@article{2010D25,
  author     = {Engquist, Bj\"{o}rn and Ying, Lexing},
  title      = {Sweeping preconditioner for the {H}elmholtz equation: moving perfectly matched layers},
  journal    = {Multiscale Model. Simul.},
  fjournal   = {Multiscale Modeling \& Simulation. A SIAM Interdisciplinary Journal},
  volume     = {9},
  number     = {2},
  pages      = {686--710},
  year       = {2011},
  doi        = {10.1137/100804644},
  issn       = {1540-3459,1540-3467},
  mrclass    = {65F08 (65N22)},
  mrnumber   = {2818416},
  mrreviewer = {Kees\ Vuik},
}

@article{2016D26,
  author     = {Eslaminia, Mehran and Guddati, Murthy N.},
  title      = {A double-sweeping preconditioner for the {H}elmholtz equation},
  journal    = {J. Comput. Phys.},
  fjournal   = {Journal of Computational Physics},
  volume     = {314},
  pages      = {800--823},
  year       = {2016},
  doi        = {10.1016/j.jcp.2016.03.022},
  issn       = {0021-9991,1090-2716},
  mrclass    = {65N22 (65F08)},
  mrnumber   = {3484959},
}

@article{2016D41,
  author     = {Liu, Fei and Ying, Lexing},
  title      = {Additive sweeping preconditioner for the {H}elmholtz equation},
  journal    = {Multiscale Model. Simul.},
  fjournal   = {Multiscale Modeling \& Simulation. A SIAM Interdisciplinary Journal},
  volume     = {14},
  number     = {2},
  pages      = {799--822},
  year       = {2016},
  doi        = {10.1137/15M1017144},
  issn       = {1540-3459,1540-3467},
  mrclass    = {65F08 (35J05 35J25 65N06 65N22)},
  mrnumber   = {3499544},
  mrreviewer = {Ahmed\ Salam},
}

@article{2013D52,
  author     = {Stolk, Christiaan C.},
  title      = {A rapidly converging domain decomposition method for the {H}elmholtz equation},
  journal    = {J. Comput. Phys.},
  fjournal   = {Journal of Computational Physics},
  volume     = {241},
  pages      = {240--252},
  year       = {2013},
  doi        = {10.1016/j.jcp.2013.01.039},
  issn       = {0021-9991,1090-2716},
  mrclass    = {65N06 (65N55)},
  mrnumber   = {3647407},
}

@article{2014D53,
  author     = {Stolk, Christiaan C.},
  title      = {An improved sweeping domain decomposition preconditioner for the {H}elmholtz equation},
  journal    = {Adv. Comput. Math.},
  fjournal   = {Advances in Computational Mathematics},
  volume     = {43},
  number     = {1},
  pages      = {45--76},
  year       = {2017},
  doi        = {10.1007/s10444-016-9475-y},
  issn       = {1019-7168,1572-9044},
  mrclass    = {65N55 (65N22)},
  mrnumber   = {3598832},
  mrreviewer = {Benjamin\ Wi-Lian\ Ong},
}

@article{stddmwang,
  author     = {Cui, Tao and Wang, Ziming and Xiang, Xueshuang},
  title      = {A source transfer domain decomposition method for time-harmonic {M}axwell's equations},
  journal    = {J. Sci. Comput.},
  fjournal   = {Journal of Scientific Computing},
  volume     = {103},
  number     = {2},
  pages      = {Paper No. 44, 31},
  year       = {2025},
  doi        = {10.1007/s10915-025-02872-7},
  issn       = {0885-7474,1573-7691},
  mrclass    = {65N55 (35Q61)},
  mrnumber   = {4884527},
  mrreviewer = {Shi\ Shu},
}

@article{2019D39,
  author     = {Leng, Wei and Ju, Lili},
  title      = {An additive overlapping domain decomposition method for the {H}elmholtz equation},
  journal    = {SIAM J. Sci. Comput.},
  fjournal   = {SIAM Journal on Scientific Computing},
  volume     = {41},
  number     = {2},
  pages      = {A1252--A1277},
  year       = {2019},
  doi        = {10.1137/18M1196170},
  issn       = {1064-8275,1095-7197},
  mrclass    = {65N55 (65F08 65F10 65Y05 65Y20)},
  mrnumber   = {3942861},
  mrreviewer = {Shi\ Shu},
}

@article{2019DSTDDM,
  author     = {Xiang, Xueshuang},
  title      = {Double source transfer domain decomposition method for {H}elmholtz problems},
  journal    = {Commun. Comput. Phys.},
  fjournal   = {Communications in Computational Physics},
  volume     = {26},
  number     = {2},
  pages      = {434--468},
  year       = {2019},
  doi        = {10.4208/cicp.oa-2018-0063},
  issn       = {1815-2406,1991-7120},
  mrclass    = {65N55},
  mrnumber   = {3939399},
}

@article{2010D25SweepingPML,
  author     = {Engquist, Bj\"{o}rn and Ying, Lexing},
  title      = {Sweeping preconditioner for the {H}elmholtz equation: moving perfectly matched layers},
  journal    = {Multiscale Model. Simul.},
  fjournal   = {Multiscale Modeling \& Simulation. A SIAM Interdisciplinary Journal},
  volume     = {9},
  number     = {2},
  pages      = {686--710},
  year       = {2011},
  doi        = {10.1137/100804644},
  issn       = {1540-3459,1540-3467},
  mrclass    = {65F08 (65N22)},
  mrnumber   = {2818416},
  mrreviewer = {Kees\ Vuik},
}

@article{D17,
  author     = {Chen, Zhiming and Wu, Haijun},
  title      = {An adaptive finite element method with perfectly matched absorbing layers for the wave scattering by periodic structures},
  journal    = {SIAM J. Numer. Anal.},
  fjournal   = {SIAM Journal on Numerical Analysis},
  volume     = {41},
  number     = {3},
  pages      = {799--826},
  year       = {2003},
  doi        = {10.1137/S0036142902400901},
  issn       = {0036-1429,1095-7170},
  mrclass    = {65N30 (35Q60 78A45 78M10)},
  mrnumber   = {2005183},
  mrreviewer = {Jean-Luc\ Guermond},
}

@article{D18,
  author     = {Chen, Zhiming and Wu, Xinming},
  title      = {An adaptive uniaxial perfectly matched layer method for time-harmonic scattering problems},
  journal    = {Numer. Math. Theory Methods Appl.},
  fjournal   = {Numerical Mathematics. Theory, Methods and Applications},
  volume     = {1},
  number     = {2},
  pages      = {113--137},
  year       = {2008},
  issn       = {1004-8979,2079-7338},
  mrclass    = {65N30 (35J05 35J25)},
  mrnumber   = {2441968},
  mrreviewer = {Luc\ Paquet},
}

@article{D20,
  author     = {Chen, Zhiming and Zheng, Weiying},
  title      = {Convergence of the uniaxial perfectly matched layer method for time-harmonic scattering problems in two-layered media},
  journal    = {SIAM J. Numer. Anal.},
  fjournal   = {SIAM Journal on Numerical Analysis},
  volume     = {48},
  number     = {6},
  pages      = {2158--2185},
  year       = {2010},
  doi        = {10.1137/090750603},
  issn       = {0036-1429,1095-7170},
  mrclass    = {65N30 (65N50 76M25 76Q05 78A45)},
  mrnumber   = {2763660},
  mrreviewer = {Barbara\ Zubik-Kowal},
}

@article{D23,
  author     = {Deraemaeker, Arnaud and Babu\v{s}ka, Ivo and Bouillard, Philippe},
  title      = {Dispersion and pollution of the {FEM} solution for the {Helmholtz} equation in one, two and three dimensions},
  journal    = {International Journal for Numerical Methods in Engineering},
  volume     = {46},
  number     = {4},
  pages      = {471--499},
  year       = {1999},
  doi        = {10.1002/(SICI)1097-0207(19991010)46:4<471::AID-NME684>3.0.CO;2-6},
}

@article{D35,
  author     = {Ihlenburg, F. and Babu\v{s}ka, I.},
  title      = {Finite element solution of the {H}elmholtz equation with high wave number. {I}. {T}he {$h$}-version of the {FEM}},
  journal    = {Comput. Math. Appl.},
  fjournal   = {Computers \& Mathematics with Applications. An International Journal},
  volume     = {30},
  number     = {9},
  pages      = {9--37},
  year       = {1995},
  doi        = {10.1016/0898-1221(95)00144-N},
  issn       = {0898-1221,1873-7668},
  mrclass    = {65N30 (65N15)},
  mrnumber   = {1353516},
  mrreviewer = {Manfred\ Dobrowolski},
}

@article{D36,
  author     = {Ihlenburg, Frank and Babu\v{s}ka, Ivo},
  title      = {Finite element solution of the {H}elmholtz equation with high wave number. {II}. {T}he {$h$}-{$p$} version of the {FEM}},
  journal    = {SIAM J. Numer. Anal.},
  fjournal   = {SIAM Journal on Numerical Analysis},
  volume     = {34},
  number     = {1},
  pages      = {315--358},
  year       = {1997},
  doi        = {10.1137/S0036142994272337},
  issn       = {0036-1429},
  mrclass    = {65N30},
  mrnumber   = {1445739},
  mrreviewer = {Manfred\ Dobrowolski},
}

@article{2010Analysis,
  author     = {Kim, Seungil and Pasciak, Joseph E.},
  title      = {Analysis of a {C}artesian {PML} approximation to acoustic scattering problems in {$\mathbb{R}^2$}},
  journal    = {J. Math. Anal. Appl.},
  fjournal   = {Journal of Mathematical Analysis and Applications},
  volume     = {370},
  number     = {1},
  pages      = {168--186},
  year       = {2010},
  doi        = {10.1016/j.jmaa.2010.05.006},
  issn       = {0022-247X,1096-0813},
  mrclass    = {76Q05 (76M10 76M25)},
  mrnumber   = {2651138},
  mrreviewer = {Weiying\ Zheng},
}

@article{D38,
  author     = {Lassas, Matti and Somersalo, Erkki},
  title      = {Analysis of the {PML} equations in general convex geometry},
  journal    = {Proc. Roy. Soc. Edinburgh Sect. A},
  fjournal   = {Proceedings of the Royal Society of Edinburgh. Section A. Mathematics},
  volume     = {131},
  number     = {5},
  pages      = {1183--1207},
  year       = {2001},
  doi        = {10.1017/S0308210500001335},
  issn       = {0308-2105,1473-7124},
  mrclass    = {35A35 (35J05 35P25 76Q05)},
  mrnumber   = {1862449},
  mrreviewer = {Hideo\ Soga},
}

@article{D54,
  author     = {Thompson, Lonny L.},
  title      = {A review of finite-element methods for time-harmonic acoustics},
  journal    = {The Journal of the Acoustical Society of America},
  volume     = {119},
  number     = {3},
  pages      = {1315--1330},
  year       = {2006},
  doi        = {10.1121/1.2164987},
}

@article{d15,
  author     = {Chen, Zhiming and Liu, Xuezhe},
  title      = {An adaptive perfectly matched layer technique for time-harmonic scattering problems},
  journal    = {SIAM J. Numer. Anal.},
  fjournal   = {SIAM Journal on Numerical Analysis},
  volume     = {43},
  number     = {2},
  pages      = {645--671},
  year       = {2005},
  doi        = {10.1137/040610337},
  issn       = {0036-1429,1095-7170},
  mrclass    = {65N30 (35J05 35P25 65N15 65N50 78A45)},
  mrnumber   = {2177884},
  mrreviewer = {Natalia\ Kopteva},
}

@article{chen2008adaptive,
  author     = {Chen, Junqing and Chen, Zhiming},
  title      = {An adaptive perfectly matched layer technique for 3-{D} time-harmonic electromagnetic scattering problems},
  journal    = {Math. Comp.},
  fjournal   = {Mathematics of Computation},
  volume     = {77},
  number     = {262},
  pages      = {673--698},
  year       = {2008},
  doi        = {10.1090/S0025-5718-07-02055-8},
  issn       = {0025-5718,1088-6842},
  mrclass    = {78A46 (65N30 78M10)},
  mrnumber   = {2373175},
  mrreviewer = {Snorre\ H.\ Christiansen},
}

@article{chen2013adaptive,
  author     = {Chen, Zhiming and Cui, Tao and Zhang, Linbo},
  title      = {An adaptive anisotropic perfectly matched layer method for 3-{D} time harmonic electromagnetic scattering problems},
  journal    = {Numer. Math.},
  fjournal   = {Numerische Mathematik},
  volume     = {125},
  number     = {4},
  pages      = {639--677},
  year       = {2013},
  doi        = {10.1007/s00211-013-0550-8},
  issn       = {0029-599X,0945-3245},
  mrclass    = {78A45 (65N12 65N30 74J20)},
  mrnumber   = {3127326},
  mrreviewer = {Agustin\ Mart\'{\i}n},
}

@article{chen2016elastic,
  author     = {Chen, Zhiming and Xiang, Xueshuang and Zhang, Xiaohui},
  title      = {Convergence of the {PML} method for elastic wave scattering problems},
  journal    = {Math. Comp.},
  fjournal   = {Mathematics of Computation},
  volume     = {85},
  number     = {302},
  pages      = {2687--2714},
  year       = {2016},
  doi        = {10.1090/mcom/3100},
  issn       = {0025-5718,1088-6842},
  mrclass    = {65N30 (65N12 65N50 76Q05)},
  mrnumber   = {3522967},
  mrreviewer = {Hans-Peter\ Helfrich},
}

@article{d16,
  author     = {Chen, Zhiming and Zheng, Weiying},
  title      = {Convergence of the uniaxial perfectly matched layer method for time-harmonic scattering problems in two-layered media},
  journal    = {SIAM J. Numer. Anal.},
  fjournal   = {SIAM Journal on Numerical Analysis},
  volume     = {48},
  number     = {6},
  pages      = {2158--2185},
  year       = {2010},
  doi        = {10.1137/090750603},
  issn       = {0036-1429,1095-7170},
  mrclass    = {65N30 (65N50 76M25 76Q05 78A45)},
  mrnumber   = {2763660},
  mrreviewer = {Barbara\ Zubik-Kowal},
}

@article{BERENGER1994185,
  author     = {Berenger, Jean-Pierre},
  title      = {A perfectly matched layer for the absorption of electromagnetic waves},
  journal    = {J. Comput. Phys.},
  fjournal   = {Journal of Computational Physics},
  volume     = {114},
  number     = {2},
  pages      = {185--200},
  year       = {1994},
  doi        = {10.1006/jcph.1994.1159},
  issn       = {0021-9991,1090-2716},
  mrclass    = {78-08 (78A40)},
  mrnumber   = {1294924},
}

@article{C18,
  author     = {Kim, Seungil and Pasciak, Joseph E.},
  title      = {Analysis of a {C}artesian {PML} approximation to acoustic scattering problems in {$\mathbb{R}^2$}},
  journal    = {J. Math. Anal. Appl.},
  fjournal   = {Journal of Mathematical Analysis and Applications},
  volume     = {370},
  number     = {1},
  pages      = {168--186},
  year       = {2010},
  doi        = {10.1016/j.jmaa.2010.05.006},
  issn       = {0022-247X,1096-0813},
  mrclass    = {76Q05 (76M10 76M25)},
  mrnumber   = {2651138},
  mrreviewer = {Weiying\ Zheng},
}

@article{r2r3,
  author     = {Bramble, James H. and Pasciak, Joseph E.},
  title      = {Analysis of a {C}artesian {PML} approximation to acoustic scattering problems in {$\mathbb{R}^2$} and {$\mathbb{R}^3$}},
  journal    = {J. Comput. Appl. Math.},
  fjournal   = {Journal of Computational and Applied Mathematics},
  volume     = {247},
  pages      = {209--230},
  year       = {2013},
  doi        = {10.1016/j.cam.2012.12.022},
  issn       = {0377-0427,1879-1778},
  mrclass    = {76Q05 (65N30)},
  mrnumber   = {3023309},
  mrreviewer = {Vadim\ D.\ Kryakvin},
}

@article{lw,
  author     = {Li, Yonglin and Wu, Haijun},
  title      = {{FEM} and {CIP}-{FEM} for {H}elmholtz equation with high wave number and perfectly matched layer truncation},
  journal    = {SIAM J. Numer. Anal.},
  fjournal   = {SIAM Journal on Numerical Analysis},
  volume     = {57},
  number     = {1},
  pages      = {96--126},
  year       = {2019},
  doi        = {10.1137/17M1140522},
  issn       = {0036-1429,1095-7170},
  mrclass    = {65N30 (65N12 65N15)},
  mrnumber   = {3899165},
  mrreviewer = {Christos\ K.\ Filelis-Papadopoulos},
}

@article{melenk2011,
  author     = {Melenk, J. M. and Sauter, S.},
  title      = {Wavenumber explicit convergence analysis for {G}alerkin discretizations of the {H}elmholtz equation},
  journal    = {SIAM J. Numer. Anal.},
  fjournal   = {SIAM Journal on Numerical Analysis},
  volume     = {49},
  number     = {3},
  pages      = {1210--1243},
  year       = {2011},
  doi        = {10.1137/090776202},
  issn       = {0036-1429,1095-7170},
  mrclass    = {65N30 (35J05 65N12)},
  mrnumber   = {2812565},
  mrreviewer = {Snorre\ H.\ Christiansen},
}

@article{Xu1992,
  author     = {Xu, Jinchao},
  title      = {Iterative methods by space decomposition and subspace correction},
  journal    = {SIAM Rev.},
  fjournal   = {SIAM Review. A Publication of the Society for Industrial and Applied Mathematics},
  volume     = {34},
  number     = {4},
  pages      = {581--613},
  year       = {1992},
  doi        = {10.1137/1034116},
  issn       = {1095-7200},
  mrclass    = {65F10 (65M55 65M60 65N30 65N55)},
  mrnumber   = {1193013},
  mrreviewer = {Jan\ Z\'{\i}tko},
}

@book{2005Domain,
  author     = {Toselli, Andrea and Widlund, Olof},
  title      = {Domain decomposition methods---algorithms and theory},
  series     = {Springer Series in Computational Mathematics},
  volume     = {34},
  pages      = {xvi+450},
  year       = {2005},
  publisher  = {Springer-Verlag, Berlin},
  doi        = {10.1007/b137868},
  isbn       = {3-540-20696-5},
  mrclass    = {65-02 (65N55 74S05 76M10)},
  mrnumber   = {2104179},
  mrreviewer = {R\'{e}mi\ Vaillancourt},
}

\end{document}